\documentclass[11pt]{article}
\usepackage{amssymb,amsfonts}
\usepackage{amsmath,amsthm,amsxtra}

\usepackage{graphicx}
\usepackage{amscd}
\usepackage{ulem}
\usepackage{makeidx}
\usepackage{fancybox}

\usepackage[linktocpage=true, colorlinks=true, linkcolor=blue, citecolor=red, urlcolor=green]{hyperref}

\newcommand{\ccc}{{\mathbf C}}

\newcommand{\nnn}{{\mathbf N}}

\newcommand{\rrr}{{\mathbf R}}
\newcommand{\zzz}{{\mathbf Z}}

\newcommand{\hhh}{{\frak{h}}}

\newcommand{\FFF}{{\frak{F}}}

\newtheorem{prop}{Proposition}[section]
\newtheorem{lemma}{Lemma}[section]
\newtheorem{cor}{Corollary}[section]

\newtheorem{rem}{Remark}[section]
\newtheorem{note}{Note}[section]

\numberwithin{equation}{section}

\begin{document}

\title{On the characters of a certain series of \\
N=4 superconformal modules II}

\author{\footnote{12-4 Karato-Rokkoudai, Kita-ku, Kobe 651-1334, 
Japan, \qquad
wakimoto.minoru.314@m.kyushu-u.ac.jp, \hspace{5mm}
wakimoto@r6.dion.ne.jp 
}{ Minoru Wakimoto}}

\date{\empty}

\maketitle

\begin{center}
Abstract
\end{center}

In this paper we compute the characters of certain non-irreducible N=4 
superconformal modules which are different from the ones treated in 
our previous paper \cite{W2023a}, and study their relation with 
characters of N=2 superconformal modules.
Also, for these non-irreducible N=4 modules, we deduce the 
expression of characters in terms of string functions.

\tableofcontents

\section{Introduction}

At the beginning of this paper we recall the characters of N=2 
superconformal modules which are known in several literatures. 
For the convenience to compare them with N=4 characters, 
we reconstruct them in section \ref{sec:n2n4:N=2:characters}
in terms of our terminologies.

In section \ref{sec:n2n4:A(11)-characters}, we compute the characters 
of admissible $\widehat{A}(1,1)$-modules which are different from 
the ones considered in \cite{W2023a}. Note also that the 
transformation $w_0$ to define the twist 
of $\widehat{A}(1,1)$-characters is different from the one 
employed in \cite{W2023a}.

In section \ref{sec:n2n4:N=4:characters}, we compute the 
N=4 characters obtained from the quantum Hamiltonian reduction
of $\widehat{A}(1,1)$-modules.

As it is known in \cite{K1}, \cite{KP}, \cite{KW1988a}, \cite{KW1988b} 
and \cite{W2001b}, the normalized characters have better modular 
properties than the usual characters. So, in this paper, 
we deal with the normalized (super-)characters, and call them 
simply \lq \lq (super-)characters". 
Then the relation between characters and super-characters 
\begin{equation}
\text{characters} \,\
\overset{\substack{z_i \, \rightarrow \, z_i+\frac12 \\[1mm]
}}{\longleftrightarrow} \,\
\text{super-characters} \hspace{10mm}
\text{where $z_i$'s are odd variables}
\label{n2n4:eqn:2023-1230a}
\end{equation}
holds up to scalar multiples due to Lemma \ref{n2n4:lemma:2023-1231b}.

The $\lambda$-brackets $[a_{\lambda}b]$ of the N=2 superconformal algebra are shown 
as follows:
$$
\begin{array}{|c||c|c|c|c|}
\noalign{\hrule height0.8pt}
\hfil
a \backslash b
& L 
& J
& G^+
& G^-
\\[2mm]
\noalign{\hrule height0.8pt}
\hfil
L & 
(\partial+2\lambda)L+\frac{\lambda^3}{12}c &
(\partial+\lambda)J & 
(\partial+\frac{3}{2}\lambda)G^+ & 
(\partial+\frac{3}{2}\lambda)G^-
\\[2mm]
\hline
J &
\lambda J & 
\frac{\lambda}{3}c &
G^+ &
-G^-
\\[2mm]
\hline
G^+ &
\left(\frac{\partial}{2}+\frac{3}{2}\lambda \right)G^+ & 
-G^+ &
0 &
L+
\frac{1}{2}(\partial+2\lambda)J+\frac{\lambda^2}{6}c
\\[2mm]
\hline
G^- &
\left(\frac{\partial}{2}+\frac{3}{2}\lambda \right)G^- & 
G^- &
L-\frac{1}{2}(\partial+2\lambda)J+\frac{\lambda^2}{6}c &
0 
\\[2mm]
\noalign{\hrule height0.8pt}
\end{array}
$$

\noindent
and the $\lambda$-brackets of the N=4 superconformal algebra are given 
in section 8.4 of \cite{KW2004}.
Comparing these $\lambda$-brackets we see that the N=2 superconformal 
algebra is a subalgebra of the N=4 superconformal algebra with 
the same central charge. However, the formulas in Corollary 
\ref{n2n4:cor:2024-110b} may not be viewed as 
the branching of N=4 SCA with respect to its N=2 superconformal 
subalgebra because of the mismatch between their central charges
$\overset{N=2}{c}{}^{(M,m-1)}$ and $\overset{N=4}{c}{}^{(M,m)}$. 

In the case where the central charge is equal to $-6(\frac{1}{M}+1)$
\, $(M \in \nnn_{\geq 2})$, the characters of non-irreducible 
N=4 modules are modular forms and can be written by  
Mumford's theta functions and Dedekind's eta function,
and their $\ccc$-linear span is $SL_2(\zzz)$-invariant.
Explicit formulas for their modular transformation are given in 
section \ref{subsec:n2n4:m=1}.

In the simplest case where the central charge is equal to $-9$,
the characters of non-irreducible N=4 modules and their 
string functions are beautiful as they are shown in Proposition 
\ref{n2n4:prop:2024-121d} and Corollary \ref{n2n4:cor:2024-201a}.

\medskip

The formulas in Corollary \ref{n2n4:cor:2024-110b} seem to suggest 
that there will exist N=4 module structure on the spaces 
$$
\text{N=2 module} \, \otimes \, 
\FFF_{fermion}^{(\frac12)}
\, \otimes \, \FFF_{boson} \quad \text{for non-twisted N=4}
$$
and
$$
\text{Ramond N=2 module} \, \otimes \, 
\FFF_{fermion}^{(0)}
\, \otimes \, \FFF_{boson} \quad \text{for twisted N=4},
$$
where $\FFF_{fermion}^{(\varepsilon)}$ is the Fock space of free 
fermions generated by 
$$
\big\{ \psi_i, \,\ \psi^{\ast}_j \quad ; \quad 
i \, \in \, \varepsilon+\zzz_{\geq 0}, \,\ 
j \, \in \, -\varepsilon+\zzz_{>0}\big\}
$$
and
$\FFF_{boson}$ is the Fock space of free bosons generated by 
$$
\big\{ \varphi_i, \,\ \varphi^{\ast}_j \quad ; \quad 
i \, \in \, \zzz_{\geq 0}, \,\ 
j \, \in \, \zzz_{>0}\big\}
$$

In particular, Proposition \ref{n2n4:prop:2024-121d} seems to imply 
that there will exist the action of (resp. \\
Ramond twisted) 
N=4 superconformal algebra of the central chage $=-9$ on the space \\
$\FFF_{fermion}^{(\frac12)} \, \otimes \, \FFF_{boson}$
(resp. $\FFF_{fermion}^{(0)} \, \otimes \, \FFF_{boson}$). 

\medskip

In this paper, we follow notations and definitions from \cite{KRW}, 
\cite{W2001b}, \cite{W2022a} and \cite{W2023a}.

\section{Preliminaries}
\label{sec:n2n4:preliminaries}

Using the functions $\Phi^{[m,s]}_i$ defined by the formulas
(2.1a) and (2.1b) in \cite{W2022a},
we define the functions 
$\Psi^{[M,m,s; \, \varepsilon]}_{i; j,k; \, \varepsilon'}$ \, $(i=1,2)$ 
and $\Psi^{[M,m,s; \, \varepsilon]}_{j,k; \, \varepsilon'}$ by
\begin{subequations}
{\allowdisplaybreaks
\begin{eqnarray}
& & \hspace{-15mm}
\Psi^{[M,m,s; \, \varepsilon]}_{i; j,k; \, \varepsilon'}
(\tau, z_1,z_2,t) 
:=
q^{\frac{m}{M}jk} e^{\frac{2\pi im}{M}(kz_1+jz_2)} 
\Phi^{[m,s]}_i \Big(
M\tau, \, z_1+j\tau+\varepsilon, \, z_2+k\tau-\varepsilon, \, 
\frac{t}{M}\Big)
\label{n2n4:eqn:2024-119a1}
\\[2mm]
& &\hspace{-15mm}
\Psi^{[M,m,s; \, \varepsilon]}_{j,k; \, \varepsilon'}
(\tau, z_1,z_2,t) 
:= 
\Psi^{[M,m,s; \, \varepsilon]}_{1; j,k; \, \varepsilon'}
(\tau, z_1,z_2,t) 
- 
\Psi^{[M,m,s; \, \varepsilon]}_{2; j,k; \, \varepsilon'}
(\tau, z_1,z_2,t) 
\label{n2n4:eqn:2024-119a2}
\end{eqnarray}}
\end{subequations}
where $M \in \nnn$, $m \in \frac12 \nnn$ such that $(M,2m)=1$, 
and $s \in \frac12 \zzz$, 
$\varepsilon, \, \varepsilon' \in \{0, \frac12\}$ 
and $j, k \in \varepsilon'+\zzz$, 

First we note that, by computing the power series expansion of the 
functions $\Phi^{[m,s]}_i$ $(i=1,2)$ in the domain ${\rm Im}(\tau)>0$, 
we obtain the following:

\medskip

\begin{lemma} 
\label{n2n4:lemma:2024-119a}
Let $M \in \nnn$, $m \in \frac12 \nnn$, $s \in \frac12 \zzz$ and 
$\varepsilon \in \rrr$ such that $(M,2m)=1$. Then the following 
formulas hold for $j, \, k \in \rrr$ such that $0<j,k<M$. 
\begin{enumerate}
\item[{\rm 1)}] \quad $\Phi^{[m,s]}_1
(M\tau, \,\ z+j\tau+\varepsilon, \,\ -z+k\tau-\varepsilon, \,\ 0)$
$$
= \,\ \bigg[
\sum_{\substack{\ell, \, n \, \in \, \zzz \\[1mm]
\ell, \, n \, \geq \, 0}}
- 
\sum_{\substack{\ell, \, n \, \in \, \zzz \\[1mm]
\ell, \, n \, < \, 0}} \bigg] \,\ 
e^{2\pi i(n+s)(z+\varepsilon)} \, 
q^{Mm\ell^2+m\ell(j+k)} \, q^{(n+s)(M\ell+j)} 
$$
\item[{\rm 2)}] \quad $\Phi^{[m,s]}_2
(M\tau, \,\ z+j\tau+\varepsilon, \,\ -z+k\tau-\varepsilon, \,\ 0)$
$$
= \,\ \bigg[
\sum_{\substack{\ell, \, n \, \in \, \zzz \\[1mm]
\ell \, > \, 0, \,\ n \, \geq \, 0}}
- 
\sum_{\substack{\ell, \, n \, \in \, \zzz \\[1mm]
\ell \, \leq \, 0, \,\ n \, < \, 0}} \bigg] \,\ 
e^{2\pi i(n+s)(z+\varepsilon)} \, 
q^{Mm\ell^2-m\ell(j+k)} \, q^{(n+s)(M\ell-k)} 
$$
\end{enumerate}
\end{lemma}

\medskip

In the case $j=0$ or $k=0$, we have the following:

\medskip

\begin{lemma} 
\label{n2n4:lemma:2024-119c}
Let $M \in \nnn$, $m \in \frac12 \nnn$, $s \in \frac12 \zzz$ and 
$\varepsilon \in \rrr$ such that $(M,2m)=1$. Then ,
\begin{enumerate}
\item[{\rm 1)}] \,\ for $k>0$, the following formulas hold:
\begin{enumerate}
\item[{\rm (i)}]  \,\ $\Phi^{[m,s]}_1(M\tau, \, z+\varepsilon, \, 
-z+k\tau-\varepsilon, \, 0) $
$$
= \,\ \bigg[
\sum_{\substack{\ell, \, n \, \in \, \zzz \\[1mm]
\ell \, > \, 0, \,\ n \, \geq \, 0
}}
- 
\sum_{\substack{\ell, \, n \, \in \, \zzz \\[1mm]
\ell \, \leq \, 0, \,\ n \, < \, 0
}} \bigg] \,\ 
e^{2\pi i(n+s)(z+\varepsilon)} \, 
q^{Mm\ell^2+m\ell k} \, q^{(n+s)M\ell}
$$
\item[{\rm (ii)}]  \,\ $\Phi^{[m,s]}_2(M\tau, \, z+\varepsilon, \, 
-z+k\tau-\varepsilon, \, 0) $
$$
= \,\ \bigg[
\sum_{\substack{\ell, \, n \, \in \, \zzz \\[1mm]
\ell \, > \, 0, \,\ n \, \geq \, 0}} 
- 
\sum_{\substack{\ell, \, n \, \in \, \zzz \\[1mm]
\ell \, \leq \, 0, \,\ n \, < \, 0}} 
\bigg] \,\ 
e^{2\pi i(n+s)(z+\varepsilon)} \, 
q^{Mm\ell^2-m\ell k} \, q^{(n+s)(M\ell-k)} 
$$
\end{enumerate}
\item[{\rm 2)}] \,\ for $j>0$, the following formulas hold:
\begin{enumerate}
\item[{\rm (i)}] \,\ $\Phi^{[m,s]}_1(M\tau, \, -z+j\tau+\varepsilon, \, 
z-\varepsilon, \, 0) $
$$
= \,\ \bigg[
\sum_{\substack{\ell, \, n \, \in \, \zzz \\[1mm]
\ell, \, n \, \geq \, 0}}
- 
\sum_{\substack{\ell, \, n \, \in \, \zzz \\[1mm]
\ell, \, n \, < \, 0}} \bigg] \,\ 
e^{-2\pi i(n+s)(z-\varepsilon)} \, 
q^{Mm\ell^2+m\ell j} \, q^{(n+s)(M\ell+j)} 
$$
\item[{\rm (ii)}] \,\ $\Phi^{[m,s]}_2(M\tau, \, -z+j\tau+\varepsilon, \, 
z-\varepsilon, \, 0) $
$$
= \,\ \bigg[
\sum_{\substack{\ell, \, n \, \in \, \zzz \\[1mm]
\ell, \, n \, \geq \, 0}}
- 
\sum_{\substack{\ell, \, n \, \in \, \zzz \\[1mm]
\ell, \, n \, < \, 0}} \bigg] \,\ 
e^{-2\pi i(n+s)(z-\varepsilon)} \, 
q^{Mm\ell^2-m\ell j} \, q^{(n+s)M\ell} 
$$
\end{enumerate}
\end{enumerate}
\end{lemma}

\medskip

Then by Lemmas \ref{n2n4:lemma:2024-119a} and \ref{n2n4:lemma:2024-119c}
and the definition \eqref{n2n4:eqn:2024-119a1} of 
$\Psi^{[M,m,s; \, \varepsilon]}_{i; j,k; \, \varepsilon'}$, 
we obtain the following:

\medskip

\begin{lemma} 
\label{n2n4:lemma:2024-119b}
Let $M \in \nnn$, $m \in \frac12 \nnn$, $s \in \frac12 \zzz$ and 
$\varepsilon, \, \varepsilon' \in \{0, \frac12\}$ such that 
$(M,2m)=1$. Then the following formulas hold for 
$j, \, k \in \varepsilon'+\zzz$ such that $0<j,k<M$. 
\begin{enumerate}
\item[{\rm 1)}] \quad $
\Psi^{[M,m,s; \, \varepsilon]}_{1; j,k; \, \varepsilon'}(\tau, z,-z,0)$
$$
= \,\ e^{\frac{2\pi im}{M}(k-j)z} \, 
\bigg[
\sum_{\substack{\ell, \, n \, \in \, \zzz \\[1mm]
\ell, \, n \, \geq \, 0}}
- 
\sum_{\substack{\ell, \, n \, \in \, \zzz \\[1mm]
\ell, \, n \, < \, 0}} \bigg] \,\ 
e^{2\pi i(n+s)(z+\varepsilon)} \, 
q^{Mm(\ell+\frac{j}{M})(\ell+\frac{k}{M})} \, q^{(n+s)(M\ell+j)}
$$
\item[{\rm 2)}] \quad $
\Psi^{[M,m,s; \, \varepsilon]}_{2; j,k; \, \varepsilon'}(\tau, z,-z,0)$
$$
= \,\ e^{\frac{2\pi im}{M}(k-j)z} \, 
\bigg[
\sum_{\substack{\ell, \, n \, \in \, \zzz \\[1mm]
\ell \, > \, 0, \,\ n \, \geq \, 0}}
- 
\sum_{\substack{\ell, \, n \, \in \, \zzz \\[1mm]
\ell \, \leq \, 0, \,\ n \, < \, 0}} \bigg] \,\ 
e^{2\pi i(n+s)(z+\varepsilon)} \, 
q^{Mm(\ell-\frac{j}{M})(\ell-\frac{k}{M})} \, q^{(n+s)(M\ell-k)} 
$$
\end{enumerate}
\end{lemma}

\medskip

\begin{lemma} 
\label{n2n4:lemma:2024-119d}
Let $M \in \nnn$, $m \in \frac12 \nnn$, $s \in \frac12 \zzz$ and 
$\varepsilon \in \{0, \frac12\}$ such that $(M,2m)=1$. Then  

\begin{enumerate}
\item[{\rm 1)}] \,\ for $k \in \zzz$ such that $0<k<M$, the following formulas hold: 
\begin{enumerate}
\item[{\rm (i)}] \,\ $
\Psi^{[M,m,s; \, \varepsilon]}_{1; \, 0,k; \, 0}(\tau, z,-z,0)$
{\allowdisplaybreaks
\begin{eqnarray*}
&=&
e^{\frac{2\pi im}{M}kz} \, 
\bigg[
\sum_{\substack{\ell, \, n \, \in \, \zzz \\[1mm]
\ell \, > \, 0, \,\ n \, \geq \, 0
}}
- 
\sum_{\substack{\ell, \, n \, \in \, \zzz \\[1mm]
\ell \, \leq \, 0, \,\ n \, < \, 0
}} \bigg] \,\ 
e^{2\pi i(n+s)(z+\varepsilon)} \, 
q^{Mm\ell(\ell+\frac{k}{M})} \, q^{(n+s)M\ell} \hspace{10mm}
\\[2mm]
&=& 
e^{\frac{2\pi im}{M}kz} \, \Bigg\{
- \, \frac{e^{2\pi i(s-1)(z+\varepsilon)}}{1-e^{-2\pi i(z+\varepsilon)}}
\\[2mm]
& & \hspace{10mm}
+ \,\ \bigg[
\sum_{\substack{\ell, \, n \, \in \, \zzz \\[1mm]
\ell \, > \, 0, \,\ n \, \geq \, 0
}}
- 
\sum_{\substack{\ell, \, n \, \in \, \zzz \\[1mm]
\ell, \, n \, < \, 0}} \bigg] \, 
e^{2\pi i(n+s)(z+\varepsilon)} \, 
q^{Mm\ell(\ell+\frac{k}{M})} \, q^{(n+s)M\ell} \Bigg\}
\end{eqnarray*}}
\item[{\rm (ii)}]  \,\ $
\Psi^{[M,m,s; \, \varepsilon]}_{2; \, 0,k; \, 0}(\tau, z,-z,0)$
$$
= \,\ e^{\frac{2\pi im}{M}kz} \, 
\bigg[
\sum_{\substack{\ell, \, n \, \in \, \zzz \\[1mm]
\ell \, > \, 0, \,\ n \, \geq \, 0}} 
- 
\sum_{\substack{\ell, \, n \, \in \, \zzz \\[1mm]
\ell \, \leq \, 0, \,\ n \, < \, 0}} 
\bigg] \,\ 
e^{2\pi i(n+s)(z+\varepsilon)} \, 
q^{Mm\ell(\ell-\frac{k}{M})} \, q^{(n+s)(M\ell-k)} 
$$
\end{enumerate}
\item[{\rm 2)}] \,\ for $j \in \zzz$ such that $0<j<M$, the following formulas hold:
\begin{enumerate}
\item[{\rm (i)}] \,\ $
\Psi^{[M,m,s; \, \varepsilon]}_{1; \, j,0 \, ; \, 0}(\tau, -z,z,0)$
$$
= \,\ e^{\frac{2\pi im}{M}jz} \, 
\bigg[
\sum_{\substack{\ell, \, n \, \in \, \zzz \\[1mm]
\ell, \, n \, \geq \, 0}}
- 
\sum_{\substack{\ell, \, n \, \in \, \zzz \\[1mm]
\ell, \, n \, < \, 0}} \bigg] \,\ 
e^{-2\pi i(n+s)(z-\varepsilon)} \, 
q^{Mm\ell(\ell+\frac{j}{M})} \, q^{(n+s)(M\ell+j)} 
$$
\item[{\rm (ii)}] \,\ $
\Psi^{[M,m,s; \, \varepsilon]}_{2; \, j,0 \, ; \, 0}(\tau, -z,z,0)$
{\allowdisplaybreaks
\begin{eqnarray*}
&=& 
e^{\frac{2\pi im}{M}jz} \, \bigg[
\sum_{\substack{\ell, \, n \, \in \, \zzz \\[1mm]
\ell, \, n \, \geq \, 0}}
- 
\sum_{\substack{\ell, \, n \, \in \, \zzz \\[1mm]
\ell, \, n \, < \, 0}} \bigg] \,\ 
e^{-2\pi i(n+s)(z-\varepsilon)} \, 
q^{Mm\ell(\ell-\frac{j}{M})} \, q^{(n+s)M\ell} \hspace{10mm}
\\[2mm]
&=&
e^{\frac{2\pi im}{M}jz} \, \Bigg\{
\frac{e^{-2\pi is(z-\varepsilon)}}{1-e^{-2\pi i(z-\varepsilon)}}
\\[2mm]
& & \hspace{10mm}
+ \,\ \bigg[
\sum_{\substack{\ell, \, n \, \in \, \zzz \\[1mm]
\ell \, > \, 0, \,\ n \, \geq \, 0
}}
- 
\sum_{\substack{\ell, \, n \, \in \, \zzz \\[1mm]
\ell, \, n \, < \, 0}} \bigg] \,\ 
e^{-2\pi i(n+s)(z-\varepsilon)} \, 
q^{Mm\ell(\ell-\frac{j}{M})} \, q^{(n+s)M\ell} \Bigg\}
\end{eqnarray*}}
\end{enumerate}
\end{enumerate}
\end{lemma}

\medskip

The following Lemma \ref{n2n4:lemma:2023-1231b} can be checked easily
by the definition of 
$\Psi^{[M,m,s ; \, \varepsilon]}_{j,k; \, \varepsilon'}$ : 

\medskip

\begin{lemma}  
\label{n2n4:lemma:2023-1231b}
Let $M$ and $m$ be coprime positive integers and $s \in \frac12 \zzz$
and $\varepsilon, \, \varepsilon' \in \{0, \frac12\}$. 
Then the following formulas hold for $j, k \in \varepsilon'+\zzz$.
\begin{enumerate}
\item[{\rm 1)}] \,\ $\Psi^{[M,m,s ; \, 0]}_{j,k; \, \varepsilon'}
(\tau,z_1+\frac12,z_2-\frac12,t)
\, = \,\ 
e^{\frac{\pi im}{M}(k-j)} \, 
\Psi^{[M,m,s ; \, \frac12]}_{j,k; \, \varepsilon'}(\tau,z_1,z_2,t)$
\item[{\rm 2)}] \,\ $\Psi^{[M,m,s ; \, \frac12]}_{j,k; \, \varepsilon'}
(\tau, \, z_1-\tfrac12, \, z_2+\tfrac12, \, t)
\, = \,\
e^{\frac{\pi im}{M}(j-k)} \, 
\Psi^{[M,m,s ; \, 0]}_{j,k; \, \varepsilon'}(\tau,z_1,z_2,t)$
\end{enumerate}
\end{lemma}

\medskip

In the case $m=1$ and $s \in \zzz$, the functions 
$\Psi^{[M,1,s,; \varepsilon]}_{j,k; \, \varepsilon'}$ have good modular 
properties due to the $\widehat{sl}(2|1)$-denominator identity 
as follows:

\medskip

\begin{lemma} 
\label{n2n4:lemma:2023-1231a}
Let $M\in \nnn$, $s \in \zzz$ and 
$\varepsilon, \varepsilon' \in \{0, \frac12\}$. Then 
the modular transformations of 
$\Psi^{[M,1,s,; \varepsilon]}_{j,k; \, \varepsilon'}$, 
for $j,k \in \varepsilon'+\zzz$, are given by the following 
formulas :
\begin{enumerate}
\item[{\rm 1)}] \,\ $\Psi^{[M,1,s; \, \varepsilon]}_{j,k; \, \varepsilon'}
\Big(-\dfrac{1}{\tau}, \dfrac{z_1}{\tau}, \dfrac{z_2}{\tau}, t\Big)
\, = \, 
\dfrac{\tau}{M} \, e^{\frac{2\pi i}{M\tau}z_1z_2} \hspace{-5mm}
\sum\limits_{\substack{\\[0.5mm]
(a,b) \, \in \, (\varepsilon+\zzz/M\zzz)^2}} \hspace{-5mm}
e^{-\frac{2\pi i}{M}(ak+bj)} \, 
\Psi^{[M,1,s; \, \varepsilon']}_{a,b; \, \varepsilon}
(\tau, z_1, z_2,t)$
\item[{\rm 2)}] \,\ $\Psi^{[M,1,s; \, \varepsilon]}_{j,k; \, \varepsilon'}
(\tau+1, z_1, z_2,t)
\, = \, 
e^{\frac{2\pi i}{M}jk} \, 
\Psi^{[M,1,s; \, \varepsilon+\varepsilon' \, {\rm mod} \, \zzz]}_{
j,k; \, \varepsilon'}
(\tau, z_1, z_2,t)$
\end{enumerate}
\end{lemma}

\medskip

Also the following formula holds in the case $m=1$ and $s \in \zzz$, 
which can be seen easily from the formula (2.3) in \cite{W2023a}:

\medskip

\begin{note} 
\label{n2n4:note:2023-1231b}
For $M \in \nnn$ and $s \in \zzz$, the following formula holds:
$$
\Psi^{[M,1,s; \, \varepsilon]}_{j,k: \, \varepsilon'} 
(\tau,z_1, z_2,0)
\,\ = \,\ 
\Psi^{[M,1,s; \, \varepsilon]}_{k,j: \, \varepsilon'} 
(\tau,z_2, z_1,0)
$$
\end{note}

\medskip

In the simplest case where $(M,m, s)=(2,1,0)$, these functions are as follows:

\medskip

\begin{note} \,\ 
\label{n2n4:note:2023-1231a}
$\left\{
\begin{array}{lcc}
\Psi^{[2,1,0;\frac12]}_{\frac12, \frac12; \, \frac12}(\tau, z, -z, 0)
&=&
\dfrac{\eta(\tau)^3}{\vartheta_{00}(\tau,z)}
\\[4mm]
\Psi^{[2,1,0;0]}_{\frac12, \frac12; \, \frac12}(\tau, z, -z, 0)
&=&
\dfrac{\eta(\tau)^3}{\vartheta_{01}(\tau,z)}
\end{array}\right. \quad \left\{
\begin{array}{lcc}
\Psi^{[2,1,0;\frac12]}_{1, 0; \, 0}(\tau, z, -z, 0)
&=&
\dfrac{\eta(\tau)^3}{\vartheta_{10}(\tau,z)}
\\[4mm]
\Psi^{[2,1,0;0]}_{1,0; \, 0}(\tau, z, -z, 0)
&=&
\dfrac{i \, \eta(\tau)^3}{\vartheta_{11}(\tau,z)}
\end{array}\right. $
\end{note}

\begin{proof} Letting $M=2$ and $(z_1,z_2)=(z,-z)$ in the formula 
(2.3) in \cite{W2023a} and using 
$$
\vartheta_{11}(\tau, z+\tfrac12) \, = \, - \, \vartheta_{10}(\tau, z)
\quad \text{and} \quad 
\vartheta_{11}(2\tau, \, \tau) \,\ = \,\ 
- i \, q^{-\frac14} \frac{\eta(\tau)^2}{\eta(2\tau)}
$$
one has 
{\allowdisplaybreaks
\begin{eqnarray*}
\Psi^{[2,1,0; \, \frac12]}_{j,k; \, \varepsilon'}(\tau, z,-z,0)
&=&
q^{\frac{jk}{2}-\frac14} \, e^{\pi i(k-j)z} \,\ 
\frac{\eta(\tau)^2 \, \eta(2\tau)^2
}{
\vartheta_{10}(2\tau, \, z+j\tau) \, 
\vartheta_{10}(2\tau, \, z-k\tau)}
\\[2mm]
\Psi^{[2,1,0; \, 0]}_{j,k; \, \varepsilon'}(\tau, z,-z,0)
&=&
q^{\frac{jk}{2}-\frac14} \, e^{\pi i(k-j)z} \,\ 
\frac{\eta(\tau)^2 \, \eta(2\tau)^2
}{
\vartheta_{11}(2\tau, \, z+j\tau) \, 
\vartheta_{11}(2\tau, \, z-k\tau)}
\end{eqnarray*}}
where $\varepsilon' \in \{0, \frac12\}$ and $j,k \in \varepsilon'+\zzz$
such that $j+k=1$. 
Then the formulas in Note \ref{n2n4:note:2023-1231a} are obtained since
\begin{subequations}
{\allowdisplaybreaks
\begin{eqnarray}
& &
\left\{
\begin{array}{lcr}
\vartheta_{10}\Big(2\tau, \, z+\dfrac{\tau}{2}\Big) \, 
\vartheta_{10}\Big(2\tau, \, z-\dfrac{\tau}{2}\Big)
&=&
q^{-\frac18} \, \dfrac{\eta(2\tau)^2}{\eta(\tau)} \, 
\vartheta_{00}(\tau,z)
\\[4mm]
\vartheta_{11}\Big(2\tau, \, z+\dfrac{\tau}{2}\Big) \, 
\vartheta_{11}\Big(2\tau, \, z-\dfrac{\tau}{2}\Big)
&=&
q^{-\frac18} \, \dfrac{\eta(2\tau)^2}{\eta(\tau)} \, 
\vartheta_{01}(\tau,z)
\end{array} \right.
\label{n2n4:eqn:2024-119b1}
\\[2mm]
& & \left\{
\begin{array}{lcr}
\vartheta_{10}(2\tau, z+\tau) \, \vartheta_{10}(2\tau, z) 
&=& 
q^{-\frac14} e^{-\pi iz} \, 
\dfrac{\eta(2\tau)^2}{\eta(\tau)} \, \vartheta_{10}(\tau, z)
\\[4mm]
\vartheta_{11}(2\tau, z+\tau) \, \vartheta_{11}(2\tau, z) 
&=& 
- i \, q^{-\frac14} e^{-\pi iz} \, 
\dfrac{\eta(2\tau)^2}{\eta(\tau)} \, \vartheta_{11}(\tau, z)
\end{array} \right.
\label{n2n4:eqn:2024-119b2}
\end{eqnarray}}
\end{subequations}

\vspace{-8mm}

\end{proof}

\medskip

The following formula is obtained from the Kac-Peterson's identity 
(2.1) in \cite{W2022e} by computing the power series expansion of 
$\Phi^{(-)[\frac12, \frac12]}_1(\tau, z,-z,0)$ 
in the domain ${\rm Im}(\tau)>0$ and ${\rm Im}(z)<0$.

\medskip

\begin{note} \quad 
\label{note:2024-204c}
$\dfrac{\eta(\tau)^3}{\vartheta_{11}(\tau,z)} 
\,\ = \,\ 
- \, i \, \bigg[
\sum\limits_{\substack{j, \, k \, \in \, \zzz \\[1mm]
j, \, k \, \geq 0}}
-
\sum\limits_{\substack{j, \, k \, \in \, \zzz \\[1mm]
j, \, k \, < 0}} \bigg] \,\ 
(-1)^j \, e^{-\pi i(2k+1)z} \, q^{\frac12 j(j+1)+jk} $
\end{note}

\medskip

The following formulas will be used in section 
\ref{subsec:n2n4:string-fn} to compute the power series 
expansion of characters.

\medskip

\begin{note} \,\ 
\label{note:2024-204b}
\begin{enumerate}
\item[{\rm 1)}] $\vartheta_{00}(\tau,z)^2 = \, 
\underbrace{
\eta(2\tau) \bigg[\dfrac{\eta(2\tau)^2}{\eta(\tau)\eta(4\tau)}\bigg]^2
}_{\substack{|| \\[-1mm] {\displaystyle \hspace{-5mm}
1+\cdots
}}}
\sum\limits_{n \in \zzz} e^{4\pi inz}q^{n^2}
+ \, 
2 
\underbrace{
\eta(2\tau) \bigg[\dfrac{\eta(4\tau)}{\eta(2\tau)}\bigg]^2
}_{\substack{|| \\[-2mm] {\displaystyle \hspace{-5mm}
q^{\frac14}+\cdots
}}}
\sum\limits_{n \in \zzz} e^{2\pi i(2n+1)z}q^{(n+\frac12)^2}$

\item[{\rm 2)}] $\vartheta_{10}(\tau,z)^2 = \, 
2 \underbrace{
\eta(2\tau) \bigg[\frac{\eta(4\tau)}{\eta(2\tau)}\bigg]^2
}_{\substack{|| \\[-2mm] {\displaystyle \hspace{-5mm}
q^{\frac14}+\cdots
}}} 
\sum\limits_{n \in \zzz} e^{4\pi inz}q^{n^2}
\, + \, 
\underbrace{
\eta(2\tau) \bigg[\frac{\eta(2\tau)^2}{\eta(\tau)\eta(4\tau)}\bigg]^2
}_{\substack{|| \\[-1mm] {\displaystyle \hspace{-5mm}
1+\cdots
}}} 
\sum\limits_{n \in \zzz} e^{2\pi i(2n+1)z}q^{(n+\frac12)^2}$
\end{enumerate}
\end{note}

\begin{proof} The formula in 1) is obtained immediately from Lemma 2.2 
in \cite{W2022d} and the power series expression of 
$\vartheta_{00}(2\tau,2z)$ and $\vartheta_{10}(2\tau,2z)$.
The formula in 2) follows from 1) by replacing $z$ with 
$z+\frac{\tau}{2}$.
\end{proof}

\section{Characters of N=2 superconformal modules}
\label{sec:n2n4:N=2:characters}

\subsection{Integrable modules of $\widehat{sl}(2|1)$}
\label{subsec:n2n4:sl(21):integrable}

We consider the Dynkin diagram 
\setlength{\unitlength}{1mm}
\begin{picture}(25,12)
\put(6,1){\circle{3}}
\put(17,1){\circle{3}}
\put(7.5,1){\line(1,0){8}}
\put(7,2){\line(1,1){3.7}}
\put(16,2){\line(-1,1){3.7}}
\put(6,1){\makebox(0,0){$\times$}}
\put(17,1){\makebox(0,0){$\times$}}
\put(2,3){\makebox(0,0){$\alpha_1$}}
\put(21,3){\makebox(0,0){$\alpha_2$}}
\put(11.5,7){\circle{3}}
\put(15.5,8.5){\makebox(0,0){$\alpha_0$}}
\put(11.5,-1){\makebox(0,0){$1$}}
\end{picture}
of the affine Lie superalgebra $\widehat{sl}(2|1)$ with the 
inner product $( \,\ | \,\ )$ such that $
\Big((\alpha_i|\alpha_j)\Big)_{i,j=0,1,2} = \left(
\begin{array}{rrr}
2 & -1 & -1 \\[0mm]
-1 & 0 & 1 \\[0mm]
-1 & 1 & 0
\end{array}\right) $. Then the dual Coxeter number of 
$\widehat{sl}(2|1)$ is $h^{\vee}=1$. 
Let $\hhh$ (resp. $\overline{\hhh}$) be the Cartan subalgebra 
of $\widehat{sl}(2|1)$ (resp. $sl(2|1)$) and $\Lambda_0$ 
be the element in $\hhh^{\ast}$ satisfying the conditions 
$(\Lambda_0|\alpha_j)=\delta_{j,0}$ and $(\Lambda_0|\Lambda_0)=0$.
Let $\delta=\sum_{i=0}^2\alpha_i$ be the primitive imaginary root 
and $\rho=\Lambda_0$ be the Weyl vector. 
Define the coordinates on $\hhh$ by 
\begin{subequations}
\begin{equation}
h \,\ = \,\ 2\pi i(-\tau \Lambda_0-z_2\alpha_1-z_1\alpha_2+t\delta)
\,\ =: \,\ 
(\tau,z_1,z_2,t)
\label{n2n4:eqn:2024-101a1}
\end{equation}
then
\begin{equation}
\left\{
\begin{array}{lcl}
e^{-\alpha_1(h)} &=& e^{2\pi iz_1} \\[1mm]
e^{-\alpha_2(h)} &=& e^{2\pi iz_2}
\end{array}\right. 
\label{n2n4:eqn:2024-101a2}
\end{equation}
\end{subequations}
For $\alpha \in \hhh$, let $t_{\alpha}$ be the linear transformation 
of $\hhh$ defined in \cite{K1}:
\begin{equation}
t_{\alpha}(\lambda) \,\ := \,\
\lambda +(\lambda|\alpha) - \Big\{
\frac{(\alpha|\alpha)}{2}(\lambda|\delta)+(\lambda|\alpha)
\Big\} \, \delta
\label{n2n4:eqn:2024-101b}
\end{equation}
Then it is easy to see the following:
\begin{equation} \left\{
\begin{array}{lcl}
t_{j\theta}(\Lambda_0) 
&=& 
\Lambda_0 \, + \, j\theta \, - \, j^2 \delta
\\[2mm]
t_{j\theta}(\alpha_i) 
&=& \alpha_i \, - \, j \delta \hspace{20mm} (i=1,2)
\end{array} \right.
\label{n2n4:eqn:2024-101c}
\end{equation}
where $\theta :=\alpha_1+\alpha_2$ is the highest root of $sl(2|1)$.

\medskip

For $m \in \rrr$, let $P_m$ be the set of weights $\lambda$ 
of $\widehat{sl}(2|1)$ satisfying the conditions 
\begin{enumerate}
\item[(i)] \quad $\lambda$ is integrable with respect to 
$\alpha_0$ and $\theta$ 
\item[(ii)] \quad $(\lambda|\delta) \, = \, m$ 
\item[(iii)] \quad $\lambda$ is atypical with respect to $\alpha_1$,
namely $(\lambda|\alpha_1) \, = \, 0$ 
\end{enumerate}
Then, by the integrability conditions for Lie superalgebras explained in 
\cite{KW2001} or \cite{W2004}, we have the following:

\medskip

\begin{lemma}
\label{n2n4:lemma:2024-101a}
\begin{enumerate}
\item[{\rm 1)}] \,\ $P_m \ne \emptyset \,\ \Longleftrightarrow \,\ 
m \in \zzz_{\geq 0}$ 
\item[{\rm 2)}] \,\ $m \in \zzz_{\geq 0} \,\ \Longrightarrow \,\ 
P_m = \Big\{
\lambda^{[m,m_2]} \, := \, m \Lambda_0 + m_2 \alpha_1
\,\ ; \,\ m_2 \in \zzz_{\geq 0} \,\ \text{such that} \,\ 
m_2 \, \leq \, m
\Big\}$
\end{enumerate}
\end{lemma}

\medskip

Noticing, for $\lambda =\lambda^{[m,m_2]} \in P_m$, that 
\begin{equation}
\lambda+\rho  = (m+1) \Lambda_0+m_2\alpha_1
\hspace{10mm} \text{and} \hspace{10mm}
|\lambda+\rho|^2 \,\ = \,\ 0 \, ,
\label{n2n4:eqn:2024-101d}
\end{equation}
we put 
\begin{equation}
F_{\lambda^{[m,m_2]}+\rho}^{(\pm)} \,\ := \,\ 
\sum_{j \in \zzz}t_{j\theta}
\Big(\frac{e^{\lambda^{[m,m_2]}+\rho}}{1 \pm e^{-\alpha_1}}\Big)
\label{n2n4:eqn:2024-101e}
\end{equation}
which is computed easily by \eqref{n2n4:eqn:2024-101c} as follows: 
\begin{subequations}
\begin{equation}
F_{\lambda^{[m,m_2]}+\rho}^{(\pm)}
\, = \, 
e^{(m+1)\Lambda_0}
\sum_{j \in \zzz}\frac{
e^{-j(m+1)(\alpha_1+\alpha_2) + m_2 \alpha_1} q^{j^2(m+1)-jm_2}
}{1\pm e^{-\alpha_1}q^j}
\label{n2n4:eqn:2024-101f1}
\end{equation}
Applying the reflection $r_{\theta}$, one has 
\begin{equation}
r_{\theta}F_{\lambda^{[m,m_2]}+\rho}^{(\pm)}
\, = \, 
e^{(m+1)\Lambda_0}
\sum_{j \in \zzz}\frac{
e^{j(m+1)(\alpha_1+\alpha_2) - m_2 \alpha_2} q^{j^2(m+1)-jm_2}
}{1\pm e^{\alpha_2}q^j}
\label{n2n4:eqn:2024-101f2}
\end{equation}
\end{subequations}
These formulas are written in terms of the coordinates defined by 
\eqref{n2n4:eqn:2024-101a1} and mock theta functions defined by 
(2.1a) and (2.1b) in \cite{W2022a} as follows:

\medskip

\begin{lemma} \,\
\label{n2n4:lemma:2024-101b}
For $\lambda^{[m,m_2]} \in P_m$, the following formulas hold:
\begin{enumerate}
\item[{\rm 1)}]
\begin{enumerate}
\item[{\rm (i)}] \,\ $F^{(+)}_{\lambda^{[m,m_2]}+\rho}
(\tau, z_1,z_2,t)
\, = \,\ 
(-1)^{m_2} \, \Phi^{[m+1,-m_2]}_1(\tau, \, z_1+\frac12, \, 
z_2-\frac12, \, -t)$
\item[{\rm (ii)}] \,\ $\big[r_{\theta}F^{(+)}_{\lambda^{[m,m_2]}+\rho}\big]
(\tau, z_1,z_2,t)
\, = \,\ 
(-1)^{m_2} \, \Phi^{[m+1,-m_2]}_2(\tau, \, z_1+\frac12, \, 
z_2-\frac12, \, -t)$
\end{enumerate}
\item[{\rm 2)}]
\begin{enumerate}
\item[{\rm (i)}] \,\ $F^{(-)}_{\lambda^{[m,m_2]}+\rho}
(\tau, z_1,z_2,t)
\, = \,\ 
\Phi^{[m+1,-m_2]}_1(\tau, \, z_1, \, z_2, \, -t)$
\item[{\rm (ii)}] \,\ $\big[r_{\theta}F^{(-)}_{\lambda^{[m,m_2]}+\rho}\big]
(\tau, z_1,z_2,t)
\, = \,\ 
\Phi^{[m+1,-m_2]}_2(\tau, \, z_1, \, z_2, \, -t)$
\end{enumerate}
\end{enumerate}
\end{lemma}

\medskip

Then the numerators of the character and the super-character of the 
irreducible $\widehat{sl}(2|1)$-module $L(\lambda^{[m,m_2]})$ are 
given by 
$$
\widehat{R}^{(\pm)} \cdot {\rm ch}_{L(\lambda^{[m,m_2]})}^{(\pm)}
\,\ = 
\sum_{w \, \in \, \langle r_{\theta}\rangle}
\varepsilon(w) \, w\big(F^{(\pm)}_{\lambda^{[m,m_2]}+\rho}\big)
$$
and are obtained as follows:

\medskip

\begin{prop} \quad 
\label{n2n4:prop:2024-101a}
For $\lambda^{[m,m_2]} \in P_m$ the following formulas hold:
\begin{enumerate}
\item[{\rm 1)}] \quad $\big[\widehat{R}^{(+)} \cdot 
{\rm ch}_{L(\lambda^{[m,m_2]})}^{(+)}\big](\tau, z_1,z_2,t)
\,\ = \,\ 
(-1)^{m_2} \, \Phi^{[m+1,-m_2]}(\tau, \, z_1+\frac12, \, 
z_2-\frac12, \, -t)$
\item[{\rm 2)}] \quad $\big[\widehat{R}^{(-)} \cdot 
{\rm ch}_{L(\lambda^{[m,m_2]})}^{(-)}\big](\tau, z_1,z_2,t)
\,\ = \,\ 
\Phi^{[m+1,-m_2]}(\tau, \, z_1, \, z_2, \, -t)$
\end{enumerate}
where $\widehat{R}^{(+)}$ (resp. $\widehat{R}^{(-)}$) is the 
denominator (resp. super-denominator) of $\widehat{sl}(2|1)$.
\end{prop}

\subsection{Characters of principal admissible 
$\widehat{sl}(2|1)$-modules}
\label{subsec:n2n4:sl(21):principal:admissible:set}


To describe principal admissible weights, we consider the principal 
admissible subsets for $\widehat{sl}(2|1)$ defined by
\begin{subequations}
\begin{equation}
\Pi^{(M)}_{k_1,k_2} \,\ := \,\ \big\{
k_0\delta+\alpha_0, \,\ k_1\delta+\alpha_1, \,\ k_2\delta+\alpha_2\big\}
\label{n2n4:eqn:2024-102a1}
\end{equation}
where
\begin{equation}
k_i \, \in \, \zzz_{\geq 0} \quad (i=0,1,2)
\hspace{10mm} \text{and} \quad M=\sum_{i=0}^2k_i+1
\label{n2n4:eqn:2024-102a2}
\end{equation}
\end{subequations}
Following section 3 in \cite{W2001b}, we take $(\overline{y}, \beta)$ 
such that  $\Pi^{(M)}_{k_1,k_2}=t_{\beta}\overline{y}(\Pi^{(M)}_{0,0})$.

\medskip

\begin{lemma} 
\label{n2n4:lemma:2024-102a}
For a principal admissible subset $\Pi^{(M)}_{k_1,k_2}
=t_{\beta}\overline{y}(\Pi^{(M)}_{0,0})$ and 
$z=-z_1\alpha_2-z_2\alpha_1 \in \overline{\hhh}$, the 
following hold:
\begin{enumerate}
\item[{\rm 1)}] \,\ $\overline{y} \, = \, {\rm identity}, 
\hspace{5mm} \beta \, = \, - k_1 \alpha_2 - k_2 \alpha_1$, 
\hspace{5mm} $|\beta|^2 = \, 2 k_1 k_2$
\quad and \quad $(\overline{y}^{-1}\beta|\alpha_1)=-k_1$.
\item[{\rm 2)}] \,\ $(\beta|z) \, = \,  k_1 z_2 +k_2 z_1$
\item[{\rm 3)}] \,\ $\overline{y}^{-1}(z+\tau\beta) 
\, = \, - (z_1+k_1\tau) \alpha_2 - (z_2+k_2\tau)\alpha_1$
\end{enumerate}
\end{lemma}

\medskip

Let $m$ be a non-negative integer and $M$ be a positive integer 
such that $(M, m+1)=1$. 
For an integrable weight $\lambda^{[m,m_2]} \in P_m$ and 
a principal admissible subset $\Pi^{(M)}_{k_1,k_2}
=t_{\beta}\overline{y}(\Pi^{(M)}_{0,0})$, we define the principal
admissible weight $\lambda^{(M)[m,m_2]}_{k_1,k_2}$ of level 
$K=\frac{m+1}{M}-1$ by the following: 
\begin{subequations}
{\allowdisplaybreaks
\begin{eqnarray}
\lambda^{(M)[m,m_2]}_{k_1,k_2} 
&:=& 
(t_{\beta}\overline{y}). \big(\lambda^{[m,m_2]} \, - \, 
(M-1)(K+1)\Lambda_0\big) 
\label{n2n4:eqn:2024-102b1}
\\[2mm]
&=&
K\Lambda_0 \, + \, (K+1)\beta \, - \, \frac{|\beta|^2}{2}(K+1)\delta
\, + \, m_2 \, \big[\overline{y}\alpha_1 \, - \, (\beta|\overline{y}\alpha_1)\delta\big]
\label{n2n4:eqn:2024-102b2}
\end{eqnarray}}
Using Lemma \ref{n2n4:lemma:2024-102a}, this formula is rewritten as follows:
\begin{equation}
\lambda^{(M)[m,m_2]}_{k_1,k_2} 
=
K\Lambda_0 - (K+1)k_1\alpha_2 
- (K+1)\Big(k_2-\frac{m_2}{K+1}\Big)\alpha_1
- (K+1)k_1\Big(k_2-\frac{m_2}{K+1}\Big)\delta
\label{n2n4:eqn:2024-102b3}
\end{equation}
\end{subequations}
We note that
\begin{equation}
\big|\lambda^{(M)[m,m_2]}_{k_1,k_2} +\rho\big|^2 \,\ = \,\ 0
\label{n2n4:eqn:2024-102c}
\end{equation}
by \eqref{n2n4:eqn:2024-102b1} and \eqref{n2n4:eqn:2024-101d}. 

The character of a principal admissible module $L(\lambda)$, for 
$\lambda=(t_{\beta}\overline{y}).(\lambda^0-(M-1)(K+h^{\vee})\Lambda_0)$, 
is given by Theorem 3.2 in \cite{KW1989} or Theorem 3.3.4 in \cite{W2001b} :
\begin{equation}
\big[\widehat{R}^{(\pm)} \cdot {\rm ch}^{(\pm)}_{L(\lambda)}\big]
(\tau,z,t)=
\big[\widehat{R}^{(\pm)} \cdot {\rm ch}^{(\pm)}_{L(\lambda^0)}\big]
\bigg(M\tau, \, \overline{y}^{-1}(z+\tau\beta), \, 
\frac{1}{M}\Big(t+(z|\beta)+\frac{\tau |\beta|^2}{2}\Big)\bigg)
\label{n2n4:eqn:2024-102d}
\end{equation}
Using this formula and Proposition \ref{n2n4:prop:2024-101a} and 
Lemma \ref{n2n4:lemma:2024-102a}, we obtain the following:

\medskip

\begin{prop} 
\label{n2n4:prop:2024-102a}
The numerator of the character and the super-character
of the principal admissible $\widehat{sl}(2|1)$-module
$L(\lambda^{(M)[m,m_2]}_{k_1,k_2})$ are given as follows:
\begin{enumerate}
\item[{\rm 1)}] \quad $\big[\widehat{R}^{(+)} \cdot 
{\rm ch}_{L(\lambda^{(M)[m,m_2]}_{k_1,k_2})}^{(+)}\big]
(\tau, z_1, z_2,t)$

$ = \, (-1)^{m_2} 
q^{\frac{m+1}{M}k_1k_2} \, e^{\frac{2\pi i(m+1)}{M}(k_2z_1+k_1z_2)} \, 
\Phi^{[m+1,-m_2]}(M\tau, \, z_1+k_1\tau +\frac12, \, 
z_2+k_2\tau-\frac12, \, - \frac{t}{M})$
\item[{\rm 2)}] \quad $\big[\widehat{R}^{(-)} \cdot 
{\rm ch}_{L(\lambda^{(M)[m,m_2]}_{k_1,k_2})}^{(-)}\big](h)$

$= \,\ 
q^{\frac{m+1}{M}k_1k_2} \, e^{\frac{2\pi i(m+1)}{M}(k_2z_1+k_1z_2)} \, 
\Phi^{[m+1,-m_2]}
(M\tau, \, z_1+k_1\tau, \, z_2+k_2\tau, \, - \frac{t}{M})$
\end{enumerate}
\end{prop}

\subsection{Twisted $\widehat{sl}(2|1)$-characters}
\label{subsec:n2n4:sl(21):twist:alpha1-alpha2}

In this section, we consider the characters of $\widehat{sl}(2|1)$-modules 
twisted by $t_{-\xi}$ where $\xi := \frac12(\alpha_1-\alpha_2)$.
To compute the twisted characters, we note the following:

\medskip

\begin{note} 
\label{n2n4:note:2024-103a}
For $\xi = \frac12 (\alpha_1-\alpha_2)$ and 
$h = 2\pi i(- \tau \Lambda_0 
\underbrace{- z_1 \alpha_2 - z_2 \alpha_1}_{\displaystyle z} + t \delta)
\in \hhh$, the following formulas hold:
\begin{enumerate}
\item[{\rm 1)}] \quad $\alpha_i(\xi) \,\ = \,\ \left\{
\begin{array}{rcl}
-\frac12 & \,\ & (i \, = \, 1) \\[1mm]
\frac12 & & (i \, = \, 2)
\end{array} \right. $
\item[{\rm 2)}] \quad $(\xi|z) \,\ = \,\ - \, \frac12(z_1-z_2)$
\item[{\rm 3)}]
\begin{enumerate}
\item[{\rm (i)}] \quad $t_{-\xi}(z) \,\ = \,\ 
z \,\ - \,\ \frac12(z_1-z_2) \, \delta$
\item[{\rm (ii)}] \quad $t_{-\xi}(\Lambda_0) \,\ = \,\ 
\Lambda_0 \,\ - \,\ \xi \,\ + \,\ \frac14 \, \delta$
\end{enumerate}
\item[{\rm 4)}] \quad ${\displaystyle 
t_{-\xi}(h) \,\ = \,\ 
\Big(\tau, \,\ z_1+\frac{\tau}{2}, \,\ z_2-\frac{\tau}{2}, \,\ 
t + \frac{z_2-z_1}{2} - \frac{\tau}{4} \Big)
}$
\end{enumerate}
\end{note}

\medskip

Then, the twisted (super-)characters defined by 
\begin{subequations}
\begin{equation}
\big[\widehat{R}^{(\pm) {\rm tw}} \cdot 
{\rm ch}_{\lambda}^{(\pm) {\rm tw}}\big](h)
\, := \, 
\big[\widehat{R}^{(\pm)} \cdot {\rm ch}_{\lambda}^{(\pm)}\big]
(t_{-\xi}(h))
\label{n2n4:eqn:2024-103a1}
\end{equation}
namely 
\begin{equation}
\big[\widehat{R}^{(\pm) {\rm tw}} \cdot 
{\rm ch}_{\Lambda}^{(\pm) {\rm tw}}\big](\tau, z_1,z_2,t)
:= 
\big[\widehat{R}^{(\pm)} \cdot {\rm ch}_{\lambda}^{(\pm)}\big]
\Big(\tau, \, z_1+\frac{\tau}{2}, \, z_2-\frac{\tau}{2}, \, 
t + \frac{z_2-z_1}{2} - \frac{\tau}{4} \Big)
\label{n2n4:eqn:2024-103a2}
\end{equation}
\end{subequations}
are obtained by using the formulas for non-twisted (super-)characters
in Proposition \ref{n2n4:prop:2024-102a} as follows

\medskip

\begin{prop} 
\label{n2n4:prop:2024-103a}
The numerator of the character and the super-character
of the principal admissible $\widehat{sl}(2|1)$-module
$L(\lambda^{(M)[m,m_2]}_{k_1,k_2})$ twisted by $t_{-\xi}$ \, 
($\xi=\frac12(\alpha_1-\alpha_2)$) are given as follows:
\begin{enumerate}
\item[{\rm (i)}] \quad $\big[\widehat{R}^{(+){\rm tw}} \cdot 
{\rm ch}_{L\big(\lambda^{(M)[m,m_2]}_{k_1,k_2}\big)}^{(+){\rm tw}}\big]
(\tau, z_1, z_2,t)$
{\allowdisplaybreaks
\begin{eqnarray*}
& & \hspace{-13mm}
= \,\ (-1)^{m_2} q^{\frac{m+1}{M}(k_1+\frac12)(k_2-\frac12)+\frac{m+1}{4M}} \, 
e^{\frac{2\pi i(m+1)}{M}[(k_2-\frac12)z_1+(k_1+\frac12)z_2]} 
\\[2mm]
& & \hspace{-8mm}
\times \,\  
\Phi^{[m+1;-m_2]}\bigg(M\tau, \,\ 
z_1+\Big(k_1+\frac12\Big)\tau +\frac12, \,\ 
z_2+\Big(k_2-\frac12\Big)\tau-\frac12, \,\ 
- \frac{1}{M}\Big(t - \frac{\tau}{4} \Big)\bigg)
\end{eqnarray*}}
\item[{\rm (ii)}] \quad $\big[\widehat{R}^{(-){\rm tw}} \cdot 
{\rm ch}_{L\big(\Lambda^{(M)[m,m_2]}_{k_1,k_2}\big)}^{(-){\rm tw}}\big]
(\tau, z_1, z_2,t)$
{\allowdisplaybreaks
\begin{eqnarray*}
& & \hspace{-13mm}
= \,\ q^{\frac{m+1}{M}(k_1+\frac12)(k_2-\frac12)+\frac{m+1}{4M}} \, 
e^{\frac{2\pi i(m+1)}{M}[(k_2-\frac12)z_1+(k_1+\frac12)z_2]} 
\\[2mm]
& & \hspace{-8mm}
\times \,\  
\Phi^{[m+1;-m_2]}\bigg(M\tau, \,\ 
z_1+\Big(k_1+\frac12\Big)\tau, \,\ 
z_2+\Big(k_2-\frac12\Big)\tau, \,\ 
- \frac{1}{M}\Big(t - \frac{\tau}{4} \Big)\bigg)
\end{eqnarray*}}
\end{enumerate}
\end{prop}

\subsection{Quantum Hamiltonian reduction of $\widehat{sl}(2|1)$-modules}
\label{subsec:n2n4:sl(21):quantum}

We consider the quantum Hamiltonian reduction associated 
to the pair $(x=\frac12 \theta, f=e_{\theta})$ for $\widehat{sl}(2|1)$.
Taking a basis $J^{(N=2)}_0:=\alpha_1-\alpha_2$ of $\overline{\hhh}^f$, 
the character and the super-character of the quantum reduction 
$\overset{N=2}{H}(\lambda)$ of an $\widehat{sl}(2|1)$-module 
$L(\lambda)$ are given by the formula 
\begin{equation}
\big[\overset{N=2}{R}{}^{(\pm)} \cdot 
{\rm ch}_{\overset{N=2}{H}(\lambda)}^{(\pm)}\big](\tau, z)
\,\ = \,\ 
\big[\widehat{R}^{(\pm)} \cdot {\rm ch}_{L(\lambda)}^{(\pm)}\big]
\Big(\underbrace{
2\pi i\big(-\tau (\Lambda_0+x)+zJ^{(N=2)}_0+
\frac{\tau}{2}(x|x)\delta\big)}_{\substack{|| \\[1mm] {\displaystyle h
}}}\Big)
\label{n2n4:eqn:2024-104a}
\end{equation}
and similarly for twisted characters, where 
{\allowdisplaybreaks
\begin{eqnarray}
h &=& 2\pi i \Big(-\tau \Lambda_0 \,\ - \,\ \frac{\tau}{2}(\alpha_1 +\alpha_2)
\,\ + \,\ z (\alpha_1 -\alpha_2) \,\ + \,\ \frac{\tau}{4} \, \delta\Big)
\nonumber
\\[3mm]
&=& 
2\pi i \Big\{-\tau \Lambda_0 
\,\ - \,\ \Big(z+\frac{\tau}{2}\Big)\alpha_2 
\,\ - \,\ \Big(-z+\frac{\tau}{2}\Big)\alpha_1 
\,\ + \,\ \frac{\tau}{4}\delta\Big\}
\nonumber
\\[1mm]
&=& 
\Big(\tau, \,\ z+\frac{\tau}{2}, \,\ -z+\frac{\tau}{2}, \,\ \frac{\tau}{4} \Big)
\label{n2n4:eqn:2024-104b}
\end{eqnarray}}
Then the formula \eqref{n2n4:eqn:2024-104a} is written in terms of 
coordinates as follows:
\begin{subequations}
{\allowdisplaybreaks
\begin{eqnarray}
\big(\overset{N=2}{R}{}^{(\pm)} \cdot 
{\rm ch}_{\overset{N=2}{H}(\lambda)}^{(\pm)}\big)(\tau, z)
&=&
\big(\widehat{R}^{(\pm)} \cdot {\rm ch}_{L(\lambda)}^{(\pm)}\big)
\Big(\tau, \,\ z+\frac{\tau}{2}, \,\ -z+\frac{\tau}{2}, \,\ \frac{\tau}{4} \Big)
\label{n2n4:eqn:2024-104c1}
\\[1mm]
\big(\overset{N=2}{R}{}^{(\pm){\rm tw}} \cdot 
{\rm ch}_{\overset{N=2}{H}(\lambda)}^{(\pm){\rm tw}}\big)(\tau, z)
&=&
\big(\widehat{R}^{(\pm){\rm tw}} \cdot 
{\rm ch}_{L(\lambda)}^{(\pm){\rm tw}}\big)
\Big(\tau, \,\ z+\frac{\tau}{2}, \,\ -z+\frac{\tau}{2}, \,\ \frac{\tau}{4} \Big)
\label{n2n4:eqn:2024-104c2}
\end{eqnarray}}
\end{subequations}

Applying these formulas \eqref{n2n4:eqn:2024-104c1} and 
\eqref{n2n4:eqn:2024-104c2} to a principal admissible weight 
$\lambda=\lambda^{(M)[m,m_2]}_{k_1,k_2}$ and using 
Propositions \ref{n2n4:prop:2024-102a} and \ref{n2n4:prop:2024-103a},
we obtain the following:

\medskip

\begin{lemma} \,\
\label{n2n4:lemma:2024-104a}
\begin{enumerate}
\item[{\rm 1)}] 
\begin{enumerate}
\item[{\rm (i)}] \quad $\big[\overset{N=2}{R}{}^{(+)} \cdot 
{\rm ch}^{(+)}_{\overset{N=2}{H}(\lambda^{(M)[m,m_2]}_{k_1,k_2})}\big](\tau,z)
\,\ = \,\ 
(-1)^{m_2} \, \Psi^{[M,m+1,-m_2;\frac12]}_{k_1+\frac12, k_2+\frac12; \frac12}
(\tau, \, z, \, -z, \, 0)$
\item[{\rm (ii)}] \quad $\big[\overset{N=2}{R}{}^{(-)} \cdot 
{\rm ch}^{(-)}_{\overset{N=2}{H}(\lambda^{(M)[m,m_2]}_{k_1,k_2})}\big](\tau,z)
\,\ = \,\ 
\Psi^{[M,m+1,-m_2;0]}_{k_1+\frac12, k_2+\frac12; \frac12}
(\tau, \, z, \, -z, \, 0)$
\end{enumerate}
\item[{\rm 2)}] 
\begin{enumerate}
\item[{\rm (i)}] \quad $\big[\overset{N=2}{R}{}^{(+)} \cdot 
{\rm ch}^{(+)}_{\overset{N=2}{H}(\lambda^{(M)[m,m_2]}_{k_1,k_2})}\big](\tau,z)
\,\ = \,\ 
(-1)^{m_2} \, \Psi^{[M,m+1,-m_2;\frac12]}_{k_1+1, k_2; 0}
(\tau, \, z, \, -z, \, 0)$
\item[{\rm (ii)}] \quad $\big[\overset{N=2}{R}{}^{(-)} \cdot 
{\rm ch}^{(-)}_{\overset{N=2}{H}(\lambda^{(M)[m,m_2]}_{k_1,k_2})}\big](\tau,z)
\,\ = \,\ 
\Psi^{[M,m+1,-m_2;0]}_{k_1+1, k_2; 0}(\tau, \, z, \, -z, \, 0)$
\end{enumerate}
\end{enumerate}
\end{lemma}

\medskip

The denominators of N=2 superconformal algebra are given 
up to the normalization factors as follows:
\begin{subequations}
\begin{equation}
\overset{N=2}{R}{}^{(\varepsilon)}_{\varepsilon'} (\tau, z)
\,\ = \,\  \frac{\eta(\tau)^3}{
\vartheta_{1-2\varepsilon', \, 1-2\varepsilon}(\tau,z)}
\label{n2n4:eqn:2024-104d1}
\end{equation}
where
\begin{equation}
\varepsilon \, = \,\ \left\{
\begin{array}{ccr}
\frac12 & & \text{denominator} \\[0mm]
0 & & \text{super-denominator}
\end{array}\right. \hspace{10mm} 
\varepsilon' \, = \,\ \left\{
\begin{array}{ccr}
\frac12 & & \text{non-twisted} \\[0mm]
0 & & \text{Ramond twisted}
\end{array}\right. 
\label{n2n4:eqn:2024-104d2}
\end{equation}
\end{subequations}
The formula \eqref{n2n4:eqn:2024-104d1} is written explicitly as follows:
\begin{equation} \left\{
\begin{array}{rcr}
\overset{N=2}R{}^{(\frac12)}_{\frac12}(\tau,z) &=&
\dfrac{\eta(\tau)^3}{\vartheta_{00}(\tau,z)}
\\[4mm]
\overset{N=2}R{}^{(0)}_{\frac12}(\tau, z) &=&
\dfrac{\eta(\tau)^3}{\vartheta_{01}(\tau,z)}
\end{array}\right.
\quad \text{and} \quad \left\{
\begin{array}{rcr}
\overset{N=2}R{}^{(\frac12)}_{0}(\tau,z) &=&
\dfrac{\eta(\tau)^3}{\vartheta_{10}(\tau,z)}
\\[4mm]
\overset{N=2}R{}^{(0)}_{0}(\tau, z) &=&
\dfrac{\eta(\tau)^3}{\vartheta_{11}(\tau,z)}
\end{array}\right.
\label{n2n4:eqn:2024-104e}
\end{equation}
And the modular transformation of N=2 denominators is as follows:

\medskip

\begin{note} \,\ 
\label{n2n4:note:2024-104a}
\begin{enumerate}
\item[{\rm 1)}] \quad ${\displaystyle 
\overset{N=2}{R}{}^{(\varepsilon)}_{\varepsilon'}
\Big(-\frac{1}{\tau}, \, \frac{z}{\tau}\Big) \,\ = \,\ 
- \, (-1)^{(1-2\varepsilon)(1-2\varepsilon')} \, \tau \, 
e^{\frac{2\pi iz^2}{\tau}} \, 
\overset{N=2}{R}{}^{(\varepsilon')}_{\varepsilon}(\tau,z)
}$
\item[{\rm 2)}] \quad ${\displaystyle 
\overset{N=2}{R}{}^{(\varepsilon)}_{\varepsilon'}(\tau+1, \, z) 
\,\ = \,\ 
e^{\pi i\varepsilon'} \, 
\overset{N=2}{R}{}^{(\varepsilon+\varepsilon')}_{\varepsilon'}(\tau,z)
}$
\end{enumerate}
\end{note}

\medskip

Then, by Lemma \ref{n2n4:lemma:2024-104a} and \eqref{n2n4:eqn:2024-104d1},
the characters of $\overset{N=2}{H}(\lambda^{(M)[m,m_2]}_{k_1,k_2})$ 
are obtained as follows:

\medskip

\begin{prop} \,\ 
\label{n2n4:prop:2024-104a}
\begin{enumerate}
\item[{\rm 1)}] 
\begin{enumerate}
\item[{\rm (i)}] \,\ ${\rm ch}^{(+)}_{
\overset{N=2}{H}(\lambda^{(M)[m,m_2]}_{k_1,k_2})}(\tau,z)
= \,\ 
(-1)^{m_2} \, \Psi^{[M,m+1,-m_2;\frac12]}_{k_1+\frac12, k_2+\frac12; \frac12}
(\tau, \, z, \, -z, \, 0) \,\ 
\dfrac{\vartheta_{00}(\tau,z)}{\eta(\tau)^3}$
\item[{\rm (ii)}] \,\ ${\rm ch}^{(-)}_{
\overset{N=2}{H}(\lambda^{(M)[m,m_2]}_{k_1,k_2})}(\tau,z)
= \,\ 
\Psi^{[M,m+1,-m_2;0]}_{k_1+\frac12, k_2+\frac12; \frac12}
(\tau, \, z, \, -z, \, 0)\,\ 
\dfrac{\vartheta_{01}(\tau,z)}{\eta(\tau)^3}$
\end{enumerate}
\item[{\rm 2)}] 
\begin{enumerate}
\item[{\rm (i)}] \,\ ${\rm ch}^{(+){\rm tw}}_{
\overset{N=2}{H}(\lambda^{(M)[m,m_2]}_{k_1,k_2})}(\tau,z)
= \,\ 
(-1)^{m_2} \, \Psi^{[M,m+1,-m_2;\frac12]}_{k_1+1, k_2; 0}
(\tau, \, z, \, -z, \, 0) \,\ 
\dfrac{\vartheta_{10}(\tau,z)}{\eta(\tau)^3}$
\item[{\rm (ii)}] \,\ ${\rm ch}^{(-){\rm tw}}_{
\overset{N=2}{H}(\lambda^{(M)[m,m_2]}_{k_1,k_2})}(\tau,z)
= \,\ 
\Psi^{[M,m+1,-m_2;0]}_{k_1+1, k_2; 0}(\tau, \, z, \, -z, \, 0) \,\ 
\dfrac{\vartheta_{11}(\tau,z)}{\eta(\tau)^3}$
\end{enumerate}
\end{enumerate}
\end{prop}

\medskip

In the case $m=0$, the modular transformations of these characters 
are given in \cite{RY} and \cite{KW1994} and their fusion algebras 
via Verlinde's formula are computed in \cite{W1998}.

In the case $(M,m)=(2,0)$, the above Proposition \ref{n2n4:prop:2024-104a}
together with Note \ref{n2n4:note:2023-1231a} gives 
\begin{equation}
\begin{array}{l}
{\rm ch}^{(\pm)}_{
\overset{N=2}{H}(\lambda^{(2)[0,0]}_{0,0})}(\tau,z)
\,\ = \,\ 
{\rm ch}^{(+){\rm tw}}_{
\overset{N=2}{H}(\lambda^{(2)[0,0]}_{0,0})}(\tau,z)
\,\ = \,\ 1
\\[4.5mm]
{\rm ch}^{(-){\rm tw}}_{
\overset{N=2}{H}(\lambda^{(2)[0,0]}_{0,0})}(\tau,z)
\,\ = \,\ i
\end{array}
\label{n2n4:eqn:2024-104g}
\end{equation}

We note that the central charge of the N=2 module obtained from 
an $\widehat{sl}(2|1)$-module of level $K$ is 
\begin{subequations}
\begin{equation}
\overset{N=2}{c}(K) \,\ = \,\ -3(2K+1)
\label{n2n4:eqn:2024-104f1}
\end{equation}
so
\begin{equation}
\overset{N=2}{c}{}^{(M,m)} 
\, := \,\ 
\text{the central charge of} \,\ 
\overset{N=2}{H}(\lambda^{(M)[m,m_2]}_{k_1,k_2})
\, = \,  -3\Big(\dfrac{2(m+1)}{M}-1\Big)
\label{n2n4:eqn:2024-104f2}
\end{equation}
\end{subequations}

\section{Characters of $\widehat{A}(1,1)$-modules}
\label{sec:n2n4:A(11)-characters}

\subsection{Characters of integrable $\widehat{A}(1,1)$-modules}
\label{subsec:n2n4:A(11):char:integrable}

We consider the Dynkin diagram  of the affine Lie superalgebra
$\widehat{A}(1,1) = \widehat{(sl(2|2)/\ccc I)}$ \\
\setlength{\unitlength}{1mm}
\begin{picture}(32,15)
\put(5,4){\circle{3}}
\put(14,4){\circle{3}}
\put(23,4){\circle{3}}
\put(5,4){\makebox(0,0){$\times$}}
\put(23,4){\makebox(0,0){$\times$}}
\put(5,0){\makebox(0,0){$\alpha_1$}}
\put(14,0){\makebox(0,0){$\alpha_2$}}
\put(23,0){\makebox(0,0){$\alpha_3$}}
\put(6.5,4){\line(1,0){6}}
\put(15.5,4){\line(1,0){6}}
\put(14,10){\circle{3}}
\put(18,11){\makebox(0,0){$\alpha_0$}}
\put(6,5){\line(3,2){6.5}}
\put(22,5){\line(-3,2){6.5}}
\put(9.5,1.5){\makebox(0,0){$1$}}
\put(18.5,1.5){\makebox(0,0){$1$}}
\put(6,9){\makebox(0,0){$-1$}}
\put(21,8.5){\makebox(0,0){$-1$}}
\end{picture}
with the inner product $( \,\ | \,\ )$ such that \\
$\Big((\alpha_i|\alpha_j)\Big)_{i,j=0,1,2,3} = 
\left(
\begin{array}{rrrr}
2 & -1 & 0 & -1 \\[0mm]
-1 & 0 & 1 & 0 \\[0mm]
0 & 1 & -2 & 1 \\[0mm]
-1 & 0 & 1 & 0
\end{array} \right)$ . Then the dual Coxeter number 
of $\widehat{A}(1,1)$ is $h^{\vee}=0$. Let $\hhh$ 
(resp. $\overline{\hhh}$) be the Cartan subalgebra of $\widehat{A}(1,1)$ 
(resp. $A(1,1)$) and $\Lambda_0$ be the element in $\hhh^{\ast}$ 
satisfying the conditions 
$(\Lambda_0|\alpha_j) = \delta_{j,0}$ and $(\Lambda_0|\Lambda_0)=0$.
Let $\delta=\sum_{i=0}^3\alpha_i$ be the primitive imaginary root
and $\rho= -\frac12 (\alpha_1+\alpha_3)$ be the Weyl vector.
Define the coordinates on the Cartan subalgebra $\hhh$ of 
$\widehat{A}(1,1)$ by the formula (3.7) in \cite{W2023a}.

\medskip

For $m \in \rrr$, let $P_m$ be the set of weights $\Lambda$ 
of $\widehat{sl}(2|2)$ satisfying the conditions 
\begin{enumerate}
\item[(i)] \quad $\Lambda$ is integrable with respect to 
$\alpha_0$ and $\theta := \sum_{i=1}^3\alpha_i$ \, ,
\item[(ii)] \quad $(\Lambda|\delta) \, = \, m$ \, ,
\item[(iii)] \quad $\lambda$ is atypical with respect to 
$\alpha_1$ and $\alpha_3$,
namely $(\Lambda|\alpha_i) \, = \, 0 \,\ (i=1,3)$ \, .
\end{enumerate}
Then, by the integrability conditions for Lie superalgebras explained in 
\cite{KW2001} or \cite{W2004}, we have the following:

\medskip

\begin{lemma}
\label{n2n4:lemma:2024-105a}
\begin{enumerate}
\item[{\rm 1)}] \,\ $P_m \ne \emptyset \,\ \Longleftrightarrow \,\ 
m \in \zzz_{\geq 0}$ 
\item[{\rm 2)}] \,\ If \, $m \in \zzz_{\geq 0}$, then 
$$
P_m = \bigg\{
\Lambda^{(m,m_2)} \, := \, m \Lambda_0 + \frac{m_2}{2}( \alpha_1+\alpha_3)
\,\ ; \,\ m_2 \in \zzz_{\geq 0} \,\ \text{such that} \,\ 
m_2 \, \leq \, m
\bigg\}
$$
\end{enumerate}
\end{lemma}

\medskip

Note that
\begin{subequations}
\begin{equation}
\Lambda^{(m,m_2)}+\rho \,\ = \,\ 
m \Lambda_0 + \frac{m_2-1}{2}( \alpha_1+\alpha_3)
\label{n2n4:eqn:2024-105a1}
\end{equation}
and that 
\begin{equation}
|\Lambda^{(m,m_2)}+\rho|^2 \,\ = \,\ 0
\label{n2n4:eqn:2024-105a2}
\end{equation}
\end{subequations}

For $\Lambda^{(m,m_2)} \in P_m$, we put
\begin{equation}
F^{(\pm)}_{\Lambda^{(m,m_2)}+\rho} \, := \, \sum_{j \in \zzz}
t_{j\theta}\Big(\frac{e^{\Lambda^{(m,m_2)}+\rho}}{
(1+e^{-\alpha_1})(1+e^{-\alpha_3})}\Big)
\label{n2n4:eqn:2024-105b}
\end{equation}
Then, noticing that
\begin{equation} \left\{
\begin{array}{lclc}
t_{j\theta}(\Lambda_0) &=& 
\Lambda_0+j\theta -j^2\delta &
\\[1.5mm]
t_{j\theta}(\alpha_i) &=& \alpha_i-j\delta & (i=1,3)
\end{array} \right.
\text{and} \quad 
\left\{
\begin{array}{lcl}
r_{\theta}(\alpha_1) &=& -(\alpha_2+\alpha_3) \\[1.5mm]
r_{\theta}(\alpha_3) &=& -(\alpha_1+\alpha_2)
\end{array}\right. 
\label{n2n4:eqn:2024-105c}
\end{equation}
and using functions $\Phi^{(A(1|1))[m,s]}_i$ defined by (2.4a) in \cite{W2023a},
we have the following:

\medskip

\begin{lemma} 
\label{n2n4:lemma:2024-105b}
Let $m \in \nnn$ and $\Lambda =\Lambda^{(m,m_2)} \in P_m$. 
Then $F^{(\pm)}_{\Lambda+\rho}$ and 
$r_{\theta}F^{(\pm)}_{\Lambda+\rho}$ are given as follows:
\begin{enumerate}
\item[{\rm 1)}]
\begin{enumerate}
\item[{\rm (i)}] \,\ $F^{(+)}_{\Lambda^{(m;m_2)}+\rho}(\tau, z_1,z_2,t)
\, = \, 
e^{m\Lambda_0} \, \sum\limits_{j \in \zzz} \, 
\dfrac{e^{2\pi i mj(z_1-z_2) \, - \, 2\pi i \, (m_2-1) \, z_1} \, q^{mj^2- j(m_2-1)}
}{(1+e^{2\pi i z_1} q^j)^2}$
$$
= \,\ (-1)^{m_2+1} \, \Phi^{(A(1|1))[m,-m_2+1]}_1
(\tau, \, z_1+\tfrac12, \, -z_2-\tfrac12, \, -t)
$$
\item[{\rm (ii)}] \,\ $\big(r_{\theta}F^{(+)}_{\Lambda^{(m;m_2)}+\rho}\big)
(\tau, z_1,z_2,t)
\, = \, 
e^{m\Lambda_0} \, \sum\limits_{j \in \zzz} \, 
\dfrac{
e^{-2\pi i mj(z_1-z_2)\, - \, 2\pi i \, (m_2-1) \, z_2} \, q^{mj^2- j(m_2-1)}
}{(1+e^{2\pi i z_2} q^j)^2}$
$$
= \,\ 
(-1)^{m_2+1} \, \Phi^{(A(1|1))[m,-m2+1]}_2
(\tau, \, z_1+\tfrac12, \, -z_2-\tfrac12, \, -t)
$$
\end{enumerate}
\item[{\rm 2)}]
\begin{enumerate}
\item[{\rm (i)}] \,\ $F^{(-)}_{\Lambda^{(m;m_2)}+\rho}(\tau, z_1,z_2,t)
\, = \, 
e^{m\Lambda_0} \, \sum\limits_{j \in \zzz} \, 
\dfrac{e^{2\pi i mj(z_1-z_2) \, - \, 2\pi i \, (m_2-1) \, z_1} \, q^{mj^2- j(m_2-1)}
}{(1-e^{2\pi i z_1} q^j)^2}$
$$
= \,\ 
\Phi^{(A(1|1))[m,-m_2+1]}_1(\tau, \, z_1, \, -z_2, \, -t)
$$
\item[{\rm (ii)}] \,\ $\big(r_{\theta}F^{(-)}_{\Lambda^{(m;m_2)}+\rho}\big)
(\tau, z_1,z_2,t)
\, = \, 
e^{m\Lambda_0} \, \sum\limits_{j \in \zzz} \, 
\dfrac{
e^{-2\pi i mj(z_1-z_2)\, - \, 2\pi i \, (m_2-1) \, z_2} \, q^{mj^2- j(m_2-1)}
}{(1-e^{2\pi i z_2} q^j)^2}$
$$
= \,\ 
\Phi^{(A(1|1))[m,-m2+1]}_2(\tau, \, z_1, \, -z_2, \, -t)
$$
\end{enumerate}
\end{enumerate}
\end{lemma}

\medskip

Then we have

\medskip

\begin{prop} 
\label{n2n4:prop:2024-105a}
Let $m \in \nnn$ and $\Lambda = \Lambda^{(m, m_2)} \in P_m$. Then 
the numerators of the character and the super-character of 
$L(\Lambda^{(m, m_2)})$ are given as follows:
\begin{enumerate}
\item[{\rm 1)}] \, $\big[\widehat{R}^{(+)} \cdot 
{\rm ch}_{L(\Lambda^{(m, m_2)})}^{(+)}\big]
(\tau, z_1,z_2,t) 
\, = \, 
(-1)^{m_2+1} \, \Phi^{(A(1|1))[m,-m_2+1]}
(\tau, \, z_1+\frac12, \, -z_2-\frac12, \, -t)$
\item[{\rm 2)}] \, $\big[\widehat{R}^{(-)} \cdot 
{\rm ch}_{L(\Lambda^{(m, m_2)})}^{(-)}\big]
(\tau, z_1,z_2,t) 
\, = \, 
\Phi^{(A(1|1))[m,-m_2+1]}(\tau, \, z_1, \, -z_2, \, -t)$
\end{enumerate}
\end{prop}

\subsection{Characters of principal admissible $\widehat{A}(1,1)$-modules}
\label{subsec:n2n4:sl(22):char:principal-admissible} 

Let $m$ and $M$ be coprime positive integers. For an integrable 
weight $\Lambda^{(m, m_2)} \in P_m^{\widehat{A}(1,1)}$ and the 
principal admissible simple subset $\Pi^{(M)(\heartsuit)}_{k_1,k_2}$ \,\ 
$(\heartsuit = {\rm I} \sim {\rm IV})$ of $\widehat{A}(1,1)$ defined 
in section 4.1 of \cite{W2023a}, we consider the principal 
admissible weight
\begin{subequations}
{\allowdisplaybreaks
\begin{eqnarray}
& & \hspace{-15mm}
\Lambda^{(M) (m,m_2) (\heartsuit)}_{k_1,k_2}
\,\ = \,\ 
(t_{\beta}\overline{y}).\Big(
\Lambda^{(m, m_2)}-(M-1)\frac{m}{M}\Lambda_0\Big)
\label{n2n4:eqn:2024-106a1}
\\[2mm]
& & \hspace{-10mm}
= \, 
\frac{m}{M} \Lambda_0
+ \frac{m}{M} \beta
+ \frac{m_2-1}{2} \overline{y}(\alpha_1+\alpha_3)
- \rho
+ \Big[- \frac{mk_1(k_1+k_2)}{M}+k_1(m_2-1)\Big] \delta
\label{n2n4:eqn:2024-106a2}
\end{eqnarray}}
\end{subequations}
Then, by using the character formula \eqref{n2n4:eqn:2024-102d}
for a principal admissible module, 
the numerators of the character and the super-character 
of the principal admissible module 
$L(\Lambda^{(M) (m,m_2) (\heartsuit)}_{k_1,k_2})$ 
$(\heartsuit = {\rm I} \sim {\rm IV})$ are obtained as folows:

\medskip

\begin{lemma} \,\ 
\label{n2n4:lemma:2024-106a}
\begin{enumerate}
\item[{\rm 1)}] \,\ $\big[\widehat{R}^{(\pm)} \cdot 
{\rm ch}^{(\pm)}_{L(\Lambda^{(M) (m,m_2) \, {\rm (I)}}_{k_1,k_2})}\big]
(\tau, z_1, z_2,t) \,\ =$
$$
\big(\widehat{R}^{(\pm)} \cdot {\rm ch}^{(\pm)}_{L(\Lambda^{(m, m_2)})}\big)
\Big(M\tau, \, 
z_1+k_1\tau, \, 
z_2-(k_1+k_2)\tau, \, 
\dfrac{1}{M}\big[t+(k_1+k_2)z_1-k_1z_2+k_1(k_1+k_2)\tau\big]\Big)
$$
\item[{\rm 2)}] \,\ $\big[\widehat{R}^{\pm} \cdot 
{\rm ch}_{L(\Lambda^{(M) (m,m_2) \, {\rm (II)}}_{k_1,k_2})}\big]
(\tau, z_1, z_2,t) \,\ =$
$$ \hspace{-5mm}
\big(R^{(\pm)} \cdot {\rm ch}^{(\pm)}_{L(\Lambda^{(m, m_2)})}\big)
\Big(M\tau, \, 
-z_1+k_1\tau, \, 
-z_2-(k_1+k_2)\tau, \, 
\dfrac{1}{M}\big[t-(k_1+k_2)z_1+k_1z_2+k_1(k_1+k_2)\tau\big]\Big)
$$
\item[{\rm 3)}] \,\ $\big[\widehat{R}^{\pm} \cdot 
{\rm ch}_{L(\Lambda^{(M) (m,m_2) \, {\rm (III)}}_{k_1,k_2})}\big]
(\tau, z_1, z_2,t) \,\ =$
$$ \hspace{-5mm}
\big(R^{(\pm)} \cdot {\rm ch}^{(\pm)}_{L(\Lambda^{(m, m_2)})}\big)
\Big(M\tau, \, -z_2+k_1\tau, \, 
-z_1-(k_1+k_2)\tau, \, 
\dfrac{1}{M}\big[t+k_1z_1-(k_1+k_2)z_2+k_1(k_1+k_2)\tau\big]\Big)
$$
\item[{\rm 4)}] \,\ $\big[\widehat{R}^{\pm} \cdot 
{\rm ch}_{L(\Lambda^{(M) (m,m_2) \, {\rm (IV)}}_{k_1,k_2})}\big]
(\tau, z_1, z_2,t) \,\ =$
$$
\big(R^{(\pm)} \cdot {\rm ch}^{(\pm)}_{L(\Lambda^{(m, m_2)})}\big)
\Big(M\tau, \, 
z_2+k_1\tau, \,  
z_1-(k_1+k_2)\tau, \, 
\dfrac{1}{M}\big[t-k_1z_1+(k_1+k_2)z_2+k_1(k_1+k_2)\tau\big]\Big)
$$
\end{enumerate}
\end{lemma}

\medskip

Using Proposition \ref{n2n4:prop:2024-105a}, 
these formulas are rewritten as follows:

\medskip 

\begin{prop} \,\
\label{n2n4:prop:2024-106a}
\begin{enumerate}
\item[{\rm 1)}]  
\begin{enumerate}
\item[{\rm (i)}] \,\ $\big[\widehat{R}^{(+)} \cdot 
{\rm ch}^{(+)}_{L(\Lambda^{(M) (m,m_2) {\rm (I))}}_{k_1,k_2})}\big]
(\tau, z_1, z_2,t)
\,\ = \,\ 
(-1)^{m_2+1}
e^{\frac{2\pi im}{M}[t+(k_1+k_2)z_1-k_1z_2]}$
$$
\times \,\ 
q^{\frac{m}{M}k_1(k_1+k_2)} \, 
\Phi^{(A(1|1))[m,-m_2+1]}
(M\tau, \,\ z_1+k_1\tau+\tfrac12, \,\ -z_2+(k_1+k_2)\tau-\tfrac12, \,\ 0)
$$
\item[{\rm (ii)}] \,\ $\big[\widehat{R}^{(+)} \cdot 
{\rm ch}^{(+)}_{L(\Lambda^{(M)(m, m_2){\rm (II)}}_{k_1,k_2})}\big]
(\tau, z_1, z_2,t) 
\,\ = \,\ 
(-1)^{m_2+1}
e^{\frac{2\pi im}{M}[t-(k_1+k_2)z_1+k_1z_2]} $
$$
\times \,\ 
q^{\frac{m}{M}k_1(k_1+k_2)} \, 
\Phi^{(A(1|1))[m,-m_2+1]}
(M\tau, \,\ -z_1+k_1\tau+\tfrac12, \,\ z_2+(k_1+k_2)\tau-\tfrac12, \,\ 0)
$$
\item[{\rm (iii)}] \,\ $\big[\widehat{R}^{(+)} \cdot 
{\rm ch}^{(+)}_{L(\Lambda^{(M)(m, m_2){\rm (III)}}_{k_1,k_2})}\big]
(\tau, z_1, z_2,t) 
\,\ = \,\ 
(-1)^{m_2+1}
e^{\frac{2\pi im}{M}[t+k_1z_1-(k_1+k_2)z_2]} $
$$
\times \,\ 
q^{\frac{m}{M}k_1(k_1+k_2)} \, 
\Phi^{(A(1|1))[m,-m_2+1]}
(M\tau, \,\ -z_2+k_1\tau+\tfrac12, \,\ z_1+(k_1+k_2)\tau-\tfrac12, \,\ 0)
$$
\item[{\rm (iv)}] \,\ $\big[\widehat{R}^{(+)} \cdot 
{\rm ch}^{(+)}_{L(\Lambda^{(M)(m, m_2){\rm (IV)}}_{k_1,k_2})}\big]
(\tau, z_1, z_2,t) 
\,\ = \,\ 
(-1)^{m_2+1}
e^{\frac{2\pi im}{M}[t-k_1z_1+(k_1+k_2)z_2]} $
$$
\times \,\ 
q^{\frac{m}{M}k_1(k_1+k_2)} \, 
\Phi^{(A(1|1))[m,-m_2+1]}
(M\tau, \,\ z_2+k_1\tau+\tfrac12, \,\ -z_1+(k_1+k_2)\tau-\tfrac12, \,\ 0)
$$
\end{enumerate}
\item[{\rm 2)}]  
\begin{enumerate}
\item[{\rm (i)}] \,\ $\big[\widehat{R}^{(-)} \cdot 
{\rm ch}^{(-)}_{L(\Lambda^{(M) (m,m_2) {\rm (I))}}_{k_1,k_2})}\big]
(\tau, z_1, z_2,t) \,\ = \,\ 
e^{\frac{2\pi im}{M}[t+(k_1+k_2)z_1-k_1z_2]} $
$$
\times \,\ 
q^{\frac{m}{M}k_1(k_1+k_2)} \, 
\Phi^{(A(1|1))[m,-m_2+1]}
(M\tau, \, z_1+k_1\tau, \, -z_2+(k_1+k_2)\tau, \, 0)
$$
\item[{\rm (ii)}] \,\ $\big[\widehat{R}^{(-)} \cdot 
{\rm ch}^{(-)}_{L(\Lambda^{(M)(m, m_2){\rm (II)}}_{k_1,k_2})}\big]
(\tau, z_1, z_2,t) \,\ = \,\ 
e^{\frac{2\pi im}{M}[t-(k_1+k_2)z_1+k_1z_2]} $
$$
\times \,\ 
q^{\frac{m}{M}k_1(k_1+k_2)} \, 
\Phi^{(A(1|1))[m,-m_2+1]}
(M\tau, \, -z_1+k_1\tau, \, z_2+(k_1+k_2)\tau, \, 0)
$$
\item[{\rm (iii)}] \,\ $\big[\widehat{R}^{(-)} \cdot 
{\rm ch}^{(-)}_{L(\Lambda^{(M)(m, m_2){\rm (III)}}_{k_1,k_2})}\big]
(\tau, z_1, z_2,t) \,\ = \,\ 
e^{\frac{2\pi im}{M}[t+k_1z_1-(k_1+k_2)z_2]} $
$$
\times \,\ 
q^{\frac{m}{M}k_1(k_1+k_2)} \, 
\Phi^{(A(1|1))[m,-m_2+1]}
(M\tau, \, -z_2+k_1\tau, \, z_1+(k_1+k_2)\tau, \, 0)
$$
\item[{\rm (iv)}] \,\ $\big[\widehat{R}^{(-)} \cdot 
{\rm ch}^{(-)}_{L(\Lambda^{(M)(m, m_2){\rm (IV)}}_{k_1,k_2})}\big]
(\tau, z_1, z_2,t) \,\ = \,\
e^{\frac{2\pi im}{M}[t-k_1z_1+(k_1+k_2)z_2]} $
$$
\times \,\ 
q^{\frac{m}{M}k_1(k_1+k_2)} \, 
\Phi^{(A(1|1))[m,-m_2+1]}
(M\tau, \, z_2+k_1\tau, \, -z_1+(k_1+k_2)\tau, \, 0)
$$
\end{enumerate}
\end{enumerate}
\end{prop}

\subsection{Twisted $\widehat{A}(1,1)$-characters}
\label{subsec:n2n4:A(11):twisted:characters}


In this section, we consider the $\widehat{A}(1,1)$-characters 
twisted by $w_0 \, := \, r_{\alpha_2}t_{-\frac12 \alpha_2}$,
which is different from the one employed in \cite{W2023a}.
The action of $w_0$ on $\hhh$ is given by
\begin{equation}
\left\{
\begin{array}{lcl}
w_0(\alpha_0) &=& \alpha_0 \\[1mm]
w_0(\alpha_1) &=& \alpha_1+\alpha_2+\frac12 \delta \\[1mm]
w_0(\alpha_2) &=& -\alpha_2-\delta \\[1mm]
w_0(\alpha_3) &=& \alpha_2+\alpha_3+\frac12 \delta
\end{array}\right. 
\quad \text{and} \quad 
w_0(\Lambda_0) \,\ = \,\ 
\Lambda_0 \, + \, \dfrac12 \, \alpha_2 \, + \, \dfrac14 \, \delta
\label{n2n4:eqn:2024-107b2}
\end{equation}
so
\begin{equation}
w_0(\tau, z_1, z_2,t) \,\ = \,\ \Big(\tau, \,\ 
-z_2+\frac{\tau}{2}, \,\ -z_1+\frac{\tau}{2}, \,\ 
t+\frac{z_1+z_2}{2}-\frac{\tau}{4}\Big)
\label{n2n4:eqn:2024-107b}
\end{equation}
Then the twisted characters 
\begin{equation}
{\rm ch}^{(\pm) \, {\rm tw}}_{L(\Lambda^{(M)(m, m_2)(\heartsuit)}_{k_1,k_2})}
(\tau, z_1,z_2,t) \, := \, 
{\rm ch}^{(\pm)}_{L(\Lambda^{(M)(m, m_2)(\heartsuit)}_{k_1,k_2})}
\big(w_0 (\tau, z_1,z_2,t)\big)
\label{n2n4:eqn:2024-107a}
\end{equation}
of the principal admissible $\widehat{A}(1,1)$-modules 
$L(\Lambda^{(M)(m, m_2)(\heartsuit)}_{k_1,k_2})$ are given as follows:

\medskip

\begin{prop} \,\
\label{n2n4:prop:2024-107a}
\begin{enumerate}
\item[{\rm 1)}]  
\begin{enumerate}
\item[{\rm (i)}] \,\ $\big[\widehat{R}^{(+){\rm tw}} \cdot 
{\rm ch}^{(+){\rm tw}}_{L(\Lambda^{(M)(m, m_2){\rm (I)}}_{k_1,k_2})}\big]
(\tau, z_1, z_2,t)$
{\allowdisplaybreaks
\begin{eqnarray*}
&=& (-1)^{m_2+1}
e^{\frac{2\pi im}{M}(t-\frac{\tau}{4})} \, 
e^{\frac{2\pi im}{M}[(k_1+\frac12)z_1-(k_1+k_2-\frac12)z_2]} \, 
q^{\frac{m}{M}[k_1(k_1+k_2)+\frac12k_2]} 
\\[2mm]
& & \hspace{-5mm}
\times \,\ \Phi^{(A(1|1))[m,-m_2+1]}(M\tau, \,\ 
-z_2+(k_1+\tfrac12)\tau+\tfrac12, \,\ 
z_1+(k_1+k_2-\tfrac12)\tau-\tfrac12, \,\ 0)
\end{eqnarray*}}
\item[{\rm (ii)}] \,\ $\big[\widehat{R}^{(+){\rm tw}} \cdot 
{\rm ch}^{(+){\rm tw}}_{L(\Lambda^{(M)(m, m_2){\rm (II)}}_{k_1,k_2})}
\big](\tau, z_1, z_2,t)$
{\allowdisplaybreaks
\begin{eqnarray*}
&=&(-1)^{m_2+1}
e^{\frac{2\pi im}{M}(t-\frac{\tau}{4})} \, 
e^{\frac{2\pi im}{M}[-(k_1-\frac12)z_1+(k_1+k_2+\frac12)z_2]} \, 
q^{\frac{m}{M}[k_1(k_1+k_2)-\frac12k_2]} \, 
\\[2mm]
& &
\times \,\ \Phi^{(A(1|1))[m,-m_2+1]}(M\tau, \,\ 
z_2+(k_1-\tfrac12)\tau+\tfrac12, \,\ 
-z_1+(k_1+k_2+\tfrac12)\tau-\tfrac12, \,\ 0)
\end{eqnarray*}}
\item[{\rm (iii)}] \,\ $\big[\widehat{R}^{(+){\rm tw}} \cdot 
{\rm ch}^{(+){\rm tw}}_{L(\Lambda^{(M)(m, m_2){\rm (III)}}_{k_1,k_2})}
\big](\tau, z_1, z_2,t)$
{\allowdisplaybreaks
\begin{eqnarray*}
&=&(-1)^{m_2+1}
e^{\frac{2\pi im}{M}(t-\frac{\tau}{4})} \, 
e^{\frac{2\pi im}{M}[(k_1+k_2+\frac12)z_1-(k_1-\frac12)z_2]} \, 
q^{\frac{m}{M}[k_1(k_1+k_2)-\frac12k_2]}
\\[2mm]
& &
\times \,\ \Phi^{(A(1|1))[m,-m_2+1]}(M\tau, \,\ 
z_1+(k_1-\tfrac12)\tau+\tfrac12, \,\ 
-z_2+(k_1+k_2+\tfrac12)\tau-\tfrac12, \,\ 0)
\end{eqnarray*}}
\item[{\rm (iv)}] \,\ $\big[\widehat{R}^{(+){\rm tw}} \cdot 
{\rm ch}^{(+){\rm tw}}_{L(\Lambda^{(M)(m, m_2){\rm (IV)}}_{k_1,k_2})}
\big](\tau, z_1, z_2,t)$
{\allowdisplaybreaks
\begin{eqnarray*}
&=&(-1)^{m_2+1}
e^{\frac{2\pi im}{M}(t-\frac{\tau}{4})} \, 
e^{\frac{2\pi im}{M}[-(k_1+k_2-\frac12)z_1+(k_1+\frac12)z_2]} \, 
q^{\frac{m}{M}[k_1(k_1+k_2)+\frac12k_2]}
\\[2mm]
& & \hspace{-5mm}
\times \,\ \Phi^{(A(1|1))[m,-m_2+1]}(M\tau, \,\ 
-z_1+(k_1+\tfrac12)\tau+\tfrac12, \,\ 
z_2+(k_1+k_2-\tfrac12)\tau-\tfrac12, \,\ 0)
\end{eqnarray*}}
\end{enumerate}
\item[{\rm 2)}]  
\begin{enumerate}
\item[{\rm (i)}] \,\ $\big[\widehat{R}^{(-){\rm tw}} \cdot 
{\rm ch}^{(-){\rm tw}}_{L(\Lambda^{(M)(m, m_2){\rm (I)}}_{k_1,k_2})}\big]
(\tau, z_1, z_2,t)$
{\allowdisplaybreaks
\begin{eqnarray*}
&=&
e^{\frac{2\pi im}{M}(t-\frac{\tau}{4})} \, 
e^{\frac{2\pi im}{M}[(k_1+\frac12)z_1-(k_1+k_2-\frac12)z_2]} \, 
q^{\frac{m}{M}[k_1(k_1+k_2)+\frac12k_2]} 
\\[2mm]
& & \hspace{-5mm}
\times \,\ \Phi^{(A(1|1))[m,-m_2+1]}
(M\tau, \,\ -z_2+(k_1+\tfrac12)\tau, \,\ z_1+(k_1+k_2-\tfrac12)\tau, \,\ 0)
\end{eqnarray*}}
\item[{\rm (ii)}] \,\ $\big[\widehat{R}^{(-){\rm tw}} \cdot 
{\rm ch}^{(-){\rm tw}}_{L(\Lambda^{(M)(m, m_2){\rm (II)}}_{k_1,k_2})}
\big](\tau, z_1, z_2,t)$
{\allowdisplaybreaks
\begin{eqnarray*}
&=&
e^{\frac{2\pi im}{M}(t-\frac{\tau}{4})} \, 
e^{\frac{2\pi im}{M}[-(k_1-\frac12)z_1+(k_1+k_2+\frac12)z_2]} \, 
q^{\frac{m}{M}[k_1(k_1+k_2)-\frac12k_2]} 
\\[2mm]
& &
\times \,\ \Phi^{(A(1|1))[m,-m_2+1]}
(M\tau, \,\ z_2+(k_1-\tfrac12)\tau, \,\ -z_1+(k_1+k_2+\tfrac12)\tau, \,\ 0)
\end{eqnarray*}}
\item[{\rm (iii)}] \,\ $\big[\widehat{R}^{(-){\rm tw}} \cdot 
{\rm ch}^{(-){\rm tw}}_{L(\Lambda^{(M)(m, m_2){\rm (III)}}_{k_1,k_2})}
\big](\tau, z_1, z_2,t)$
{\allowdisplaybreaks
\begin{eqnarray*}
&=&
e^{\frac{2\pi im}{M}(t-\frac{\tau}{4})} \, 
e^{\frac{2\pi im}{M}[(k_1+k_2+\frac12)z_1-(k_1-\frac12)z_2]} \, 
q^{\frac{m}{M}[k_1(k_1+k_2)-\frac12k_2]}
\\[2mm]
& &
\times \,\ \Phi^{(A(1|1))[m,-m_2+1]}
(M\tau, \,\ z_1+(k_1-\tfrac12)\tau, \,\ -z_2+(k_1+k_2+\tfrac12)\tau, \,\ 0)
\end{eqnarray*}}
\item[{\rm (iv)}] \,\ $\big[\widehat{R}^{(-){\rm tw}} \cdot 
{\rm ch}^{(-){\rm tw}}_{L(\Lambda^{(M)(m, m_2){\rm (IV)}}_{k_1,k_2})}
\big](\tau, z_1, z_2,t)$
{\allowdisplaybreaks
\begin{eqnarray*}
&=&
e^{\frac{2\pi im}{M}(t-\frac{\tau}{4})} \, 
e^{\frac{2\pi im}{M}[-(k_1+k_2-\frac12)z_1+(k_1+\frac12)z_2]} \, 
q^{\frac{m}{M}[k_1(k_1+k_2)+\frac12k_2]}
\\[2mm]
& & \hspace{-5mm}
\times \,\ \Phi^{(A(1|1))[m,-m_2+1]}
(M\tau, \,\ -z_1+(k_1+\tfrac12)\tau, \,\ z_2+(k_1+k_2-\tfrac12)\tau, \,\ 0)
\end{eqnarray*}}
\end{enumerate}
\end{enumerate}
\end{prop}

\subsection{Characters of non-irreducible $\widehat{A}(1,1)$-modules}
\label{subsec:n2n4:A(11):char:non-irred}


In this section, just like in section 7 of \cite{W2023a}, 
we consider the non-irreducible $\widehat{A}(1,1)$-module 
$$
\ddot{L}(\Lambda^{(m,m_2)}) \,\ := \,\ 
L(\Lambda^{(m,m_2)}) \oplus L(\Lambda^{(m,m_2+1)})
$$
where $m \in \nnn$ and $m_2 \in \zzz_{\geq 0}$ such that $m_2 \leq m-1$. 
Since
$$
\Lambda^{(m,m_2+1)} \, = \, 
m\Lambda_0+\dfrac{m_2+1}{2}(\alpha_1+\alpha_3)
\, = \, 
\Lambda^{(m,m_2)} + \dfrac12 (\alpha_1+\alpha_3)
\, = \, 
\Lambda^{(m,m_2)} + \alpha_1
$$
and $\alpha_1$ is an odd root, the parity of the highest weight 
of $L(\Lambda^{(m,m_2+1)})$
is opposite to that of $L(\Lambda^{(m,m_2)})$. So the 
character and the super-character of $\ddot{L}(\Lambda^{(m,m_2)})$
are given by the following formulas:
\begin{equation}
\begin{array}{ccc}
{\rm ch}^{(+)}_{\ddot{L}(\Lambda^{(m,m_2)})} &=&
{\rm ch}^{(+)}_{L(\Lambda^{(m,m_2)})}
+
{\rm ch}^{(+)}_{L(\Lambda^{(m,m_2+1)})}
\\[3mm]
{\rm ch}^{(-)}_{\ddot{L}(\Lambda^{(m,m_2)})} &=&
{\rm ch}^{(-)}_{L(\Lambda^{(m,m_2)})}
-
{\rm ch}^{(-)}_{L(\Lambda^{(m,m_2+1)})}
\end{array}
\label{n2n4:eqn:2024-108a}
\end{equation}

\medskip

We consider the corresponding principal admissible $\widehat{A}(1,1)$-modules
\begin{equation}
\ddot{L}(\Lambda_{k_1,k_2}^{(M)(m, m_2){\rm (\heartsuit)}})
\, := \, 
L(\Lambda_{k_1,k_2}^{(M)(m, m_2){\rm (\heartsuit)}}) \oplus 
L(\Lambda_{k_1,k_2}^{(M)(m, m_2+1){\rm (\heartsuit)}}) \hspace{5mm}
\big(\heartsuit = {\rm I} \sim {\rm IV}\big)
\label{n2n4:eqn:2024-108c}
\end{equation}
Then, by \eqref{n2n4:eqn:2024-108a}, the characters of these 
$\widehat{A}(1,1)$-modules  are given by 
\begin{equation}
\begin{array}{ccc}
{\rm ch}^{(\pm)}_{\ddot{L}(\Lambda^{(M)(m,m_2) (\heartsuit)}_{k_1,k_2})} 
&=&
{\rm ch}^{(\pm)}_{L(\Lambda^{(M)(m,m_2) (\heartsuit)}_{k_1,k_2})} 
\, \pm \, 
{\rm ch}^{(\pm)}_{L(\Lambda^{(M)(m,m_2+1) (\heartsuit)}_{k_1,k_2})} 
\\[4mm]
{\rm ch}^{(\pm) {\rm tw}}_{\ddot{L}(\Lambda^{(M)(m,m_2) (\heartsuit)}_{k_1,k_2})} 
&=&
{\rm ch}^{(\pm) {\rm tw}}_{L(\Lambda^{(M)(m,m_2) (\heartsuit)}_{k_1,k_2})} 
\, \pm \, 
{\rm ch}^{(\pm) {\rm tw}}_{L(\Lambda^{(M)(m,m_2+1) (\heartsuit)}_{k_1,k_2})} 
\end{array}
\label{n2n4:eqn:2024-108b}
\end{equation}
From Propositions \ref{n2n4:prop:2024-106a} and \ref{n2n4:prop:2024-107a}
and using Note 2.1 in \cite{W2023a}, we obtain the following:

\medskip 

\begin{prop} 
\label{n2n4:prop:2024-108a}
The numerators of these non-irreducible $\widehat{sl}(2|2)$-modules 
$\ddot{L}(\Lambda^{(M)(m,m_2) (\heartsuit)}_{k_1,k_2})$ are given as follows:
\begin{enumerate}
\item[{\rm 1)}]
\begin{enumerate}
\item[{\rm (i)}] \, $\big[\widehat{R}^{(+)} \cdot 
{\rm ch}^{(+)}_{\ddot{L}(\Lambda^{(M) (m,m_2) {\rm (I))}}_{k_1,k_2})}\big]
(\tau, z_1, z_2,t) =
(-1)^{m_2} \, e^{\frac{2\pi im}{M}[t+(k_1+k_2)z_1-k_1z_2]} \, 
q^{\frac{m}{M}k_1(k_1+k_2)} $
$$
\times \,\ \Phi^{[m,-m_2]}(M\tau, \,\ 
z_1+k_1\tau+\tfrac12, \,\ -z_2+(k_1+k_2)\tau-\tfrac12, \,\ 0)
$$
\item[{\rm (ii)}] \, $\big[\widehat{R}^{(+)} \cdot 
{\rm ch}^{(+)}_{\ddot{L}(\Lambda^{(M) (m,m_2) {\rm (II))}}_{k_1,k_2})}\big]
(\tau, z_1, z_2,t) =
(-1)^{m_2} \, e^{\frac{2\pi im}{M}[t-(k_1+k_2)z_1+k_1z_2]} \, 
q^{\frac{m}{M}k_1(k_1+k_2)} $
$$
\times \,\ \Phi^{[m,-m_2]}(M\tau, \,\ 
-z_1+k_1\tau+\tfrac12, \,\ z_2+(k_1+k_2)\tau-\tfrac12, \,\ 0)
$$
\item[{\rm (iii)}] \, $\big[\widehat{R}^{(+)} \cdot 
{\rm ch}^{(+)}_{\ddot{L}(\Lambda^{(M) (m,m_2) {\rm (III))}}_{k_1,k_2})}\big]
(\tau, z_1, z_2,t) =
(-1)^{m_2} \, e^{\frac{2\pi im}{M}[t-(k_1+k_2)z_2+k_1z_1]} \, 
q^{\frac{m}{M}k_1(k_1+k_2)} $
$$
\times \,\ \Phi^{[m,-m_2]}(M\tau, \,\ 
-z_2+k_1\tau+\tfrac12, \,\ z_1+(k_1+k_2)\tau-\tfrac12, \,\ 0)
$$
\item[{\rm (iv)}] \, $\big[\widehat{R}^{(+)} \cdot 
{\rm ch}^{(+)}_{\ddot{L}(\Lambda^{(M) (m,m_2) {\rm (IV))}}_{k_1,k_2})}\big]
(\tau, z_1, z_2,t) =
(-1)^{m_2} \, e^{\frac{2\pi im}{M}[t+(k_1+k_2)z_2-k_1z_1]} \, 
q^{\frac{m}{M}k_1(k_1+k_2)} $
$$
\times \,\ \Phi^{[m,-m_2]}(M\tau, \,\ 
z_2+k_1\tau+\tfrac12, \,\ -z_1+(k_1+k_2)\tau-\tfrac12, \,\ 0)
$$
\end{enumerate}
\item[{\rm 2)}]
\begin{enumerate}
\item[{\rm (i)}] \, $\big[\widehat{R}^{(-)} \cdot 
{\rm ch}^{(-)}_{\ddot{L}(\Lambda^{(M) (m,m_2) {\rm (I))}}_{k_1,k_2})}\big]
(\tau, z_1, z_2,t) = \,\ 
- \, e^{\frac{2\pi im}{M}[t+(k_1+k_2)z_1-k_1z_2]} \, 
q^{\frac{m}{M}k_1(k_1+k_2)} $
$$
\times \,\ 
\Phi^{[m,-m_2]}(M\tau, \,\ z_1+k_1\tau, \,\ -z_2+(k_1+k_2)\tau, \,\ 0)
$$
\item[{\rm (ii)}] \, $\big[\widehat{R}^{(-)} \cdot 
{\rm ch}^{(-)}_{\ddot{L}(\Lambda^{(M) (m,m_2) {\rm (II))}}_{k_1,k_2})}\big]
(\tau, z_1, z_2,t) = \,\ 
- \, e^{\frac{2\pi im}{M}[t-(k_1+k_2)z_1+k_1z_2]} \, 
q^{\frac{m}{M}k_1(k_1+k_2)} $
$$
\times \,\ 
\Phi^{[m,-m_2]}(M\tau, \,\ -z_1+k_1\tau, \,\ z_2+(k_1+k_2)\tau, \,\ 0)
$$
\item[{\rm (iii)}] \, $\big[\widehat{R}^{(-)} \cdot 
{\rm ch}^{(-)}_{\ddot{L}(\Lambda^{(M) (m,m_2) {\rm (III))}}_{k_1,k_2})}\big]
(\tau, z_1, z_2,t) = \,\ 
- \, e^{\frac{2\pi im}{M}[t-(k_1+k_2)z_2+k_1z_1]} \, 
q^{\frac{m}{M}k_1(k_1+k_2)} $
$$
\times \,\ 
\Phi^{[m,-m_2]}(M\tau, \,\ -z_2+k_1\tau, \,\ z_1+(k_1+k_2)\tau, \,\ 0)
$$
\item[{\rm (iv)}] \, $\big[\widehat{R}^{(-)} \cdot 
{\rm ch}^{(-)}_{\ddot{L}(\Lambda^{(M) (m,m_2) {\rm (IV))}}_{k_1,k_2})}\big]
(\tau, z_1, z_2,t) = \,\ 
- \, e^{\frac{2\pi im}{M}[t+(k_1+k_2)z_2-k_1z_1]} \, 
q^{\frac{m}{M}k_1(k_1+k_2)} $
$$
\times \,\ 
\Phi^{[m,-m_2]}(M\tau, \,\ z_2+k_1\tau, \,\ -z_1+(k_1+k_2)\tau, \,\ 0)
$$
\end{enumerate}
\item[$1)^{\rm tw}$]
\begin{enumerate}
\item[{\rm (i)}] \,\ $\big[\widehat{R}^{(+) \, {\rm tw}} \cdot 
{\rm ch}^{(+) \, {\rm tw}}_{\ddot{L}(\Lambda^{(M) (m,m_2) {\rm (I))}}_{k_1,k_2})}\big]
(\tau, z_1, z_2,t)$
{\allowdisplaybreaks
\begin{eqnarray*}
& & \hspace{-20mm}
= \,\ (-1)^{m_2}
e^{\frac{2\pi im}{M}(t-\frac{\tau}{4})} \, 
e^{\frac{2\pi im}{M}[(k_1+\frac12)z_1-(k_1+k_2-\frac12)z_2]} \, 
q^{\frac{m}{M}[k_1(k_1+k_2)+\frac12k_2]} 
\\[2mm]
& & \hspace{-10mm}
\times \,\ \Phi^{[m,-m_2]}(M\tau, \,\ 
-z_2+(k_1+\tfrac12)\tau+\tfrac12, \,\ 
z_1+(k_1+k_2-\tfrac12)\tau-\tfrac12, \,\ 0)
\end{eqnarray*}}
\item[{\rm (ii)}] \,\ $\big[\widehat{R}^{(+) \, {\rm tw}} \cdot 
{\rm ch}^{(+) \, {\rm tw}}_{\ddot{L}(\Lambda^{(M) (m,m_2) {\rm (II))}}_{k_1,k_2})}\big]
(\tau, z_1, z_2,t) $
{\allowdisplaybreaks
\begin{eqnarray*}
& & \hspace{-20mm}
= \,\ (-1)^{m_2}
e^{\frac{2\pi im}{M}(t-\frac{\tau}{4})} \, 
e^{\frac{2\pi im}{M}[-(k_1-\frac12)z_1+(k_1+k_2+\frac12)z_2]} \, 
q^{\frac{m}{M}[k_1(k_1+k_2)-\frac12k_2]} \, 
\\[2mm]
& & \hspace{-10mm}
\times \,\ \Phi^{[m,-m_2]}(M\tau, \,\ 
z_2+(k_1-\tfrac12)\tau+\tfrac12, \,\ 
-z_1+(k_1+k_2+\tfrac12)\tau-\tfrac12, \,\ 0)
\end{eqnarray*}}
\item[{\rm (iii)}] \,\ $\big[\widehat{R}^{(+) \, {\rm tw}} \cdot 
{\rm ch}^{(+) \, {\rm tw}}_{\ddot{L}(\Lambda^{(M) (m,m_2) {\rm (III))}}_{k_1,k_2})}\big]
(\tau, z_1, z_2,t)$
{\allowdisplaybreaks
\begin{eqnarray*}
& & \hspace{-20mm}
= \,\ (-1)^{m_2}
e^{\frac{2\pi im}{M}(t-\frac{\tau}{4})} \, 
e^{\frac{2\pi im}{M}[(k_1+k_2+\frac12)z_1-(k_1-\frac12)z_2]} \, 
q^{\frac{m}{M}[k_1(k_1+k_2)-\frac12k_2]}
\\[2mm]
& & \hspace{-10mm}
\times \,\ \Phi^{[m,-m_2]}(M\tau, \,\ 
z_1+(k_1-\tfrac12)\tau+\tfrac12, \,\ 
-z_2+(k_1+k_2+\tfrac12)\tau-\tfrac12, \,\ 0)
\end{eqnarray*}}
\item[{\rm (iv)}] \,\ $\big[\widehat{R}^{(+) \, {\rm tw}} \cdot 
{\rm ch}^{(+) \, {\rm tw}}_{\ddot{L}(\Lambda^{(M) (m,m_2) {\rm (IV))}}_{k_1,k_2})}\big]
(\tau, z_1, z_2,t)$ 
{\allowdisplaybreaks
\begin{eqnarray*}
& & \hspace{-20mm}
= \,\ (-1)^{m_2}
e^{\frac{2\pi im}{M}(t-\frac{\tau}{4})} \, 
e^{\frac{2\pi im}{M}[-(k_1+k_2-\frac12)z_1+(k_1+\frac12)z_2]} \, 
q^{\frac{m}{M}[k_1(k_1+k_2)+\frac12k_2]}
\\[2mm]
& & \hspace{-10mm}
\times \,\ \Phi^{[m,-m_2]}(M\tau, \,\ 
-z_1+(k_1+\tfrac12)\tau+\tfrac12, \,\ 
z_2+(k_1+k_2-\tfrac12)\tau-\tfrac12, \,\ 0)
\end{eqnarray*}}
\end{enumerate}
\item[$2)^{\rm tw}$]
\begin{enumerate}
\item[{\rm (i)}] \,\ $\big[\widehat{R}^{(-) \, {\rm tw}} \cdot 
{\rm ch}^{(-) \, {\rm tw}}_{\ddot{L}(\Lambda^{(M) (m,m_2) {\rm (I))}}_{k_1,k_2})}\big]
(\tau, z_1, z_2,t)$
{\allowdisplaybreaks
\begin{eqnarray*}
& & \hspace{-20mm}
= \,\ - \, 
e^{\frac{2\pi im}{M}(t-\frac{\tau}{4})} \, 
e^{\frac{2\pi im}{M}[(k_1+\frac12)z_1-(k_1+k_2-\frac12)z_2]} \, 
q^{\frac{m}{M}[k_1(k_1+k_2)+\frac12k_2]} 
\\[2mm]
& & \hspace{-10mm}
\times \,\ \Phi^{[m,-m_2]}(M\tau, \,\ 
-z_2+(k_1+\tfrac12)\tau, \,\ z_1+(k_1+k_2-\tfrac12)\tau, \,\ 0)
\end{eqnarray*}}
\item[{\rm (ii)}] \,\ $\big[\widehat{R}^{(-) \, {\rm tw}} \cdot 
{\rm ch}^{(-) \, {\rm tw}}_{\ddot{L}(\Lambda^{(M) (m,m_2) {\rm (II))}}_{k_1,k_2})}\big]
(\tau, z_1, z_2,t) $
{\allowdisplaybreaks
\begin{eqnarray*}
& & \hspace{-20mm}
= \,\ - \, 
e^{\frac{2\pi im}{M}(t-\frac{\tau}{4})} \, 
e^{\frac{2\pi im}{M}[-(k_1-\frac12)z_1+(k_1+k_2+\frac12)z_2]} \, 
q^{\frac{m}{M}[k_1(k_1+k_2)-\frac12k_2]} \, 
\\[2mm]
& & \hspace{-10mm}
\times \,\ \Phi^{[m,-m_2]}(M\tau, \,\ 
z_2+(k_1-\tfrac12)\tau, \,\ -z_1+(k_1+k_2+\tfrac12)\tau, \,\ 0)
\end{eqnarray*}}
\item[{\rm (iii)}] \,\ $\big[\widehat{R}^{(-) \, {\rm tw}} \cdot 
{\rm ch}^{(-) \, {\rm tw}}_{\ddot{L}(\Lambda^{(M) (m,m_2) {\rm (III))}}_{k_1,k_2})}\big]
(\tau, z_1, z_2,t)$
{\allowdisplaybreaks
\begin{eqnarray*}
& & \hspace{-20mm}
= \,\ - \, 
e^{\frac{2\pi im}{M}(t-\frac{\tau}{4})} \, 
e^{\frac{2\pi im}{M}[(k_1+k_2+\frac12)z_1-(k_1-\frac12)z_2]} \, 
q^{\frac{m}{M}[k_1(k_1+k_2)-\frac12k_2]}
\\[2mm]
& & \hspace{-10mm}
\times \,\ \Phi^{[m,-m_2]}(M\tau, \,\ 
z_1+(k_1-\tfrac12)\tau, \,\ -z_2+(k_1+k_2+\tfrac12)\tau, \,\ 0)
\end{eqnarray*}}
\item[{\rm (iv)}] \,\ $\big[\widehat{R}^{(-) \, {\rm tw}} \cdot 
{\rm ch}^{(-) \, {\rm tw}}_{\ddot{L}(\Lambda^{(M) (m,m_2) {\rm (IV))}}_{k_1,k_2})}\big]
(\tau, z_1, z_2,t)$ 
{\allowdisplaybreaks
\begin{eqnarray*}
& & \hspace{-20mm}
= \,\ - \, 
e^{\frac{2\pi im}{M}(t-\frac{\tau}{4})} \, 
e^{\frac{2\pi im}{M}[-(k_1+k_2-\frac12)z_1+(k_1+\frac12)z_2]} \, 
q^{\frac{m}{M}[k_1(k_1+k_2)+\frac12k_2]}
\\[2mm]
& & \hspace{-10mm}
\times \,\ \Phi^{[m,-m_2]}(M\tau, \,\ 
-z_1+(k_1+\tfrac12)\tau, \,\ z_2+(k_1+k_2-\tfrac12)\tau, \,\ 0)
\end{eqnarray*}}
\end{enumerate}
\end{enumerate}
\end{prop}

\section{Characters of N=4 superconformal modules}
\label{sec:n2n4:N=4:characters}

\subsection{Quantum Hamiltonian reduction of $\widehat{A}(1,1)$-module}
\label{subsec:n2n4:A(11):quantum}


We now consider the quantum Hamiltonian reduction of $\widehat{A}(1,1)$ 
associated to the pair $(x=\frac12\theta, f=e_{-\theta})$. First 
we note that 
$$
\text{the central charge of} \, \overset{N=4}{H}(\Lambda) \,\ = \,\ 
-6 \, \times \, \big\{
\text{the central charge of} \, L(\Lambda) \, +1\big\}
$$
so
\begin{equation}
\overset{N=4}{c}{}^{(M,m)} \, := \,\ 
\text{the central charge of} \, 
\overset{N=4}{H}(\Lambda^{(M)(m,m_2)(\heartsuit)}_{k_1,k_2})
\, = \, 
-6 \, \Big(\frac{m}{M}+1\Big)
\label{n2n4:eqn:2024-109a}
\end{equation}

Taking a basis $J_0^{(N=4)} := \alpha_2^{\vee}=-\alpha_2$ of 
$\overline{\hhh}^f$, the character and the super-character of the 
quantum Hamiltonian reduction $\overset{N=4}{H}(\Lambda)$ of an 
$\widehat{A}(1,1)$-module $L(\Lambda)$ and its twisted module 
$\overset{N=4}{H}{}^{\rm tw}(\Lambda)$ are obtained by the formulas 
(5.2a) and (5.2b) in \cite{W2023a}.
Also the numbers $(h_{\Lambda}, s_{\Lambda})$ and 
$(h_{\Lambda}^{\rm tw}, s_{\Lambda}^{\rm tw})$, for 
$\Lambda=\Lambda^{(M)(m,m_2)(\heartsuit)}_{k_1,k_2}$ and the twist 
by $w_0=r_{\alpha_2}t_{-\frac12 \alpha_2}$, are obtained by similar 
calculation as in section 6 of \cite{W2023a}:

\medskip

\begin{lemma} 
\label{n2n4:lemma:2024-109a}
Let $\Lambda =\Lambda^{(M)(m,m_2)(\heartsuit)}_{k_1,k_2}$  
$(\heartsuit =$ {\rm I} $\sim$ {\rm IV)}. Then 
\begin{enumerate}
\item[{\rm 1)}] \,\ $s_{\Lambda} \,\ = \,\ 
s^{(M)(m,m_2)(\heartsuit)}_{k_1,k_2} \,\ = \,\ 
\left\{
\begin{array}{lcl} \hspace{3mm}
\dfrac{mk_2}{M} \, - \, m_2 
& \quad & {\rm if} \quad 
\heartsuit \, = \, {\rm I \,\ or \,\ IV}
\\[4mm]
- \, \dfrac{mk_2}{M} \, + \, m_2 \, - \, 2 
& \quad & {\rm if} \quad 
\heartsuit \, = \, {\rm II \,\ or \,\ III}
\end{array}\right. $
\item[{\rm 2)}] \,\ $s_{\Lambda}^{\rm tw} \,\ = \,\ 
s^{(M)(m,m_2)(\heartsuit) {\rm tw}}_{k_1,k_2} \,\ = \,\ 
\left\{
\begin{array}{rcl}
- \, \dfrac{m(k_2-1)}{M} \, + \, m_2 \, - \, 1 
& \quad & {\rm if} \quad 
\heartsuit \, = \, {\rm I \,\ or \,\ IV}
\\[4mm]
\dfrac{m(k_2+1)}{M} \, - \, m_2 \, + \, 1 
& \quad & {\rm if} \quad 
\heartsuit \, = \, {\rm II \,\ or \,\ III}
\end{array}\right. $
\end{enumerate}
\end{lemma}

\medskip

\begin{lemma} 
\label{n2n4:lemma:2024-109b}
Let $\Lambda= \Lambda^{(M)(m, m_2)(\heartsuit)}_{k_1,k_2}$ 
$(\heartsuit = {\rm I} \, \sim \, {\rm IV})$. Then 
\begin{enumerate}
\item[{\rm 1)}] \,\ 
$h_{\Lambda} \,\ = \,\ h^{(M)(m, m_2)(\heartsuit)}_{k_1,k_2}$
{\allowdisplaybreaks
\begin{eqnarray*}
&=& \left\{
\begin{array}{ll}
\dfrac{m}{M} \Big(k_1+\dfrac12\Big)\Big(k_1+k_2+\dfrac12\Big)
- (m_2-1)\Big(k_1+\dfrac12\Big)
- \dfrac14 \Big(\dfrac{m}{M}+2\Big) 
& \hspace{-2mm}
\text{{\rm if} \hspace{0mm}
$\heartsuit =$ {\rm I} {\rm or} {\rm III} }
\\[4mm]
\dfrac{m}{M} \Big(k_1-\dfrac12\Big)\Big(k_1+k_2-\dfrac12\Big)
- (m_2-1)\Big(k_1-\dfrac12\Big)
- \dfrac14\Big(\dfrac{m}{M}+2\Big) 
& \hspace{-2mm}
\text{{\rm if} \hspace{0mm} 
$\heartsuit =$ {\rm II} {\rm or} {\rm IV} }
\end{array}\right. 
\\[2mm]
&=& 
\left\{
\begin{array}{ll}
\dfrac{m}{M} \Big(k_1+\dfrac12\Big)\Big(k_1+k_2+\dfrac12\Big)
- (m_2-1)\Big(k_1+\dfrac12\Big) -\dfrac14 + \dfrac{c}{24}
& \text{{\rm if} \hspace{0mm}
$\heartsuit =$ {\rm I} {\rm or} {\rm III} }
\\[4mm]
\dfrac{m}{M} \Big(k_1-\dfrac12\Big)\Big(k_1+k_2-\dfrac12\Big)
- (m_2-1)\Big(k_1-\dfrac12\Big) -\dfrac14 + \dfrac{c}{24}
& \text{{\rm if} \hspace{0mm} 
$\heartsuit =$ {\rm II} {\rm or} {\rm IV} }
\end{array}\right. 
\end{eqnarray*}}
\item[{\rm 2)}] $h_{\Lambda}^{\rm tw} \,\ = \,\
h^{(M)(m, m_2)(\heartsuit){\rm tw}}_{k_1,k_2}$
{\allowdisplaybreaks
\begin{eqnarray*}
&=&
\left\{
\begin{array}{lcl}
\dfrac{m(k_1+1)(k_1+k_2)}{M} 
\, - \, (k_1+1)(m_2-1)
\, - \, \dfrac14 \, \Big(\dfrac{m}{M} \, + 1\Big)
& & {\rm if} \quad \heartsuit \, = \, {\rm I}
\\[4mm]
\dfrac{m(k_1-1)(k_1+k_2)}{M} 
\, - \, (k_1-1)(m_2-1)
\, - \, \dfrac14 \, \Big(\dfrac{m}{M} \, + 1\Big)
& & {\rm if} \quad \heartsuit \, = \, {\rm II}
\\[4mm]
\dfrac{mk_1(k_1+k_2+1)}{M} 
\, - \, k_1(m_2 -1)
\, - \, \dfrac14 \, \Big(\dfrac{m}{M} \, + 1\Big)
& & {\rm if} \quad \heartsuit \, = \, {\rm III}
\\[4mm]
\dfrac{mk_1(k_1+k_2-1)}{M} 
\, - \, k_1(m_2-1)
\, - \, \dfrac14 \, \Big(\dfrac{m}{M} \, + 1\Big)
& & {\rm if} \quad \heartsuit \, = \, {\rm IV}
\end{array}\right. 
\\[2mm]
&=&
\left\{
\begin{array}{lcl}
\dfrac{m(k_1+1)(k_1+k_2)}{M} 
\, - \, (k_1+1)(m_2-1)
\, + \, \dfrac{c}{24}
&\quad & {\rm if} \quad \heartsuit \, = \, {\rm I}
\\[4mm]
\dfrac{m(k_1-1)(k_1+k_2)}{M} 
\, - \, (k_1-1)(m_2-1)
\, + \, \dfrac{c}{24}
& & {\rm if} \quad \heartsuit \, = \, {\rm II}
\\[4mm]
\dfrac{mk_1(k_1+k_2+1)}{M} 
\, - \, k_1(m_2-1) 
\, + \, \dfrac{c}{24}
& & {\rm if} \quad \heartsuit \, = \, {\rm III}
\\[4mm]
\dfrac{mk_1(k_1+k_2-1)}{M} 
\, - \, k_1(m_2-1)
\, + \, \dfrac{c}{24}
& & {\rm if} \quad \heartsuit \, = \, {\rm IV}
\end{array}\right.
\end{eqnarray*}}
\end{enumerate}
where $c := -6 (\frac{m}{M}+1)$ is the central charge of 
$\overset{N=4}{H}(\Lambda^{(M)(m,m_2)(\heartsuit)}_{k_1,k_2})$.
\end{lemma}

\medskip

In the above formulas for $h_{\Lambda}$ and $h_{\Lambda}^{\rm tw}$, 
the term $\frac{c}{24}$ is related closely to the normalization factor 
in the normalized character (cf. \cite{KW2004}).

\medskip

The domain of the parameters $(k_1,k_2)$ for the principal admissible
simple subset $\Pi^{(M), (\heartsuit)}_{k_1,k_2}$ is given by 
(4.1) in \cite{W2023a}, whereas Lemma 8.1 in \cite{W2023a} says that 
the space $\overset{N=4}{H}(\Lambda^{(M)(m,m_2)(\heartsuit)}_{k_1,k_2})$ 
of the quantum Hamiltonian reduction vanishes if 
$\alpha_0 \in \Pi^{(M), (\heartsuit)}_{k_1,k_2}$. 
Then the domain $\Omega^{(M), (\heartsuit)}$ of the parameters $(k_1,k_2)$ 
for $\Pi^{(M), (\heartsuit)}_{k_1,k_2}$, excluding the case 
$\alpha_0 \in \Pi^{(M), (\heartsuit)}_{k_1,k_2}$, is as follows:
\begin{equation}
\begin{array}{lclcl}
\Omega^{(M), ({\rm I})} &=& \big\{(k_1, k_2) \, \in \, (\zzz_{\geq 0})^2
&;& 2k_1+k_2 \, \leq \, M-2 \big\}
\\[2mm]
\Omega^{(M), ({\rm II})} &=& \big\{(k_1, k_2) \, \in \, \nnn^2
&;& 2k_1+k_2 \, \leq \, M \big\}
\\[2mm]
\Omega^{(M), ({\rm III})} &=& \big\{(k_1, k_2) \, \in \, \zzz_{\geq 0} \times \nnn
&;& 2k_1+k_2 \, \leq \, M-2 \big\}
\\[2mm]
\Omega^{(M), ({\rm IV})} &=& \big\{(k_1, k_2) \, \in \, \nnn \times \zzz_{\geq 0}
&;& 2k_1+k_2 \, \leq \, M \big\}
\end{array}
\label{n2n4:eqn:2024-119c}
\end{equation}
In this paper, we write simply $\Omega^{(\heartsuit)}$ 
for $\Omega^{(M), (\heartsuit)}$. 

\medskip

From the formula (4.1) in \cite{W2023a} and Lemmas 
\ref{n2n4:lemma:2024-109a} and \ref{n2n4:lemma:2024-109b}, 
we obtain the equivalence of N=4 modules as follows:

\medskip

\begin{prop} 
\label{n2n4:prop:2024-109a}
Let $M$ and $m$ be coprime positive integers ,and 
$m_2$ be a non-negative integer such that $0 \leq m_2 \leq m$, 
and $k_1$ and  $k_2$ be integers satisfying (4.1) in 
\cite{W2023a}. Then 
\begin{enumerate}
\item[{\rm 1)}] if \, $k_1, k_2 \geq 0$ \, and \, $2k_1+k_2 \leq M-2$,
$\left\{
\begin{array}{lcl}
\overset{N=4}{H}(\Lambda^{(M)(m,m_2)({\rm I})}_{k_1,k_2}) \hspace{-2mm}
&\cong& \hspace{-2mm}
\overset{N=4}{H}(\Lambda^{(M)(m,m_2)({\rm IV})}_{k_1+1,k_2})
\\[3mm]
\overset{N=4}{H}{}^{\rm tw}(\Lambda^{(M)(m,m_2)({\rm I})}_{k_1,k_2}) 
\hspace{-2mm} &\cong& \hspace{-2mm}
\overset{N=4}{H}{}^{\rm tw}(\Lambda^{(M)(m,m_2)({\rm IV})}_{k_1+1,k_2})
\end{array}\right. $
\item[{\rm 2)}] if $\left\{
\begin{array}{l}
k_1\geq 0 \\[1mm]
k_2 \geq 1
\end{array}\right.
$ and \, $2k_1+k_2 \leq M-2$,
$\left\{
\begin{array}{lcl}
\overset{N=4}{H}(\Lambda^{(M)(m,m_2)({\rm III})}_{k_1,k_2}) \hspace{-2mm}
&\cong& \hspace{-2mm}
\overset{N=4}{H}(\Lambda^{(M)(m,m_2)({\rm II})}_{k_1+1,k_2})
\\[3mm]
\overset{N=4}{H}{}^{\rm tw}(\Lambda^{(M)(m,m_2)({\rm III})}_{k_1,k_2}) 
\hspace{-2mm} &\cong& \hspace{-2mm}
\overset{N=4}{H}{}^{\rm tw}(\Lambda^{(M)(m,m_2)({\rm II})}_{k_1+1,k_2})
\end{array}\right. $
\end{enumerate}

So we need to consider the characters and twisted characters of 
$\overset{N=4}{H}(\Lambda^{(M)(m,m_2)(\heartsuit)}_{k_1,k_2})$
only for $\heartsuit =$ {\rm I} and {\rm III}.
\end{prop}

\medskip

From Lemma \ref{n2n4:lemma:2024-109b}, we see that
\begin{equation}
\begin{array}{ll}
h_{\Lambda^{(M)(m,m_2)(\heartsuit)}_{k_1,k_2}} 
& \hspace{5mm} > \,\  
h_{\Lambda^{(M)(m,m_2+1)(\heartsuit)}_{k_1,k_2}}
\\[3.5mm]
h_{\Lambda^{(M)(m,m_2)({\rm I})}_{k_1,k_2}}^{\rm tw} 
& \hspace{5mm} > \,\ 
h_{\Lambda^{(M)(m,m_2+1)({\rm I})}_{k_1,k_2}}^{\rm tw} 
\\[4mm]
h_{\Lambda^{(M)(m,m_2)({\rm III})}_{k_1,k_2}}^{\rm tw}
& \left\{
\begin{array}{lcl}
> \,\ h_{\Lambda^{(M)(m,m_2+1)({\rm III})}_{k_1,k_2}}^{\rm tw}
& & {\rm if} \quad k_1 >0
\\[3.5mm]
= \,\ h_{\Lambda^{(M)(m,m_2+1)({\rm III})}_{k_1,k_2}}^{\rm tw}
& & {\rm if} \quad k_1 =0
\end{array} \right.
\end{array}
\label{n2n4:eqn::2024-119d}
\end{equation}

\subsection{Non-irreducible N=4 modules}
\label{subsec:n2n4:N=4:non-irred}


In this section we cosider the N=4 superconformal modules 
$\overset{N=4}{\ddot{H}}(\Lambda_{k_1,k_2}^{(M)(m, m_2){\rm (\heartsuit)}})$
obtained by the quantum Hamiltonian reduction of 
the non-irreducible $\widehat{A}(1,1)$-modules 
$\ddot{L}(\Lambda_{k_1,k_2}^{(M)(m, m_2){\rm (\heartsuit)}})$,
where we need to consider only the cases $\heartsuit =$ I and III by 
Proposition \ref{n2n4:prop:2024-109a}.
The numerators of the characters of these N=4 modules are computed 
by using the formulas (5.2a) and (5.2b) in \cite{W2023a} and Proposition 
\ref{n2n4:prop:2024-108a}, and obtained as follows:

\medskip

\begin{lemma}  
\label{n2n4:lemma:2024-110a}
The numerators of non-twisted and twisted (super-)characters of these N=4 modules 
$\overset{N=4}{\ddot{H}}(\Lambda^{(M)(m,m_2) (\heartsuit)}_{k_1,k_2})$ 
are given as follows:
\begin{enumerate}
\item[{\rm 1)}]
\begin{enumerate}
\item[{\rm (I)}] $\big[\overset{N=4}{R}{}^{(+)} \cdot 
{\rm ch}^{(+)}_{\ddot{H}(\Lambda^{(M)(m, m_2){\rm (I)}}_{k_1,k_2})}
\big](\tau,z)
= \, 
(-1)^{m_2} \, 
\Psi^{[M,m,-m_2; \frac12]}_{k_1+\frac12, \, k_1+k_2+\frac12; \, \frac12}
(\tau, z, -z,0)$
\item[{\rm (III)}] $\big[\overset{N=4}{R}{}^{(+)} \cdot 
{\rm ch}^{(+)}_{\ddot{H}(\Lambda^{(M)(m, m_2){\rm (III)}}_{k_1,k_2})}\big](\tau,z)
= \, 
(-1)^{m_2} \, 
\Psi^{[M,m,-m_2; \frac12]}_{k_1+\frac12, \, k_1+k_2+\frac12; \, \frac12}
(\tau, -z, z,0)$
\end{enumerate}
\item[{\rm 2)}]
\begin{enumerate}
\item[{\rm (I)}] $\big[\overset{N=4}{R}{}^{(-)} \cdot 
{\rm ch}^{(-)}_{\ddot{H}(\Lambda^{(M)(m, m_2){\rm (I)}}_{k_1,k_2})}\big](\tau,z)
\,\ = \,\ - \, 
\Psi^{[M,m,-m_2; 0]}_{k_1+\frac12, \, k_1+k_2+\frac12; \, \frac12}
(\tau, z, -z,0)$
\item[{\rm (III)}] $\big[\overset{N=4}{R}{}^{(-)} \cdot 
{\rm ch}^{(-)}_{\ddot{H}(\Lambda^{(M)(m, m_2){\rm (III)}}_{k_1,k_2})}\big](\tau,z)
\,\ = \,\ - \, 
\Psi^{[M,m,-m_2; 0]}_{k_1+\frac12, \, k_1+k_2+\frac12; \, \frac12}
(\tau, -z, z,0)$
\end{enumerate}
\item[$1)^{\rm tw}$]
\begin{enumerate}
\item[{\rm (I)}] $\big[\overset{N=4}{R}{}^{(+){\rm tw}} \cdot 
{\rm ch}^{(+){\rm tw}}_{\ddot{H}(\Lambda^{(M)(m, m_2){\rm (I)}}_{k_1,k_2})}
\big](\tau,z)
\,\ = \,\ 
(-1)^{m_2} \, \Psi^{[M,m,-m_2; \frac12]}_{k_1+1, \, k_1+k_2; \, 0}
(\tau, \, -z, \, z, \, 0)$
\item[{\rm (III)}] $\big[\overset{N=4}{R}{}^{(+){\rm tw}} \cdot 
{\rm ch}^{(+){\rm tw}}_{\ddot{H}(\Lambda^{(M)(m, m_2){\rm (III)}}_{k_1,k_2})}\big](\tau,z)
\,\ = \,\ 
(-1)^{m_2} \, \Psi^{[M,m,-m_2; \frac12]}_{k_1, \, k_1+k_2+1; \, 0}
(\tau, \, z, \, -z, \, 0)$
\end{enumerate}
\item[$2)^{\rm tw}$]
\begin{enumerate}
\item[{\rm (I)}] $\big[\overset{N=4}{R}{}^{(-){\rm tw}} \cdot 
{\rm ch}^{(-){\rm tw}}_{\ddot{H}(\Lambda^{(M)(m, m_2){\rm (I)}}_{k_1,k_2})}\big](\tau,z)
\,\ = \,\ - \, 
\Psi^{[M,m,-m_2; 0]}_{k_1+1, \, k_1+k_2; \, 0}
(\tau, \, -z, \, z, \, 0)$
\item[{\rm (III)}] $\big[\overset{N=4}{R}{}^{(-){\rm tw}} \cdot 
{\rm ch}^{(-){\rm tw}}_{\ddot{H}(\Lambda^{(M)(m, m_2){\rm (III)}}_{k_1,k_2})}\big](\tau,z)
\,\ = \,\ - \, 
\Psi^{[M,m,-m_2; 0]}_{k_1, \, k_1+k_2+1; \, 0}
(\tau, \, z, \, -z, \, 0)$
\end{enumerate}
\end{enumerate}
\end{lemma}

\begin{proof} \,\ These formulas are obtained easily by using the formulas
$$
\Big[\overset{N=4}{R}{}^{(\pm)} \cdot 
{\rm ch}^{(\pm)}_{\overset{N=4}{\ddot{H}}
\big(\Lambda^{(M)(m, m_2)(\heartsuit)}_{k_1,k_2}\big)}\Big](\tau,z) 
\, = \, \Big[
\widehat{R}^{(\pm)} \cdot {\rm ch}^{(\pm)}_{\ddot{L}
\big(\Lambda^{(M)(m, m_2)(\heartsuit)}_{k_1,k_2}\big)}\Big]
\Big(\tau, \,\ z+\frac{\tau}{2}, \,\ z-\frac{\tau}{2}, \,\ 
\frac{\tau}{4}\Big)
$$
and 
$$
\Big[\overset{N=4}{R}{}^{(\pm){\rm tw}} \cdot 
{\rm ch}^{(\pm){\rm tw}}_{\overset{N=4}{\ddot{H}}
\big(\Lambda^{(M)(m, m_2)(\heartsuit)}_{k_1,k_2}\big)}\Big](\tau,z) 
\, = \, \Big[
\widehat{R}^{(\pm){\rm tw}} \cdot {\rm ch}^{(\pm){\rm tw}}_{\ddot{L}
\big(\Lambda^{(M)(m, m_2)(\heartsuit)}_{k_1,k_2}\big)}\Big]
\Big(\tau, \,\ z+\frac{\tau}{2}, \,\ z-\frac{\tau}{2}, \,\ 
\frac{\tau}{4}\Big)
$$
and the formulas in Proposition \ref{n2n4:prop:2024-108a}.
\end{proof}

\medskip

Then the characters of 
$\overset{N=4}{\ddot{H}}\big(\Lambda^{(M)(m, m_2)(\heartsuit)}_{k_1,k_2}\big)$
are obtained immediately from the formulas for the N=4 denominators 
given by (5.3) in \cite{W2023a} and the above Lemma \ref{n2n4:lemma:2024-110a}
as follows:

\medskip

\begin{prop} 
\label{n2n4:prop:2024-110a}
The non-twisted and twisted (super-)characters of these N=4 modules \\
$\overset{N=4}{\ddot{H}}(\Lambda^{(M)(m,m_2) (\heartsuit)}_{k_1,k_2})$ 
are given as follows:
\begin{enumerate}
\item[{\rm 1)}]
\begin{enumerate}
\item[{\rm (I)}] \, ${\rm ch}^{(+)}_{
\overset{N=4}{\ddot{H}}
\big(\Lambda^{(M)(m, m_2){\rm (I)}}_{k_1,k_2}\big)}(\tau,z)
\, = \, 
i \, (-1)^{m_2} \, 
\Psi^{[M,m,-m_2; \frac12]}_{k_1+\frac12, \, k_1+k_2+\frac12; \, \frac12}
(\tau, z, -z,0)
\cdot
\dfrac{\vartheta_{00}(\tau,z)^2}{\eta(\tau)^3 \, \vartheta_{11}(\tau,2z)}$
\item[{\rm (III)}] ${\rm ch}^{(+)}_{
\overset{N=4}{\ddot{H}}
\big(\Lambda^{(M)(m, m_2){\rm (III)}}_{k_1,k_2}\big)}(\tau,z)
= 
i (-1)^{m_2} 
\Psi^{[M,m,-m_2; \, \frac12]}_{k_1+\frac12, \, k_1+k_2+\frac12; \, \frac12}
(\tau, -z, z,0)
\cdot
\dfrac{\vartheta_{00}(\tau,z)^2}{\eta(\tau)^3 \vartheta_{11}(\tau,2z)}$
\end{enumerate}
\item[{\rm 2)}]
\begin{enumerate}
\item[{\rm (I)}] \, ${\rm ch}^{(-)}_{
\overset{N=4}{\ddot{H}}
\big(\Lambda^{(M)(m, m_2){\rm (I)}}_{k_1,k_2}\big)}(\tau,z)
\, = \, 
i \,\ 
\Psi^{[M,m,-m_2; \, 0]}_{k_1+\frac12, \, k_1+k_2+\frac12; \, \frac12}
(\tau, z, -z,0)
\cdot
\dfrac{\vartheta_{01}(\tau,z)^2}{\eta(\tau)^3 \, \vartheta_{11}(\tau,2z)}$
\item[{\rm (III)}] \, ${\rm ch}^{(-)}_{
\overset{N=4}{\ddot{H}}
\big(\Lambda^{(M)(m, m_2){\rm (III)}}_{k_1,k_2}\big)}(\tau,z)
\, = \, 
i \,\ 
\Psi^{[M,m,-m_2; \, 0]}_{k_1+\frac12, \, k_1+k_2+\frac12; \, \frac12}
(\tau, -z, z,0)
\cdot
\dfrac{\vartheta_{01}(\tau,z)^2}{\eta(\tau)^3 \, \vartheta_{11}(\tau,2z)}$
\end{enumerate}
\item[$1)^{\rm tw}$]
\begin{enumerate}
\item[{\rm (I)}] \, ${\rm ch}^{(+){\rm tw}}_{
\overset{N=4}{\ddot{H}}
\big(\Lambda^{(M)(m, m_2){\rm (I)}}_{k_1,k_2}\big)}(\tau,z)
\, = \, 
i \, (-1)^{m_2} \, 
\Psi^{[M,m,-m_2; \frac12]}_{k_1+1, \, k_1+k_2; \, 0}
(\tau, -z, z,0)
\cdot
\dfrac{\vartheta_{10}(\tau,z)^2}{\eta(\tau)^3 \, \vartheta_{11}(\tau,2z)}$
\item[{\rm (III)}] \, ${\rm ch}^{(+){\rm tw}}_{
\overset{N=4}{\ddot{H}}
\big(\Lambda^{(M)(m, m_2){\rm (III)}}_{k_1,k_2}\big)}(\tau,z)
\, = \, 
i \, (-1)^{m_2} \, 
\Psi^{[M,m,-m_2; \frac12]}_{k_1, \, k_1+k_2+1; \, 0}
(\tau, z, -z,0)
\cdot
\dfrac{\vartheta_{10}(\tau,z)^2}{\eta(\tau)^3 \, \vartheta_{11}(\tau,2z)}$
\end{enumerate}
\item[$2)^{\rm tw}$]
\begin{enumerate}
\item[{\rm (I)}] \, ${\rm ch}^{(-){\rm tw}}_{
\overset{N=4}{\ddot{H}}
\big(\Lambda^{(M)(m, m_2){\rm (I)}}_{k_1,k_2}\big)}(\tau,z)
\, = \, 
i \, 
\Psi^{[M,m,-m_2; 0]}_{k_1+1, \, k_1+k_2; \, 0}
(\tau, -z, z,0)
\cdot
\dfrac{\vartheta_{11}(\tau,z)^2}{\eta(\tau)^3 \, \vartheta_{11}(\tau,2z)}$
\item[{\rm (III)}] \, ${\rm ch}^{(-){\rm tw}}_{
\overset{N=4}{\ddot{H}}
\big(\Lambda^{(M)(m, m_2){\rm (III)}}_{k_1,k_2}\big)}(\tau,z)
\, = \, 
i \, 
\Psi^{[M,m,-m_2; 0]}_{k_1, \, k_1+k_2+1; \, 0}
(\tau, z, -z,0)
\cdot
\dfrac{\vartheta_{11}(\tau,z)^2}{\eta(\tau)^3 \, \vartheta_{11}(\tau,2z)}$
\end{enumerate}
\end{enumerate}
\end{prop}

\medskip 

Then, comparing the formulas for N=4 characters in Proposition 
\ref{n2n4:prop:2024-104a} with those for N=2 characters in 
Proposition \ref{n2n4:prop:2024-110a}, we obtain the following:

\medskip 

\begin{cor} 
\label{n2n4:cor:2024-110a}
These N=4 characters are written by the N=2 characters as follows:
\begin{enumerate}
\item[]
\begin{enumerate}
\item[{\rm (I)}]
\begin{enumerate}
\item[{\rm (i)}] \, ${\rm ch}^{(+)}_{
\overset{N=4}{\ddot{H}}
\big(\Lambda^{(M)(m, m_2){\rm (I)}}_{k_1,k_2}\big)}(\tau,z)
\, = \, 
i \cdot {\rm ch}^{(+)}_{
\overset{N=2}{H}\big(\lambda^{(M)(m-1, m_2)}_{k_1,k_1+k_2}\big)}
(\tau,z)
\cdot
\dfrac{\vartheta_{00}(\tau,z)}{\vartheta_{11}(\tau,2z)}$
\item[{\rm (ii)}] \, ${\rm ch}^{(-)}_{
\overset{N=4}{\ddot{H}}
\big(\Lambda^{(M)(m, m_2){\rm (I)}}_{k_1,k_2}\big)}(\tau,z)
\, = \, 
i \cdot {\rm ch}^{(-)}_{
\overset{N=2}{H}\big(\lambda^{(M)(m-1, m_2)}_{k_1,k_1+k_2}\big)}
(\tau,z)
\cdot
\dfrac{\vartheta_{01}(\tau,z)}{\vartheta_{11}(\tau,2z)}$
\end{enumerate}
\item[{\rm (III)}]
\begin{enumerate}
\item[{\rm (i)}] \, ${\rm ch}^{(+)}_{
\overset{N=4}{\ddot{H}}
\big(\Lambda^{(M)(m, m_2){\rm (III)}}_{k_1,k_2}\big)}(\tau,z)
\, = \, 
i \cdot {\rm ch}^{(+)}_{
\overset{N=2}{H}\big(\lambda^{(M)(m-1, m_2)}_{k_1,k_1+k_2}\big)}
(\tau,-z)
\cdot
\dfrac{\vartheta_{00}(\tau,z)}{\vartheta_{11}(\tau,2z)}$
\item[{\rm (ii)}] \, ${\rm ch}^{(-)}_{
\overset{N=4}{\ddot{H}}
\big(\Lambda^{(M)(m, m_2){\rm (III)}}_{k_1,k_2}\big)}(\tau,z)
\, = \, 
i \cdot {\rm ch}^{(-)}_{
\overset{N=2}{H}\big(\lambda^{(M)(m-1, m_2)}_{k_1,k_1+k_2}\big)}
(\tau,-z)
\cdot
\dfrac{\vartheta_{01}(\tau,z)}{\vartheta_{11}(\tau,2z)}$
\end{enumerate}
\item[${\rm (I)}^{\rm tw}$]
\begin{enumerate}
\item[{\rm (i)}] \, ${\rm ch}^{(+){\rm tw}}_{
\overset{N=4}{\ddot{H}}
\big(\Lambda^{(M)(m, m_2){\rm (I)}}_{k_1,k_2}\big)}(\tau,z)
\, = \, 
i \cdot {\rm ch}^{(+){\rm tw}}_{
\overset{N=2}{H}\big(\lambda^{(M)(m-1, m_2)}_{k_1,k_1+k_2}\big)}
(\tau,-z)
\cdot
\dfrac{\vartheta_{10}(\tau,z)}{\vartheta_{11}(\tau,2z)}$
\item[{\rm (ii)}] \, ${\rm ch}^{(-){\rm tw}}_{
\overset{N=4}{\ddot{H}}
\big(\Lambda^{(M)(m, m_2){\rm (I)}}_{k_1,k_2}\big)}(\tau,z)
\, = \, 
- \,i \cdot {\rm ch}^{(-){\rm tw}}_{
\overset{N=2}{H}\big(\lambda^{(M)(m-1, m_2)}_{k_1,k_1+k_2}\big)}
(\tau,-z)
\cdot
\dfrac{\vartheta_{11}(\tau,z)}{\vartheta_{11}(\tau,2z)}$
\end{enumerate}
\end{enumerate}
\end{enumerate}
\end{cor}

\begin{proof} \,\ By Proposition \ref{n2n4:prop:2024-110a}, we have 
$$
{\rm ch}^{(+)}_{
\overset{N=4}{\ddot{H}}
\big(\Lambda^{(M)(m, m_2){\rm (I)}}_{k_1,k_2}\big)}(\tau,z)
= i \, 
\underbrace{(-1)^{m_2} \, 
\Psi^{[M,m,-m_2; \frac12]}_{k_1+\frac12, \, k_1+k_2+\frac12; \, \frac12}
(\tau, z, -z,0) \,\
\frac{\vartheta_{00}(\tau,z)}{\eta(\tau)^3} }_{\substack{|| \\[-3mm] 
{\displaystyle {\rm ch}^{(+)}_{
\overset{N=2}{H}\big(\lambda^{(M)(m-1, m_2)}_{k_1,k_1+k_2}\big)}
(\tau,z)
}}} 
\times 
\frac{\vartheta_{00}(\tau,z)}{\vartheta_{11}(\tau,2z)}
$$
proving (I) (i). The proof for the rests is quite similar.
\end{proof}

\medskip

Using  the product expression of $\vartheta_{ab}(\tau,z)$ in \cite{Mum}, 
the formulas in the above Corollary \ref{n2n4:cor:2024-110a} 
are rewritten as follows: 

\medskip

\begin{cor} \,\ 
\label{n2n4:cor:2024-110b}
\begin{enumerate}
\item[{\rm (I)}]
\begin{enumerate}
\item[{\rm (i)}] \,\ ${\rm ch}^{(+)}_{
\overset{N=4}{\ddot{H}}
\big(\Lambda^{(M)(m, m_2){\rm (I)}}_{k_1,k_2}\big)}(\tau,z)$

\vspace{-5mm} 

$$
= \,\ - \, q^{-\frac18} \, e^{2\pi iz}  {\rm ch}^{(+)}_{
\overset{N=2}{H}\big(\lambda^{(M)(m-1, m_2)}_{k_1,k_1+k_2}\big)}
(\tau,z)
\cdot 
\frac{\prod\limits_{n=1}^{\infty}
(1+e^{2\pi iz}q^{n-\frac12})(1+e^{-2\pi iz}q^{n-\frac12})}
{\prod\limits_{n=1}^{\infty}(1-e^{4\pi iz}q^{n-1})(1-e^{-4\pi iz}q^{n})}
$$
\item[{\rm (ii)}] \,\ ${\rm ch}^{(-)}_{
\overset{N=4}{\ddot{H}}
\big(\Lambda^{(M)(m, m_2){\rm (I)}}_{k_1,k_2}\big)}(\tau,z)$

\vspace{-5mm}

$$
= \,\ - \, q^{-\frac18} \, e^{2\pi iz}  {\rm ch}^{(-)}_{
\overset{N=2}{H}\big(\lambda^{(M)(m-1, m_2)}_{k_1,k_1+k_2}\big)}
(\tau,z)
\cdot
\frac{\prod\limits_{n=1}^{\infty}
(1-e^{2\pi iz}q^{n-\frac12})(1-e^{-2\pi iz}q^{n-\frac12})}
{\prod\limits_{n=1}^{\infty}(1-e^{4\pi iz}q^{n-1})(1-e^{-4\pi iz}q^{n})}
$$
\end{enumerate}
\item[{\rm (III)}]
\begin{enumerate}
\item[{\rm (i)}] \,\ ${\rm ch}^{(+)}_{
\overset{N=4}{\ddot{H}}
\big(\Lambda^{(M)(m, m_2){\rm (I)}}_{k_1,k_2}\big)}(\tau,z)$

\vspace{-5mm} 

$$
= \,\ - \, q^{-\frac18} \, e^{2\pi iz}  {\rm ch}^{(+)}_{
\overset{N=2}{H}\big(\lambda^{(M)(m-1, m_2)}_{k_1,k_1+k_2}\big)}
(\tau,-z)
\cdot 
\frac{\prod\limits_{n=1}^{\infty}
(1+e^{2\pi iz}q^{n-\frac12})(1+e^{-2\pi iz}q^{n-\frac12})}
{\prod\limits_{n=1}^{\infty}(1-e^{4\pi iz}q^{n-1})(1-e^{-4\pi iz}q^{n})}
$$
\item[{\rm (ii)}] \,\ ${\rm ch}^{(-)}_{
\overset{N=4}{\ddot{H}}
\big(\Lambda^{(M)(m, m_2){\rm (I)}}_{k_1,k_2}\big)}(\tau,z)$

\vspace{-5mm}

$$
= \,\ - \, q^{-\frac18} \, e^{2\pi iz}  {\rm ch}^{(-)}_{
\overset{N=2}{H}\big(\lambda^{(M)(m-1, m_2)}_{k_1,k_1+k_2}\big)}
(\tau,-z)
\cdot
\frac{\prod\limits_{n=1}^{\infty}
(1-e^{2\pi iz}q^{n-\frac12})(1-e^{-2\pi iz}q^{n-\frac12})}
{\prod\limits_{n=1}^{\infty}(1-e^{4\pi iz}q^{n-1})(1-e^{-4\pi iz}q^{n})}
$$
\end{enumerate}
\item[${\rm (I)}^{\rm tw}$]
\begin{enumerate}
\item[{\rm (i)}] \,\ ${\rm ch}^{(+){\rm tw}}_{
\overset{N=4}{\ddot{H}}
\big(\Lambda^{(M)(m, m_2){\rm (I)}}_{k_1,k_2}\big)}(\tau,z)$

\vspace{-4mm}

$$ 
= \,\ - \, e^{\pi iz} \cdot {\rm ch}^{(+){\rm tw}}_{
\overset{N=2}{H}\big(\lambda^{(M)(m-1, m_2)}_{k_1,k_1+k_2}\big)}
(\tau,-z)
\cdot 
\frac{\prod\limits_{n=1}^{\infty}
(1+e^{2\pi iz}q^{n-1})(1+e^{-2\pi iz}q^{n})}
{\prod\limits_{n=1}^{\infty}(1-e^{4\pi iz}q^{n-1})(1-e^{-4\pi iz}q^{n})}
$$
\item[{\rm (ii)}] \, ${\rm ch}^{(-){\rm tw}}_{
\overset{N=4}{\ddot{H}}
\big(\Lambda^{(M)(m, m_2){\rm (I)}}_{k_1,k_2}\big)}(\tau,z)$

\vspace{-4mm}

$$
= \,\ -i \,\ e^{\pi iz} \cdot 
{\rm ch}^{(-){\rm tw}}_{
\overset{N=2}{H}\big(\lambda^{(M)(m-1, m_2)}_{k_1,k_1+k_2}\big)}
(\tau,-z)
\cdot 
\frac{\prod\limits_{n=1}^{\infty}
(1-e^{2\pi iz}q^{n-1})(1-e^{-2\pi iz}q^{n})}
{\prod\limits_{n=1}^{\infty}(1-e^{4\pi iz}q^{n-1})(1-e^{-4\pi iz}q^{n})}
$$
\end{enumerate}
\end{enumerate}
\end{cor}

\subsection{Expression of characters via string functions}
\label{subsec:n2n4:string-fn}


In this section, we consider the power series expansion of 
characters in the domain ${\rm Im}(\tau)>0$.
First, by Lemma \ref{n2n4:lemma:2024-119b}, we have the following:

\medskip


\begin{lemma} \quad 
\label{n2n4:lemma:2024-120a}
\begin{enumerate}
\item[{\rm 1)}] \,\ For $(k_1,k_2) \, \in \, \Omega^{\rm (I)}$,
\begin{subequations}
{\allowdisplaybreaks
\begin{eqnarray}
& & \hspace{-18mm}
\Psi^{[M,m,-m_2; \, \frac12]}_{1; \, k_1+\frac12, k_1+k_2+\frac12; \, \frac12}
(\tau, z, -z,0)
\,\ = \,\ 
(-1)^{m_2} \, e^{2\pi i(\frac{mk_2}{M}-m_2)z} 
\nonumber
\\[2mm]
& & \hspace{-15mm}
\times \,\ \bigg[
\sum_{\substack{\ell, \, n \, \in \, \zzz \\[1mm]
\ell, \, n \, \geq \, 0}}
- 
\sum_{\substack{\ell, \, n \, \in \, \zzz \\[1mm]
\ell, \, n \, < \, 0}} \bigg] \, 
(-1)^n \, e^{2\pi inz} \, 
q^{Mm(\ell+\frac{k_1+\frac12}{M})(\ell+\frac{k_1+k_2+\frac12}{M})} \, 
q^{(n-m_2)(M\ell+k_1+\frac12)}
\label{n2n4:eqn:2024-120a1}
\\[2mm]
& & \hspace{-18mm}
\Psi^{[M,m,-m_2; \, \frac12]}_{2; \, k_1+\frac12, k_1+k_2+\frac12; \, \frac12}
(\tau, z, -z,0)
\,\ = \,\ 
(-1)^{m_2} \, e^{2\pi i(\frac{mk_2}{M}-m_2)z} 
\nonumber
\\[2mm]
& & \hspace{-16mm}
\times \, \bigg[ \hspace{-2mm}
\sum_{\substack{\ell, \, n \, \in \, \zzz \\[1mm]
\ell \, > \, 0, \,\ n \, \geq \, 0}}
- 
\sum_{\substack{\ell, \, n \, \in \, \zzz \\[1mm]
\ell \, \leq \, 0, \,\ n \, < \, 0}} \hspace{-3mm} 
\bigg] 
(-1)^n  e^{2\pi inz} \, 
q^{Mm(\ell-\frac{k_1+\frac12}{M})(\ell-\frac{k_1+k_2+\frac12}{M})} \, 
q^{(n-m_2)(M\ell-(k_1+k_2+\frac12))}
\label{n2n4:eqn:2024-120a2}
\end{eqnarray}}
\end{subequations}
and 
\begin{subequations}
{\allowdisplaybreaks
\begin{eqnarray}
& & \hspace{-18mm}
\Psi^{[M,m,-m_2; \, \frac12]}_{1; \, k_1+1, k_1+k_2; \, 0}
(\tau, -z, z,0) \,\ = \,\ 
(-1)^{m_2} \, e^{2\pi i[-\frac{m}{M}(k_2-1)+m_2]z} 
\nonumber
\\[2mm]
& & \hspace{-15mm}
\times \,\ \bigg[
\sum_{\substack{\ell, \, n \, \in \, \zzz \\[1mm]
\ell, \, n \, \geq \, 0}}
- 
\sum_{\substack{\ell, \, n \, \in \, \zzz \\[1mm]
\ell, \, n \, < \, 0}} \bigg] \,\ 
(-1)^n \, e^{-2\pi inz} \, 
q^{Mm(\ell+\frac{k_1+1}{M})(\ell+\frac{k_1+k_2}{M})} \, 
q^{(n-m_2)(M\ell+k_1+1)}
\label{n2n4:eqn:2024-120b1}
\\[2mm]
& & \hspace{-18mm}
\Psi^{[M,m,-m_2; \, \frac12]}_{2; \, k_1+1, k_1+k_2; \, 0}
(\tau, -z, z,0) \,\ = \,\ 
(-1)^{m_2} \, e^{2\pi i[-\frac{m}{M}(k_2-1)+m_2]z} 
\nonumber
\\[2mm]
& & \hspace{-15mm}
\times \, \bigg[ \hspace{-2mm}
\sum_{\substack{\ell, \, n \, \in \, \zzz \\[1mm]
\ell \, > \, 0, \,\ n \, \geq \, 0}}
- 
\sum_{\substack{\ell, \, n \, \in \, \zzz \\[1mm]
\ell \, \leq \, 0 \,\ n \, < \, 0}} \hspace{-3mm}
\bigg] 
(-1)^n e^{-2\pi inz} \, 
q^{Mm(\ell-\frac{k_1+1}{M})(\ell-\frac{k_1+k_2}{M})} \, 
q^{(n-m_2)(M\ell-(k_1+k_2))}
\label{n2n4:eqn:2024-120b2}
\end{eqnarray}}
\end{subequations}
\item[{\rm 2)}] \,\ For $(k_1,k_2) \, \in \, \Omega^{\rm (III)}$,
\begin{subequations}
{\allowdisplaybreaks
\begin{eqnarray}
& & \hspace{-18mm}
\Psi^{[M,m,-m_2; \, \frac12]}_{1; \, k_1+\frac12, k_1+k_2+\frac12; \, \frac12}
(\tau, -z, z,0) \,\ = \,\ 
(-1)^{m_2} \, e^{2\pi i(-\frac{mk_2}{M}+m_2)z} 
\nonumber
\\[2mm]
& & \hspace{-15mm}
\times \,\ \bigg[
\sum_{\substack{\ell, \, n \, \in \, \zzz \\[1mm]
\ell, \, n \, \geq \, 0}}
- 
\sum_{\substack{\ell, \, n \, \in \, \zzz \\[1mm]
\ell, \, n \, < \, 0}} \bigg] \, 
(-1)^n \, e^{-2\pi inz} \, 
q^{Mm(\ell+\frac{k_1+\frac12}{M})(\ell+\frac{k_1+k_2+\frac12}{M})} \, 
q^{(n-m_2)(M\ell+k_1+\frac12)}
\label{n2n4:eqn:2024-120c1}
\\[2mm]
& & \hspace{-18mm}
\Psi^{[M,m,-m_2; \, \frac12]}_{2; \, k_1+\frac12, k_1+k_2+\frac12; \, \frac12}
(\tau, -z, z,0) \,\ = \,\ 
(-1)^{m_2} \, e^{2\pi i(-\frac{mk_2}{M}+m_2)z} 
\nonumber
\\[2mm]
& & \hspace{-16.8mm}
\times \, \bigg[\hspace{-2mm}
\sum_{\substack{\ell, \, n \, \in \, \zzz \\[1mm]
\ell \, > \, 0 ,\,\ n \, \geq \, 0}}
- 
\sum_{\substack{\ell, \, n \, \in \, \zzz \\[1mm]
\ell \, \leq \, 0, \,\ n \, < \, 0}} \hspace{-3mm}
\bigg] 
(-1)^n e^{-2\pi inz} \, 
q^{Mm(\ell-\frac{k_1+\frac12}{M})(\ell-\frac{k_1+k_2+\frac12}{M})} \, 
q^{(n-m_2)(M\ell-(k_1+k_2+\frac12))} 
\label{n2n4:eqn:2024-120c2}
\end{eqnarray}}
\end{subequations}
and
\begin{subequations}
{\allowdisplaybreaks
\begin{eqnarray}
& & \hspace{-18mm}
\Psi^{[M,m,-m_2; \, \frac12]}_{1; \, k_1, k_1+k_2+1; \, 0}
(\tau, z, -z,0) \,\ = \,\ 
(-1)^{m_2} \, e^{2\pi i[\frac{m}{M}(k_2+1)-m_2]z} 
\nonumber
\\[2mm]
& & \hspace{-15mm}
\times \,\ \bigg[
\sum_{\substack{\ell, \, n \, \in \, \zzz \\[1mm]
\ell, \, n \, \geq \, 0}}
- 
\sum_{\substack{\ell, \, n \, \in \, \zzz \\[1mm]
\ell, \, n \, < \, 0}} \bigg] \,\ 
(-1)^n \, e^{2\pi inz} \, 
q^{Mm(\ell+\frac{k_1}{M})(\ell+\frac{k_1+k_2*1}{M})} \, 
q^{(n-m_2)(M\ell+k_1)}
\label{n2n4:eqn:2024-120d1}
\\[2mm]
& & \hspace{-18mm}
\Psi^{[M,m,-m_2; \, \frac12]}_{2; \, k_1, k_1+k_2+1; \, 0}
(\tau, z, -z,0) \,\ = \,\ 
(-1)^{m_2} \, e^{2\pi i[\frac{m}{M}(k_2+1)-m_2]z} 
\nonumber
\\[2mm]
& & \hspace{-15mm}
\times \,\ \bigg[ \hspace{-2mm}
\sum_{\substack{\ell, \, n \, \in \, \zzz \\[1mm]
\ell \, > \, 0, \,\ n \, \geq \, 0}}
- 
\sum_{\substack{\ell, \, n \, \in \, \zzz \\[1mm]
\ell \, \leq \, 0, \,\ n \, < \, 0}} \hspace{-3mm}
\bigg] \, 
(-1)^n \, e^{2\pi inz} \, 
q^{Mm(\ell-\frac{k_1}{M})(\ell-\frac{k_1+k_2+1}{M})} \, 
q^{(n-m_2)(M\ell-(k_1+k_2+1))}
\label{n2n4:eqn:2024-120d2}
\end{eqnarray}}
\end{subequations}
\end{enumerate}
\end{lemma}

\medskip

From this Lemma \ref{n2n4:lemma:2024-120a}, we see the leading
term in each function as follows:

\medskip

\begin{lemma} \,\ 
\label{n2n4:lemma:2024-120b}
\begin{enumerate}
\item[{\rm 1)}] \quad For \,\ $(k_1,k_2) \, \in \, \Omega^{\rm (I)}$,
{\allowdisplaybreaks
\begin{eqnarray*}
& & \hspace{-10mm}
\text{the leading term in} \,\
\Psi^{[M,m,-m_2; \, \frac12]}_{k_1+\frac12, k_1+k_2+\frac12; \, \frac12}
(\tau, z, -z,0)
\\[2.5mm]
&=&
(-1)^{m_2} \, e^{2\pi i(\frac{mk_2}{M}-m_2)z} \, 
q^{\frac{m}{M}(k_1+\frac12)(k_1+k_2+\frac12)-m_2(k_1+\frac12)}
\\[2mm]
& & \hspace{-10mm}
\text{the leading term in} \,\
\Psi^{[M,m,-m_2; \, \frac12]}_{k_1+1, k_1+k_2; \, 0}(\tau, -z, z,0)
\\[2.5mm]
&=& \left\{
\begin{array}{rcl}
\dfrac{(-1)^{m_2}}{1+e^{-2\pi iz}} \, 
e^{2\pi i[-\frac{m}{M}(k_2-1)+m_2]z} \, 
q^{\frac{m}{M}(k_1+1)(k_1+k_2)-m_2(k_1+1)}
& &{\rm if} \,\ (k_1,k_2)=(0,0)
\\[4mm]
(-1)^{m_2} \, e^{2\pi i[-\frac{m}{M}(k_2-1)+m_2]z} \, 
q^{\frac{m}{M}(k_1+1)(k_1+k_2)-m_2(k_1+1)}
& &{\rm if} \,\ (k_1,k_2) \ne (0,0)
\end{array} \right.
\end{eqnarray*}}
\item[{\rm 2)}] \quad For \,\ $(k_1,k_2) \, \in \, \Omega^{\rm (III)}$,
{\allowdisplaybreaks
\begin{eqnarray*}
& & \hspace{-20mm}
\text{the leading term in} \,\
\Psi^{[M,m,-m_2; \, \frac12]}_{k_1+\frac12, k_1+k_2+\frac12; \, \frac12}
(\tau, -z, z,0) 
\\[2.5mm]
&=& 
(-1)^{m_2} \, e^{2\pi i(-\frac{mk_2}{M}+m_2)z} \, 
q^{\frac{m}{M}(k_1+\frac12)(k_1+k_2+\frac12)-m_2(k_1+\frac12)}
\\[2mm]
& & \hspace{-20mm}
\text{the leading term in} \,\
\Psi^{[M,m,-m_2; \, \frac12]}_{k_1, k_1+k_2+1; \, 0}(\tau, z, -z,0) 
\\[2mm]
&=&
\left\{
\begin{array}{rcl}
\dfrac{(-1)^{m_2}}{1+e^{-2\pi iz}} \,\ e^{2\pi i[\frac{m}{M}(k_2+1)-m_2]z} \, 
q^{\frac{m}{M}k_1(k_1+k_2+1)-m_2k_1}
& & {\rm if} \,\ k_1 =0
\\[4mm]
(-1)^{m_2} \,\ e^{2\pi i[\frac{m}{M}(k_2+1)-m_2]z} \, 
q^{\frac{m}{M}k_1(k_1+k_2+1)-m_2k_1}
& & {\rm if} \,\ k_1 \ne 0
\end{array} \right.
\end{eqnarray*}}
\end{enumerate}
\end{lemma}

\medskip

By Note \ref{note:2024-204c}, one has 
\begin{subequations}
\begin{equation}
\dfrac{1}{\eta(\tau)^3 \, \vartheta_{11}(\tau,2z)} 
\,\ = \,\ 
\frac{- \, i}{\eta(\tau)^6}
\bigg[
\sum_{\substack{j, \, k \, \in \, \zzz \\[1mm]
j, \, k \, \geq 0}}
-
\sum_{\substack{j, \, k \, \in \, \zzz \\[1mm]
j, \, k \, < 0}} \bigg] \,\ 
(-1)^j \, e^{-2\pi i(2k+1)z} \, q^{\frac12 j(j+1)+jk}
\label{n2n4:eqn:2024-120e1}
\end{equation}
so
\begin{equation}
\text{the leading term in} \,\ 
\frac{1}{\eta(\tau)^3 \, \vartheta_{11}(\tau,2z)} 
\,\ = \,\ - \, i \, e^{-2\pi iz} \, q^{-\frac14}
\label{n2n4:eqn:2024-120e2}
\end{equation}
\end{subequations}
And, by Note \ref{note:2024-204b}, one has
\begin{equation}
\begin{array}{clcl}
\text{the leading term in} &\vartheta_{00}(\tau,z)^2 &=& 1
\\[2mm]
\text{the leading term in} &\vartheta_{10}(\tau,z)^2 &=&
e^{2\pi iz}(1+e^{-2\pi iz})^2 \, q^{\frac14}
\end{array}
\label{n2n4:eqn:2024-120f}
\end{equation}
Then, by \eqref{n2n4:eqn:2024-120e2} and \eqref{n2n4:eqn:2024-120f},
one has
\begin{subequations}
{\allowdisplaybreaks
\begin{eqnarray}
& &
\text{the leading term in} \,\ 
\frac{\vartheta_{00}(\tau,z)^2}{\eta(\tau)^3 \vartheta_{11}(\tau, 2z)}
\,\ = \,\ 
-i \, e^{-2\pi iz}q^{-\frac14}
\label{n2n4:eqn:2024-120g1}
\\[2mm]
& &
\text{the leading term in} \,\ 
\frac{\vartheta_{00}(\tau,z)^2}{\eta(\tau)^3 \vartheta_{11}(\tau, 2z)}
\,\ = \,\ 
-i \, (1+e^{-2\pi iz})^2
\label{n2n4:eqn:2024-120g2}
\end{eqnarray}}
\end{subequations}

Then by Proposition \ref{n2n4:prop:2024-110a} and Lemma 
\ref{n2n4:lemma:2024-120b} and the formulas 
\eqref{n2n4:eqn:2024-120g1} and \eqref{n2n4:eqn:2024-120g2},
we obtain the leading terms of the characters of 
$\overset{N=4}{\ddot{H}}\big(\Lambda^{(M)(m, m_2)(\heartsuit)}_{k_1,k_2})$
as follows:

\medskip

\begin{lemma} \,\ 
\label{lemma:2024-205a}
\begin{enumerate}
\item[{\rm 1)}] \,\ For non-twisted characters :
\begin{enumerate}
\item[{\rm (I)}] the leading term in \,\
${\rm ch}^{(+)}_{\overset{N=4}{\ddot{H}}
\big(\Lambda^{(M)(m, m_2){\rm (I)}}_{k_1,k_2}\big)}(\tau,z)$ 
{\allowdisplaybreaks
\begin{eqnarray*}
&=&
e^{2\pi i(\frac{mk_2}{M}-m_2-1)z} \, 
q^{\frac{m}{M}(k_1+\frac12)(k_1+k_2+\frac12)-m_2(k_1+\frac12)-\frac14}
\\[2mm]
&=&
e^{2\pi i \, s^{(M)(m, m_2+1){\rm (I)}}_{k_1,k_2}z} \, 
q^{h^{(M)(m, m_2+1){\rm (I)}}_{k_1,k_2}-\frac{c}{24}}
\end{eqnarray*}}
\item[{\rm (III)}] the leading term in \,\ 
${\rm ch}^{(+)}_{\overset{N=4}{\ddot{H}}
\big(\Lambda^{(M)(m, m_2){\rm (I)}}_{k_1,k_2}\big)}(\tau,z)$  
{\allowdisplaybreaks
\begin{eqnarray*}
&=&
e^{2\pi i(-\frac{mk_2}{M}+m_2-1)z} \, 
q^{\frac{m}{M}(k_1+\frac12)(k_1+k_2+\frac12)-m_2(k_1+\frac12)-\frac14}
\\[2mm]
&=&
e^{2\pi i \, s^{(M)(m, m_2+1){\rm (III)}}_{k_1,k_2}z} \, 
q^{h^{(M)(m, m_2+1){\rm (III)}}_{k_1,k_2}-\frac{c}{24}}
\end{eqnarray*}}
\end{enumerate}
\item[{\rm 2)}] \,\ For twisted characters :
\begin{enumerate}
\item[{\rm (I)}] the leading term in \,\ 
${\rm ch}^{(+) {\rm tw}}_{\overset{N=4}{\ddot{H}}
\big(\Lambda^{(M)(m, m_2){\rm (I)}}_{k_1,k_2}\big)}(\tau,z)$ \,\ 
$$
= \, \left\{
\begin{array}{ll}
(1+e^{-2\pi iz}) \,\ e^{2\pi i[-\frac{m}{M}(k_2-1)+m_2]z} \, 
q^{\frac{m}{M}(k_1+1)(k_1+k_2)-m_2(k_1+1)}
&{\rm if} \,\ (k_1, k_2) =(0,0)
\\[3mm]
(1+e^{-2\pi iz})^2 \, 
\underbrace{e^{2\pi i[-\frac{m}{M}(k_2-1)+m_2]z} \, 
q^{\frac{m}{M}(k_1+1)(k_1+k_2)-m_2(k_1+1)}}_{
\substack{|| \\[0mm] {\displaystyle 
e^{2\pi i \, s^{(M)(m, m_2+1){\rm (I) \, tw}}_{k_1,k_2}z} \, 
q^{h^{(M)(m, m_2+1){\rm (I) \, tw}}_{k_1,k_2}-\frac{c}{24}}
}}}
&{\rm if} \,\ (k_1, k_2) \ne (0,0)
\end{array}\right.
$$
\item[{\rm (III)}] the leading term in \,\ 
${\rm ch}^{(+) {\rm tw}}_{\overset{N=4}{\ddot{H}}
\big(\Lambda^{(M)(m, m_2){\rm (III)}}_{k_1,k_2}\big)}(\tau,z)$ \,\ 
$$
= \,\ \left\{
\begin{array}{rcl}
(1+e^{-2\pi iz}) \,\ e^{2\pi i[\frac{m}{M}(k_2+1)-m_2]z} \, 
q^{\frac{m}{M}k_1(k_1+k_2+1)-m_2k_1}
& & {\rm if} \,\ k_1=0
\\[3mm]
(1+e^{-2\pi iz})^2 \, 
\underbrace{e^{2\pi i[\frac{m}{M}(k_2+1)-m_2]z} \, 
q^{\frac{m}{M}k_1(k_1+k_2+1)-m_2k_1}}_{
\substack{|| \\[0.5mm] {\displaystyle \hspace{-7mm}
e^{2\pi i \, s^{(M)(m, m_2+1){\rm (III) \, tw}}_{k_1,k_2}z} \, 
q^{h^{(M)(m, m_2+1){\rm (III) \, tw}}_{k_1,k_2}-\frac{c}{24}}
}}}
& & {\rm if} \,\ k_1 \ne 0
\end{array} \right.
$$
\end{enumerate}
\end{enumerate}
\end{lemma}

\medskip

For the characters of $\overset{N=4}{\ddot{H}}
\big(\Lambda^{(M)(m, m_2)(\heartsuit)}_{k_1,k_2}\big)$ given in
Proposition \ref{n2n4:prop:2024-110a}, 
the power series expression of these characters are 
obtained by easy calculation using Lemma \ref{n2n4:lemma:2024-120a} 
and Note \ref{note:2024-204b} and \eqref{n2n4:eqn:2024-120e1}
as follows:

\medskip

\begin{prop} \,\ 
\label{n2n4:prop:2024-201a}
\begin{enumerate}
\item[{\rm 1)}] \,\ For non-twisted characters :
\begin{enumerate}
\item[{\rm (I)}] \,\ ${\rm ch}^{(+)}_{
\overset{N=4}{\ddot{H}}
\big(\Lambda^{(M)(m, m_2){\rm (I)}}_{k_1,k_2}\big)}(\tau,z)
\,\ = \,\ 
\underbrace{e^{2\pi i(\frac{mk_2}{M}-m_2-1)z}}_{\substack{|| \\[0mm] 
{\displaystyle 
e^{2\pi i \, s^{(M)(m, m_2+1){\rm (I)}}_{k_1,k_2}z}
}}} \, 
\sum\limits_{n \in \zzz} \, e^{2\pi inz} \, \times \, \Bigg\{ $
{\allowdisplaybreaks
\begin{eqnarray*}
& & \hspace{-10mm}
\Bigg[\frac{\eta(2\tau)^5}{\eta(\tau)^8\eta(4\tau)^2}
\bigg[
\sum_{\substack{\ell_1, \, n_1 \, \in \, \zzz \\[1mm]
n_1 \, \equiv \, n \,\ {\rm mod} \, 2 \\[1mm]
\ell_1, \, n_1 \, \geq \, 0}}
- 
\sum_{\substack{\ell_1, \, n_1 \, \in \, \zzz \\[1mm]
n_1 \, \equiv \, n \,\ {\rm mod} \, 2 \\[1mm]
\ell_1, \, n_1 \, < \, 0}} \bigg] 
+ 2 \, 
\frac{\eta(4\tau)^2}{\eta(\tau)^6\eta(2\tau)} 
\bigg[
\sum_{\substack{\ell_1, \, n_1 \, \in \, \zzz \\[1mm]
n_1 \, \not\equiv \, n \,\ {\rm mod} \, 2 \\[1mm]
\ell_1, \, n_1 \, \geq \, 0}}
- 
\sum_{\substack{\ell_1, \, n_1 \, \in \, \zzz \\[1mm]
n_1 \, \not\equiv \, n \,\ {\rm mod} \, 2 \\[1mm]
\ell_1, \, n_1 \, < \, 0}} \bigg] 
\Bigg]
\\[2mm]
& &
\times \,\ \bigg[
\sum_{\substack{\ell_2, \, n_2 \, \in \, \zzz \\[1mm]
\ell_2, \, n_2 \, \geq 0}}
-
\sum_{\substack{\ell_2, \, n_2 \, \in \, \zzz \\[1mm]
\ell_2, \, n_2 \, < 0}} \bigg] \, 
(-1)^{n_1+\ell_2} \, 
q^{\frac14(n-n_1+2n_2)^2} \, q^{\frac12 \ell_2(\ell_2+1)+\ell_2n_2}
\\[2mm]
& & \hspace{5mm}
\times \,\ \bigg(
q^{Mm(\ell_1+\frac{k_1+\frac12}{M})(\ell_1+\frac{k_1+k_2+\frac12}{M})} \, 
q^{(n_1-m_2)(M\ell_1+k_1+\frac12)} 
\\[2mm]
& & \hspace{6.5mm}
- \,\ 
q^{Mm(\ell_1-\frac{k_1+\frac12}{M})(\ell_1-\frac{k_1+k_2+\frac12}{M})} \, 
q^{(n_1-m_2)(M\ell_1-(k_1+k_2+\frac12))} 
\bigg) \Bigg\}
\end{eqnarray*}}
\item[{\rm (III)}] \,\ ${\rm ch}^{(+)}_{
\overset{N=4}{\ddot{H}}
\big(\Lambda^{(M)(m, m_2){\rm (III)}}_{k_1,k_2}\big)}(\tau,z)
\,\ = \,\ 
\underbrace{e^{2\pi i(-\frac{mk_2}{M}+m_2+1)z}}_{\substack{|| \\[0mm] 
{\displaystyle 
e^{2\pi i \, s^{(M)(m, m_2+1){\rm (III)}}_{k_1,k_2}z}
}}} \, 
\sum\limits_{n \in \zzz} \, e^{2\pi inz} \, \times \, \Bigg\{ $
{\allowdisplaybreaks
\begin{eqnarray*}
& & \hspace{-10mm}
\Bigg[\frac{\eta(2\tau)^5}{\eta(\tau)^8\eta(4\tau)^2}
\bigg[
\sum_{\substack{\ell_1, \, n_1 \, \in \, \zzz \\[1mm]
n_1 \, \equiv \, n \,\ {\rm mod} \, 2 \\[1mm]
\ell_1, \, n_1 \, \geq \, 0}}
- 
\sum_{\substack{\ell_1, \, n_1 \, \in \, \zzz \\[1mm]
n_1 \, \equiv \, n \,\ {\rm mod} \, 2 \\[1mm]
\ell_1, \, n_1 \, < \, 0}} \bigg] 
+ 2 \, 
\frac{\eta(4\tau)^2}{\eta(\tau)^6\eta(2\tau)} 
\bigg[
\sum_{\substack{\ell_1, \, n_1 \, \in \, \zzz \\[1mm]
n_1 \, \not\equiv \, n \,\ {\rm mod} \, 2 \\[1mm]
\ell_1, \, n_1 \, \geq \, 0}}
- 
\sum_{\substack{\ell_1, \, n_1 \, \in \, \zzz \\[1mm]
n_1 \, \not\equiv \, n \,\ {\rm mod} \, 2 \\[1mm]
\ell_1, \, n_1 \, < \, 0}} \bigg] 
\Bigg]
\\[2mm]
& &
\times \,\ \bigg[
\sum_{\substack{\ell_2, \, n_2 \, \in \, \zzz \\[1mm]
\ell_2, \, n_2 \, \geq 0}}
-
\sum_{\substack{\ell_2, \, n_2 \, \in \, \zzz \\[1mm]
\ell_2, \, n_2 \, < 0}} \bigg] \, 
(-1)^{n_1+\ell_2} \, 
q^{\frac14(n+n_1+2n_2)^2} \, q^{\frac12 \ell_2(\ell_2+1)+\ell_2n_2}
\\[2mm]
& & \hspace{5mm}
\times \,\ \bigg(
q^{Mm(\ell_1+\frac{k_1+\frac12}{M})(\ell_1+\frac{k_1+k_2+\frac12}{M})} \, 
q^{(n_1-m_2)(M\ell_1+k_1+\frac12)} 
\\[2mm]
& & \hspace{6.5mm}
- \,\ 
q^{Mm(\ell_1-\frac{k_1+\frac12}{M})(\ell_1-\frac{k_1+k_2+\frac12}{M})} \, 
q^{(n_1-m_2)(M\ell_1-(k_1+k_2+\frac12))} 
\bigg) \Bigg\}
\end{eqnarray*}}
\end{enumerate}
\item[{\rm 2)}] \,\ For twisted characters :
\begin{enumerate}
\item[{\rm (I)}] \,\ ${\rm ch}^{(+){\rm tw}}_{
\overset{N=4}{\ddot{H}}
\big(\Lambda^{(M)(m, m_2){\rm (I)}}_{k_1,k_2}\big)}(\tau,z)
\,\ = \,\ 
\underbrace{e^{2\pi i[-\frac{m}{M}(k_2-1)+m_2]z}}_{\substack{|| \\[0mm] 
{\displaystyle 
e^{2\pi i \, s^{(M)(m, m_2+1){\rm (I) \, tw}}_{k_1,k_2}z}
}}} \,  \, 
\sum\limits_{n \in \zzz} \, e^{2\pi inz} \, \times \, \Bigg\{ $
{\allowdisplaybreaks
\begin{eqnarray*}
& & \hspace{-15mm}
\Bigg[
\frac{\eta(2\tau)^5}{\eta(\tau)^8\eta(4\tau)^2}
\bigg[
\sum_{\substack{\ell_1, \, n_1 \, \in \, \zzz \\[1mm]
n_1 \, \equiv \, n \,\ {\rm mod} \, 2 \\[1mm]
\ell_1, \, n_1 \, \geq \, 0}}
- 
\sum_{\substack{\ell_1, \, n_1 \, \in \, \zzz \\[1mm]
n_1 \, \equiv \, n \,\ {\rm mod} \, 2 \\[1mm]
\ell_1, \, n_1 \, < \, 0}} \bigg] 
+
2 \, \frac{\eta(4\tau)^2}{\eta(\tau)^6 \, \eta(2\tau)}
\bigg[
\sum_{\substack{\ell_1, \, n_1 \, \in \, \zzz \\[1mm]
n_1 \, \not\equiv \, n \,\ {\rm mod} \, 2 \\[1mm]
\ell_1, \, n_1 \, \geq \, 0}}
- 
\sum_{\substack{\ell_1, \, n_1 \, \in \, \zzz \\[1mm]
n_1 \, \not\equiv \, n \,\ {\rm mod} \, 2 \\[1mm]
\ell_1, \, n_1 \, < \, 0}} \bigg] 
\Bigg]
\\[2mm]
& & \hspace{-10mm}
\times \,\ 
\bigg[
\sum_{\substack{\ell_2, \, n_2 \, \in \, \zzz \\[1mm]
\ell_2, \, n_2 \, \geq 0}}
-
\sum_{\substack{\ell_2, \, n_2 \, \in \, \zzz \\[1mm]
\ell_2, \, n_2 \, < 0}} \bigg] \,\ (-1)^{n_1+\ell_2} \, 
\\[2mm]
& & \hspace{-10mm}
\times \,\ \Big(
q^{\frac14(n+n_1+2n_2+1)^2} \, 
q^{Mm(\ell_1+\frac{k_1+1}{M})(\ell_1+\frac{k_1+k_2}{M})} \, 
q^{(n_1-m_2)(M\ell_1+k_1+1)} \, 
q^{\frac12 \ell_2(\ell_2+1)+\ell_2n_2}
\\[2mm]
& & \hspace{-8mm}
- \,\ 
q^{\frac14(n+n_1+2n_2+1)^2} \, 
q^{Mm(\ell_1-\frac{k_1+1}{M})(\ell_1-\frac{k_1+k_2}{M})} \, 
q^{(n_1-m_2)(M\ell_1-(k_1+k_2))} \, 
q^{\frac12 \ell_2(\ell_2+1)+\ell_2n_2} \Big)
\Bigg\}
\end{eqnarray*}}
\item[{\rm (III)}] \,\ ${\rm ch}^{(+){\rm tw}}_{
\overset{N=4}{\ddot{H}}
\big(\Lambda^{(M)(m, m_2){\rm (III)}}_{k_1,k_2}\big)}(\tau,z)
\,\ = \,\ 
\underbrace{e^{2\pi i[\frac{m}{M}(k_2+1)-m_2]z}}_{\substack{|| \\[0mm] 
{\displaystyle 
e^{2\pi i \, s^{(M)(m, m_2+1){\rm (III) \, tw}}_{k_1,k_2}z}
}}} \, 
\sum\limits_{n \in \zzz} \, e^{2\pi inz} \, \times \Bigg\{ $
{\allowdisplaybreaks
\begin{eqnarray*}
& & \hspace{-12mm}
\Bigg[
\frac{\eta(2\tau)^5}{\eta(\tau)^8\eta(4\tau)^2} \, 
\bigg[
\sum_{\substack{\ell_1, \, n_1 \, \in \, \zzz \\[1mm]
n_1 \, \equiv \, n \,\ {\rm mod} \, 2 \\[1mm]
\ell_1, \, n_1 \, \geq \, 0}}
- 
\sum_{\substack{\ell_1, \, n_1 \, \in \, \zzz \\[1mm]
n_1 \, \equiv \, n \,\ {\rm mod} \, 2 \\[1mm]
\ell_1, \, n_1 \, < \, 0}} \bigg]
+
2 \, \frac{\eta(4\tau)^2}{\eta(\tau)^6 \, \eta(2\tau)}
\bigg[
\sum_{\substack{\ell_1, \, n_1 \, \in \, \zzz \\[1mm]
n_1 \, \not\equiv \, n \,\ {\rm mod} \, 2 \\[1mm]
\ell_1, \, n_1 \, \geq \, 0}}
- 
\sum_{\substack{\ell_1, \, n_1 \, \in \, \zzz \\[1mm]
n_1 \, \not\equiv \, n \,\ {\rm mod} \, 2 \\[1mm]
\ell_1, \, n_1 \, < \, 0}} \bigg]
\Bigg]
\\[2mm]
& & \hspace{-10mm}
\times \,\ 
\bigg[
\sum_{\substack{\ell_2, \, n_2 \, \in \, \zzz \\[1mm]
\ell_2, \, n_2 \, \geq 0}}
-
\sum_{\substack{\ell_2, \, n_2 \, \in \, \zzz \\[1mm]
\ell_2, \, n_2 \, < 0}} \bigg] \,\ (-1)^{n_1+\ell_2}
\\[2mm]
& & \hspace{-10mm}
\times \,\ \Big(
q^{\frac12(n+n_1+2n_2+1)^2} \, 
q^{Mm(\ell_1+\frac{k_1}{M})(\ell_1+\frac{k_1+k_2*1}{M})} \, 
q^{(n_1-m_2)(M\ell_1+k_1)} \, 
q^{\frac12 \ell_2(\ell_2+1)+\ell_2n_2}
\\[2mm]
& & \hspace{-8mm}
- \,\ 
q^{\frac12(n+n_1+2n_2+1)^2} \, 
q^{Mm(\ell_1-\frac{k_1}{M})(\ell_1-\frac{k_1+k_2+1}{M})} \, 
q^{(n_1-m_2)(M\ell_1-(k_1+k_2+1))} \, 
q^{\frac12 \ell_2(\ell_2+1)+\ell_2n_2} \Big)
\Bigg\}
\end{eqnarray*}}
\end{enumerate}
\end{enumerate}
\end{prop}

\subsection{The case $m=1$}
\label{subsec:n2n4:m=1}


In this section we consider the case $(m,m_2)=(1,0)$ and $M \in \nnn$. 
In this case Proposition \ref{n2n4:prop:2024-110a}, rewritten by Note 
\ref{n2n4:note:2023-1231b}, gives the following character formulas:

\medskip

\begin{prop} \,\
\label{n2n4:prop:2024-121a}
\begin{enumerate}
\item[{\rm 1)}]
\begin{enumerate}
\item[{\rm (I)}] \, ${\rm ch}^{(+)}_{
\overset{N=4}{\ddot{H}}
\big(\Lambda^{(M)(1, 0){\rm (I)}}_{k_1,k_2}\big)}(\tau,z)
\, = \, 
i \, 
\Psi^{[M,1,0; \frac12]}_{k_1+\frac12, \, k_1+k_2+\frac12; \, \frac12}
(\tau, z, -z,0)
\cdot
\dfrac{\vartheta_{00}(\tau,z)^2}{\eta(\tau)^3 \, \vartheta_{11}(\tau,2z)}$
\item[{\rm (III)}] ${\rm ch}^{(+)}_{
\overset{N=4}{\ddot{H}}
\big(\Lambda^{(M)(1, 0){\rm (III)}}_{k_1,k_2}\big)}(\tau,z)
= 
i \, 
\Psi^{[M,1,0; \, \frac12]}_{k_1+k_2+\frac12, \, k_1+\frac12; \, \frac12}
(\tau, z, -z,0)
\cdot
\dfrac{\vartheta_{00}(\tau,z)^2}{\eta(\tau)^3 \vartheta_{11}(\tau,2z)}$
\end{enumerate}
\item[{\rm 2)}]
\begin{enumerate}
\item[{\rm (I)}] \, ${\rm ch}^{(-)}_{
\overset{N=4}{\ddot{H}}
\big(\Lambda^{(M)(1, 0){\rm (I)}}_{k_1,k_2}\big)}(\tau,z)
\, = \, 
i \,\ 
\Psi^{[M,1,0; \, 0]}_{k_1+\frac12, \, k_1+k_2+\frac12; \, \frac12}
(\tau, z, -z,0)
\cdot
\dfrac{\vartheta_{01}(\tau,z)^2}{\eta(\tau)^3 \, \vartheta_{11}(\tau,2z)}$
\item[{\rm (III)}] \, ${\rm ch}^{(-)}_{
\overset{N=4}{\ddot{H}}
\big(\Lambda^{(M)(1, 0){\rm (III)}}_{k_1,k_2}\big)}(\tau,z)
\, = \, 
i \,\ 
\Psi^{[M,1,0; \, 0]}_{k_1+k_2+\frac12, \, k_1+\frac12; \, \frac12}
(\tau, z, -z,0)
\cdot
\dfrac{\vartheta_{01}(\tau,z)^2}{\eta(\tau)^3 \, \vartheta_{11}(\tau,2z)}$
\end{enumerate}
\item[$1)^{\rm tw}$]
\begin{enumerate}
\item[{\rm (I)}] \, ${\rm ch}^{(+){\rm tw}}_{
\overset{N=4}{\ddot{H}}
\big(\Lambda^{(M)(1, 0){\rm (I)}}_{k_1,k_2}\big)}(\tau,z)
\, = \, 
i \, 
\Psi^{[M,1,0; \frac12]}_{k_1+k_2, \, k_1+1; \, 0}
(\tau, z, -z,0)
\cdot
\dfrac{\vartheta_{10}(\tau,z)^2}{\eta(\tau)^3 \, \vartheta_{11}(\tau,2z)}$
\item[{\rm (III)}] \, ${\rm ch}^{(+){\rm tw}}_{
\overset{N=4}{\ddot{H}}
\big(\Lambda^{(M)(1, 0){\rm (III)}}_{k_1,k_2}\big)}(\tau,z)
\, = \, 
i \, 
\Psi^{[M,1,0; \frac12]}_{k_1, \, k_1+k_2+1; \, 0}
(\tau, z, -z,0)
\cdot
\dfrac{\vartheta_{10}(\tau,z)^2}{\eta(\tau)^3 \, \vartheta_{11}(\tau,2z)}$
\end{enumerate}
\item[$2)^{\rm tw}$]
\begin{enumerate}
\item[{\rm (I)}] \, ${\rm ch}^{(-){\rm tw}}_{
\overset{N=4}{\ddot{H}}
\big(\Lambda^{(M)(1, 0){\rm (I)}}_{k_1,k_2}\big)}(\tau,z)
\, = \, 
i \, 
\Psi^{[M,1,0; 0]}_{k_1+k_2, \, k_1+1; \, 0}
(\tau, z, -z,0)
\cdot
\dfrac{\vartheta_{11}(\tau,z)^2}{\eta(\tau)^3 \, \vartheta_{11}(\tau,2z)}$
\item[{\rm (III)}] \, ${\rm ch}^{(-){\rm tw}}_{
\overset{N=4}{\ddot{H}}
\big(\Lambda^{(M)(1, 0){\rm (III)}}_{k_1,k_2}\big)}(\tau,z)
\, = \, 
i \, 
\Psi^{[M,1,0; 0]}_{k_1, \, k_1+k_2+1; \, 0}
(\tau, z, -z,0)
\cdot
\dfrac{\vartheta_{11}(\tau,z)^2}{\eta(\tau)^3 \, \vartheta_{11}(\tau,2z)}$
\end{enumerate}
\end{enumerate}
\end{prop}

\medskip

The modular transformation of these characters is obtained from 
Lemma \ref{n2n4:lemma:2023-1231a} as follows:

\medskip

\begin{prop} 
\label{n2n4:prop:2024-121b}
The $S$-transformation of these characters is as follows:
\begin{enumerate}
\item[]
\begin{enumerate}
\item[{\rm 1)} \,\ ]
\begin{enumerate}
\item[{\rm (I)}] \,\ ${\rm ch}^{(+)}_{\overset{N=4}{\ddot{H}}
\big(\Lambda^{(M)(1, 0){\rm (I)}}_{k_1,k_2}\big)}
\Big(-\dfrac{1}{\tau}, \dfrac{z}{\tau} \Big)
\,\ = \,\ 
\dfrac{- 1}{M} \, e^{-\frac{2\pi i}{\tau}(1+\frac{1}{M})z^2} $
{\allowdisplaybreaks
\begin{eqnarray*}
& & \hspace{-8mm}
\times \,\ 
\bigg\{
\sum_{(j_1,j_2) \, \in \, \Omega^{\rm (I)}} 
e^{-\frac{2\pi i}{M}[(k_1+\frac12)(j_1+\frac12)+
(k_1+k_2+\frac12)(j_1+j_2+\frac12)]} \,\ 
{\rm ch}^{(+)}_{\overset{N=4}{\ddot{H}}
\big(\Lambda^{(M)(1, 0){\rm (I)}}_{j_1,j_2}\big)}(\tau,z)
\\[2mm]
& & \hspace{-5mm}
+ \sum_{(j_1,j_2) \, \in \, \Omega^{\rm (III)}} 
e^{-\frac{2\pi i}{M}[(k_1+\frac12)(j_1+j_2+\frac12)+
(k_1+k_2+\frac12)(j_1+\frac12)]} \,\ 
{\rm ch}^{(+)}_{\overset{N=4}{\ddot{H}}
\big(\Lambda^{(M)(1, 0){\rm (III)}}_{j_1,j_2}\big)} (\tau,z)
\bigg\}
\end{eqnarray*}}
\item[{\rm (III)}] \,\ ${\rm ch}^{(+)}_{\overset{N=4}{\ddot{H}}
\big(\Lambda^{(M)(1, 0){\rm (III)}}_{k_1,k_2}\big)}
\Big(-\dfrac{1}{\tau}, \dfrac{z}{\tau} \Big)
\,\ = \,\ 
\dfrac{- 1}{M} \, e^{-\frac{2\pi i}{\tau}(1+\frac{1}{M})z^2} $
{\allowdisplaybreaks
\begin{eqnarray*}
& & \hspace{-8mm}
\times \,\ \bigg\{
\sum_{(j_1,j_2) \, \in \, \Omega^{\rm (I)}} 
e^{-\frac{2\pi i}{M}[(k_1+k_2+\frac12)(j_1+\frac12)+
(k_1+\frac12)(j_1+j_2+\frac12)]} \,\ 
{\rm ch}^{(+)}_{\overset{N=4}{\ddot{H}}
\big(\Lambda^{(M)(1, 0){\rm (I)}}_{j_1,j_2}\big)}
\\[2mm]
& &
+ \sum_{(j_1,j_2) \, \in \, \Omega^{\rm (III)}} 
e^{-\frac{2\pi i}{M}[(k_1+k_2+\frac12)(j_1+j_2+\frac12)+
(k_1+\frac12)(j_1+\frac12)]} \,\ 
{\rm ch}^{(+)}_{\overset{N=4}{\ddot{H}}
\big(\Lambda^{(M)(1, 0){\rm (III)}}_{j_1,j_2}\big)} \bigg\}
\end{eqnarray*}}
\end{enumerate}
\item[{\rm 2)} \,\ ]
\begin{enumerate}
\item[{\rm (I)}] \,\ ${\rm ch}^{(-)}_{\overset{N=4}{\ddot{H}}
\big(\Lambda^{(M)(1, 0){\rm (I)}}_{k_1,k_2}\big)}
\Big(-\dfrac{1}{\tau}, \dfrac{z}{\tau} \Big)
\,\ = \,\ 
\dfrac{- 1}{M} \, e^{-\frac{2\pi i}{\tau}(1+\frac{1}{M})z^2} $
{\allowdisplaybreaks
\begin{eqnarray*}
& & \hspace{-10mm}
\times \,\ 
\bigg\{
\sum_{(j_1,j_2) \, \in \, \Omega^{\rm (I)}} e^{-\frac{2\pi i}{M}[
(k_1+\frac12)(j_1+j_2)+
(k_1+k_2+\frac12)(j_1+1)]} \,\ 
{\rm ch}^{(+){\rm tw}}_{\overset{N=4}{\ddot{H}}
\big(\Lambda^{(M)(1, 0){\rm (I)}}_{j_1,j_2}\big)}
\\[2mm]
& &
+ \sum_{(j_1,j_2) \, \in \, \Omega^{\rm (III)}} e^{-\frac{2\pi i}{M}[
(k_1+\frac12)j_1+
(k_1+k_2+\frac12)(j_1+j_2+1)]} \,\ 
{\rm ch}^{(+){\rm tw}}_{\overset{N=4}{\ddot{H}}
\big(\Lambda^{(M)(1, 0){\rm (III)}}_{j_1,j_2}\big)} \bigg\}
\end{eqnarray*}}
\item[{\rm (III)}] \,\ ${\rm ch}^{(-)}_{\overset{N=4}{\ddot{H}}
\big(\Lambda^{(M)(1, 0){\rm (III)}}_{k_1,k_2}\big)}
\Big(-\dfrac{1}{\tau}, \dfrac{z}{\tau} \Big)
\,\ = \,\ 
\dfrac{- 1}{M} \, e^{-\frac{2\pi i}{\tau}(1+\frac{1}{M})z^2} $
{\allowdisplaybreaks
\begin{eqnarray*}
& & \hspace{-10mm}
\times \,\ \bigg\{
\sum_{(j_1,j_2) \, \in \, \Omega^{\rm (I)}} e^{-\frac{2\pi i}{M}[
(k_1+k_2+\frac12)(j_1+j_2)+
(k_1+\frac12)(j_1+1)]} \,\ 
{\rm ch}^{(+){\rm tw}}_{\overset{N=4}{\ddot{H}}
\big(\Lambda^{(M)(1, 0){\rm (I)}}_{j_1,j_2}\big)}
\\[2mm]
& &
+ \sum_{(j_1,j_2) \, \in \, \Omega^{\rm (III)}} e^{-\frac{2\pi i}{M}[
(k_1+k_2+\frac12)j_1+
(k_1+\frac12)(j_1+j_2+1)]} \,\ 
{\rm ch}^{(+){\rm tw}}_{\overset{N=4}{\ddot{H}}
\big(\Lambda^{(M)(1, 0){\rm (III)}}_{j_1,j_2}\big)} \bigg\}
\end{eqnarray*}}
\end{enumerate}
\item[${\rm 1)}^{\rm tw}$]
\begin{enumerate}
\item[{\rm (I)}] \quad ${\rm ch}^{(+) {\rm tw}}_{\overset{N=4}{\ddot{H}}
\big(\Lambda^{(M)(1, 0){\rm (I)}}_{k_1,k_2}\big)}
\Big(-\dfrac{1}{\tau}, \dfrac{z}{\tau} \Big)
\,\ = \,\ 
\dfrac{- 1}{M} \, e^{-\frac{2\pi i}{\tau}(1+\frac{1}{M})z^2} $
{\allowdisplaybreaks
\begin{eqnarray*}
& & \hspace{-10mm}
\times \,\ \bigg\{
\sum_{(j_1,j_2) \, \in \, \Omega^{\rm (I)}} e^{-\frac{2\pi i}{M}[
(k_1+k_2)(j_1+\frac12)+
(k_1+1)(j_1+j_2+\frac12)]} \,\ 
{\rm ch}^{(-)}_{\overset{N=4}{\ddot{H}}
\big(\Lambda^{(M)(1, 0){\rm (I)}}_{j_1,j_2}\big)}
\\[2mm]
& &
+ \sum_{(j_1,j_2) \, \in \, \Omega^{\rm (III)}} e^{-\frac{2\pi i}{M}[
(k_1+k_2)(j_1+j_2+\frac12)+
(k_1+1)(j_1+\frac12)]} \,\ 
{\rm ch}^{(-)}_{\overset{N=4}{\ddot{H}}
\big(\Lambda^{(M)(1, 0){\rm (III)}}_{j_1,j_2}\big)} \bigg\}
\end{eqnarray*}}
\item[{\rm (III)}] \quad ${\rm ch}^{(+) {\rm tw}}_{\overset{N=4}{\ddot{H}}
\big(\Lambda^{(M)(1, 0){\rm (III)}}_{k_1,k_2}\big)}
\Big(-\dfrac{1}{\tau}, \dfrac{z}{\tau} \Big)
\,\ = \,\ 
\dfrac{- 1}{M} \, e^{-\frac{2\pi i}{\tau}(1+\frac{1}{M})z^2} $
{\allowdisplaybreaks
\begin{eqnarray*}
& & \hspace{-10mm}
\times \,\ \bigg\{
\sum_{(j_1,j_2) \, \in \, \Omega^{\rm (I)}} e^{-\frac{2\pi i}{M}[
k_1(j_1+\frac12)+
(k_1+k_2+1)(j_1+j_2+\frac12)]} \,\ 
{\rm ch}^{(-)}_{\overset{N=4}{\ddot{H}}
\big(\Lambda^{(M)(1, 0){\rm (I)}}_{j_1,j_2}\big)}
\\[2mm]
& &
+ \sum_{(j_1,j_2) \, \in \, \Omega^{\rm (III)}} e^{-\frac{2\pi i}{M}[
k_1(j_1+j_2+\frac12)+
(k_1+k_2+1)(j_1+\frac12)]} \,\ 
{\rm ch}^{(-)}_{\overset{N=4}{\ddot{H}}
\big(\Lambda^{(M)(1, 0){\rm (III)}}_{j_1,j_2}\big)} \bigg\}
\end{eqnarray*}}
\end{enumerate}
\item[${\rm 2)}^{\rm tw}$]
\begin{enumerate}
\item[{\rm (I)}] \quad ${\rm ch}^{(-) {\rm tw}}_{\overset{N=4}{\ddot{H}}
\big(\Lambda^{(M)(1, 0){\rm (I)}}_{k_1,k_2}\big)}
\Big(-\dfrac{1}{\tau}, \dfrac{z}{\tau} \Big)
\,\ = \,\ 
\dfrac{1}{M} \, e^{-\frac{2\pi i}{\tau}(1+\frac{1}{M})z^2} $
{\allowdisplaybreaks
\begin{eqnarray*}
& & \hspace{-10mm}
\times \,\ \bigg\{
\sum_{(j_1,j_2) \, \in \, \Omega^{\rm (I)}} e^{-\frac{2\pi i}{M}[
(k_1+k_2)(j_1+j_2)+
(k_1+1)(j_1+1)]} \,\ 
{\rm ch}^{(-){\rm tw}}_{\overset{N=4}{\ddot{H}}
\big(\Lambda^{(M)(1, 0){\rm (I)}}_{j_1,j_2}\big)}
\\[2mm]
& &
+ \sum_{(j_1,j_2) \, \in \, \Omega^{\rm (III)}} e^{-\frac{2\pi i}{M}[
(k_1+k_2)j_1+
(k_1+1)(j_1+j_2+1)]} \,\ 
{\rm ch}^{(-){\rm tw}}_{\overset{N=4}{\ddot{H}}
\big(\Lambda^{(M)(1, 0){\rm (III)}}_{j_1,j_2}\big)} \bigg\}
\end{eqnarray*}}
\item[{\rm (III)}] \quad ${\rm ch}^{(-) {\rm tw}}_{\overset{N=4}{\ddot{H}}
\big(\Lambda^{(M)(1, 0){\rm (III)}}_{k_1,k_2}\big)}
\Big(-\dfrac{1}{\tau}, \dfrac{z}{\tau} \Big)
\,\ = \,\ 
\dfrac{1}{M} \, e^{-\frac{2\pi i}{\tau}(1+\frac{1}{M})z^2} $
{\allowdisplaybreaks
\begin{eqnarray*}
& & \hspace{-10mm}
\times \,\ \bigg\{
\sum_{(j_1,j_2) \, \in \, \Omega^{\rm (I)}} e^{-\frac{2\pi i}{M}[
k_1(j_1+j_2)+
(k_1+k_2+1)(j_1+1)]} \,\ 
{\rm ch}^{(-){\rm tw}}_{\overset{N=4}{\ddot{H}}
\big(\Lambda^{(M)(1, 0){\rm (I)}}_{j_1,j_2}\big)}
\\[2mm]
& &
+ \sum_{(j_1,j_2) \, \in \, \Omega^{\rm (III)}} e^{-\frac{2\pi i}{M}[
k_1j_1+
(k_1+k_2+1)(j_1+j_2+1)]} \,\ 
{\rm ch}^{(-){\rm tw}}_{\overset{N=4}{\ddot{H}}
\big(\Lambda^{(M)(1, 0){\rm (III)}}_{j_1,j_2}\big)} \bigg\}
\end{eqnarray*}}
\end{enumerate}
\end{enumerate}
\end{enumerate}
\end{prop}

\medskip

\begin{prop} 
\label{n2n4:prop:2024-121c}
The $T$-transformation is given as follows:
\begin{enumerate}
\item[{\rm 1)}] \,\ ${\rm ch}^{(\pm)}_{\overset{N=4}{\ddot{H}}
\big(\Lambda^{(M)(1, 0)(\heartsuit)}_{k_1,k_2}\big)}(\tau+1, z)
\,\ = \,\ 
- \, i \, e^{\frac{2\pi i}{M}(k_1+\frac12)(k_1+k_2+\frac12)} \,\ 
{\rm ch}^{(\mp)}_{\overset{N=4}{\ddot{H}}
\big(\Lambda^{(M)(1, 0)(\heartsuit)}_{k_1,k_2}\big)}(\tau, z)$
\item[{\rm 2)}]
\begin{enumerate}
\item[{\rm (i)}] \,\ ${\rm ch}^{(\pm){\rm tw}}_{\overset{N=4}{\ddot{H}}
\big(\Lambda^{(M)(1, 0){\rm (I)}}_{k_1,k_2}\big)}(\tau+1, z)
\,\ = \,\ 
e^{\frac{2\pi i}{M}(k_1+1)(k_1+k_2)} \,\ 
{\rm ch}^{(\pm){\rm tw}}_{\overset{N=4}{\ddot{H}}
\big(\Lambda^{(M)(1, 0){\rm (I)}}_{k_1,k_2}\big)}(\tau, z)$
\item[{\rm (ii)}] \,\ ${\rm ch}^{(\pm){\rm tw}}_{\overset{N=4}{\ddot{H}}
\big(\Lambda^{(M)(1, 0){\rm (III)}}_{k_1,k_2}\big)}(\tau+1, z)
\,\ = \,\ 
e^{\frac{2\pi i}{M}k_1(k_1+k_2+1)} \,\ 
{\rm ch}^{(\pm){\rm tw}}_{\overset{N=4}{\ddot{H}}
\big(\Lambda^{(M)(1, 0){\rm (III)}}_{k_1,k_2}\big)}(\tau, z)$
\end{enumerate}
\end{enumerate}
\end{prop}

\medskip

\begin{rem}  
\label{n2n4:rem:2024-201a}
The formulas in Proposition \ref{n2n4:prop:2024-121a} 
can be written by Mumford's theta functions 
$\vartheta_{ab}$ and Dedekind's eta function $\eta(\tau)$ 
by using {\rm (2.3)} in {\rm \cite{W2023a}}. 
\end{rem}

\medskip

In particular in the case $(m,m_2)=(1,0)$ and $M=2$, namely
for non-irreducible N=4 module with central charge $= -9$,
the formulas in Corollary \ref{n2n4:cor:2024-110a} are simplified by 
\eqref{n2n4:eqn:2024-104g} to give the following:

\medskip

\begin{prop} \,\ 
\label{n2n4:prop:2024-121d}
\begin{enumerate}
\item[{\rm 1)}]
\begin{enumerate}
\item[{\rm (i)}] \,\ ${\rm ch}^{(+)}_{\overset{N=4}{\ddot{H}}
\big(\Lambda^{(2)(1, 0){\rm (I)}}_{0,0}\big)}(\tau,z)
\,\ = \,\ 
i \,\ 
\dfrac{\vartheta_{00}(\tau,z)}{\vartheta_{11}(\tau,2z)}$
$$
= \,\ - \, q^{-\frac18} \, e^{2\pi iz} \,\  
\dfrac{\prod\limits_{n=1}^{\infty}
(1+e^{2\pi iz}q^{n-\frac12})(1+e^{-2\pi iz}q^{n-\frac12})}
{\prod\limits_{n=1}^{\infty}(1-e^{4\pi iz}q^{n-1})(1-e^{-4\pi iz}q^{n})}
$$
\item[{\rm (ii)}] \,\ ${\rm ch}^{(-)}_{\overset{N=4}{\ddot{H}}
\big(\Lambda^{(2)(1, 0){\rm (I)}}_{0,0}\big)}(\tau,z)
\,\ = \,\ 
i \,\ 
\dfrac{\vartheta_{01}(\tau,z)}{\vartheta_{11}(\tau,2z)}$
$$
= \,\ - \, q^{-\frac18} \, e^{2\pi iz} \,\ 
\dfrac{\prod\limits_{n=1}^{\infty}
(1-e^{2\pi iz}q^{n-\frac12})(1-e^{-2\pi iz}q^{n-\frac12})}
{\prod\limits_{n=1}^{\infty}(1-e^{4\pi iz}q^{n-1})(1-e^{-4\pi iz}q^{n})}
$$
\end{enumerate}
\item[{\rm 2)}]
\begin{enumerate}
\item[{\rm (i)}] ${\rm ch}^{(+) \, {\rm tw}}_{\overset{N=4}{\ddot{H}}
\big(\Lambda^{(2)(1, 0){\rm (I)}}_{0,0}\big)}(\tau,z)
= \, 
i \, \dfrac{\vartheta_{10}(\tau,z)}{\vartheta_{11}(\tau,2z)}
= \, - \, e^{\pi iz} 
\dfrac{\prod\limits_{n=1}^{\infty}
(1+e^{2\pi iz}q^{n-1})(1+e^{-2\pi iz}q^{n})}
{\prod\limits_{n=1}^{\infty}(1-e^{4\pi iz}q^{n-1})(1-e^{-4\pi iz}q^{n})}$
\item[{\rm (ii)}] ${\rm ch}^{(-) \, {\rm tw}}_{\overset{N=4}{\ddot{H}}
\big(\Lambda^{(2)(1, 0){\rm (I)}}_{0,0}\big)}(\tau,z)
= \, 
\dfrac{\vartheta_{11}(\tau,z)}{\vartheta_{11}(\tau,2z)}
\, = \,\ e^{\pi iz} \, 
\dfrac{\prod\limits_{n=1}^{\infty}
(1-e^{2\pi iz}q^{n-1})(1-e^{-2\pi iz}q^{n})}
{\prod\limits_{n=1}^{\infty}(1-e^{4\pi iz}q^{n-1})(1-e^{-4\pi iz}q^{n})}$
\end{enumerate}
\end{enumerate}
\end{prop}

\medskip

The power seies expansion of these formulas is computed easily by 
using Note \ref{note:2024-204c} and obtained as follows:

\medskip

\begin{cor} \,\ 
\label{n2n4:cor:2024-201a}
\begin{enumerate}
\item[{\rm 1)}] ${\rm ch}^{(+)}_{\overset{N=4}{\ddot{H}}
\big(\Lambda^{(2)(1,0){\rm (I)}}_{0,0}\big)}(\tau,z)
= \, 
\dfrac{e^{-2\pi iz}}{\eta(\tau)^3} \, 
\sum\limits_{n \in \zzz} \, e^{2\pi inz} \, \bigg[
\sum\limits_{\substack{j, \, k \, \in \, \zzz \\[1mm]
j, \, k \, \geq 0}}
-
\sum\limits_{\substack{j, \, k \, \in \, \zzz \\[1mm]
j, \, k \, < 0}} \bigg] 
(-1)^j \,  q^{\frac12 j(j+1)+jk+\frac12 (n+2k)^2}$

\item[{\rm 2)}] ${\rm ch}^{(+) \, {\rm tw}}_{\overset{N=4}{\ddot{H}}
\big(\Lambda^{(2)(1,0){\rm (I)}}_{0,0}\big)}(\tau,z)
= \, 
\dfrac{e^{-\pi iz}}{\eta(\tau)^3} \, 
\sum\limits_{n \in \zzz} \, e^{2\pi inz} \, 
\bigg[
\sum\limits_{\substack{j, \, k \, \in \, \zzz \\[1mm]
j, \, k \, \geq 0}}
-
\sum\limits_{\substack{j, \, k \, \in \, \zzz \\[1mm]
j, \, k \, < 0}} \bigg] 
(-1)^j \, q^{\frac12 j(j+1)+jk+\frac12(n-2k-\frac12)^2} $
\end{enumerate}
\end{cor}


\begin{thebibliography}{99}

\bibitem{FKW} E. Frenkel, V. G. Kac and M. Wakimoto : 
Characters and fusion rules for W-algebras via quantized 
Drinfeld-Sokolov reduction, Commun. Math. Phys. 147 (1992), 
295-328.

\bibitem{K1} V. G. Kac : Infinite-Dimensional Lie Algebras, 3rd edition,
Cambridge University Press, 1990.

\bibitem{KP} V. G. Kac and D. Peterson : Infinite-dimensional
Lie algebras, theta functions and modular forms, 
Advances in Math. 53 (1984), 125-264.

\bibitem{KRW} V. G. Kac, S.-S. Roan and M. Wakimoto : 
Quantum reduction for affine superalgebras, Commun. Math. Phys. 
241 (2003), 307-342.


\bibitem{KW1988a} V. G. Kac and M. Wakimoto : 
Modular and conformal invariance constraints in 
representation theory of affine algebras, Advances in Math. 70 
(1988), 156-236.

\bibitem{KW1988b} V. G. Kac and M. Wakimoto : 
Modular invariant representations of infinite-dimensional Lie
algebras and superalgebras, Proc. Nat'l Acad. Sci. USA. 85 (1988), 
4956-4960.

\bibitem{KW1989} V. G. Kac and M. Wakimoto : 
Classification of modular invariant representations of 
affine algebras, in \lq \lq Infinite-dimensional 
Lie algebras and groups", Advanced series in Math. Phys. 7, 
World Scientific, 1989, 138-177.

\bibitem{KW1994} V. G. Kac and M. Wakimoto : 
Integrable highest weight modules over affine superalgebras 
and number theory, in \lq \lq Lie Theory and Geometry $\sim$ 
in honor of Bertram Kostant" Progress in Math. Phys. Vol.123, 
Birkh\"{a}user, 1994, 415-456.

\bibitem{KW2001} V. G. Kac and M. Wakimoto : 
Integrable highest weight modules over affine superalgebras and 
Appell's function, Commun. Math. Phys. 215 (2001), 631-682.

\bibitem{KW2004} V. G. Kac and M. Wakimoto : 
Quantum reduction and representation theory of superconformal algebras,
Advances in Math. 185 (2004), 400-458.

\bibitem{KW2005} V. G. Kac and M. Wakimoto : 
Quantum reduction in the twisted case, Progress in Math. 237 
Birkh\"{a}user  (2005), 85-126. math-ph/0404049. 

\bibitem{KW2014} V. G. Kac and M. Wakimoto : 
Representations of affine superalgebras and mock theta functions, 
Transformation Groups 19 (2014), 387-455.
arXiv:1308.1261.

\bibitem{KW2016a} V. G. Kac and M. Wakimoto :
Representations of affine superalgebras and mock theta functions II, 
Advances in Math. 300 (2016), 17-70. 
arXiv:1402.0727.

\bibitem{KW2016b} V. G. Kac and M. Wakimoto :
Representations of affine superalgebras and mock theta functions III, 
Izv. Math. 80 (2016), 693-750. 
arXiv:1505.01047.

\bibitem{KW2017a} V. G. Kac and M. Wakimoto : A characterization of modified 
mock theta functions, Transformation Groups 22 (2017), 
arXiv:1510.05683.

\bibitem{KW2017b} V. G. Kac and M. Wakimoto : Representation of superconformal 
algebras and mock theta functions, Trudy Moskow Math. Soc. 78 (2017), 64-88.
arXiv:1701.03344.

\bibitem{Mum} D. Mumford : Tata Lectures on Theta I, Progress in Math. 28, 
Birkh\"{a}user Boston, 1983.

\bibitem{RY} F. Ravanini ans S.-K. Yang : Modular invariance in N=2 
superconformal field theories, Phys. Lett. 195B (1987), 202-208.

\bibitem{W1998} M. Wakimoto : Fusion rules for N=2 superconformal modules, 
arXiv:hep-th/9807144.

\bibitem{W2001b} M. Wakimoto : Lectures on Infinite-Dimensional Lie Algebra, 
World Scientific, 2001.

\bibitem{W2004} M. Wakimoto : Launching on a voyage into representation theory
of Lie superalgebras, Sugaku Expositions 17 (2004), 103-124.

\bibitem{W2022a} M. Wakimoto : Mock theta functions and characters of 
N=3 superconformal modules, arXiv:2202.03098.



\bibitem{W2022d} M. Wakimoto : Mock theta functions and characters of 
N=3 superconformal modules III, arXiv:2207.04644.

\bibitem{W2022e} M. Wakimoto : Mock theta functions and characters of 
N=3 superconformal modules IV, arXiv:2209.00234.



\bibitem{W2023a} M. Wakimoto : On the characters of a certain 
series of N=4 superconformal modules, arXiv:2301.04028.


\bibitem{Z} S. Zwegers : Mock theta functions, PhD Thesis, Universiteit 
Utrecht, 2002, arXiv:0807.483.

\end{thebibliography}
\end{document}